\documentclass[a4paper, 11pt]{article}
\usepackage[margin=1in]{geometry}
\usepackage[colorlinks,
linkcolor=blue,
urlcolor=magenta,
citecolor=red]{hyperref}
\usepackage{amsthm}
\usepackage{amsmath}
\usepackage{amssymb}
\usepackage{xcolor}
\usepackage{upgreek}
\usepackage{mathrsfs}
\usepackage{marginnote}
\usepackage{cleveref}
\usepackage{ulem}

\usepackage{mathtools}
\usepackage{cancel}

\newcommand{\loc}{\ensuremath{\text{loc}}}

\newcommand{\mb}[1]{\ensuremath{\mathbb{#1}}}
\newcommand{\N}{\mb{N}}

\newcommand{\R}{\mb{R}}

\newcommand{\sgn}{\mathop{\mathrm{sgn}}}

\newfont{\bl}{msbm10 scaled \magstep2}

\newcommand{\beq}{\begin{equation}}
\newcommand{\eeq}{\end{equation}}

\newcommand{\notmid}{\mid\kern-0.5em\not\kern0.5em}

\newcommand{\al}{\alpha}

\newcommand{\de}{\delta}
\newcommand{\eps}{\varepsilon}

\newcommand{\supp}{\mathop{\mathrm{supp}}}

\renewcommand{\Im}{\ensuremath{\mathop{\mathrm{Im}}}}

\newtheorem{thm}{Theorem}[section]
\newtheorem{lem}[thm]{Lemma}
\newtheorem{prop}[thm]{Proposition}
\newtheorem{cor}[thm]{Corollary}
\theoremstyle{definition}
\newtheorem{defi}[thm]{Definition}
\newtheorem{conjecture}[thm]{Conjecture}
\newtheorem{rem}[thm]{Remark}

\newtheorem{notation}[thm]{Notation}
\newtheorem{que}[thm]{Question}

\newcommand{\vol}{\mathrm{vol}}

\newcommand{\LLS}{Lorentzian length space }
\newcommand{\LLSn}{Lorentzian length space}

\newcommand{\LpLSn}{Lorentzian pre-length space}

\newcommand{\dint}{\mathrm{d}}

\newcommand{\I}{{}^-\!I}

\newcommand{\id}{\mathrm{id}}

\renewcommand{\L}{\mathcal{L}}

\newcommand{\Ric}{\mathrm{Ric}}
\newcommand{\m}{\mathfrak{m}}
\newcommand{\n}{\mathfrak{n}}

\newcommand{\tmcp}{\textnormal{\textsf{TMCP}}}
\newcommand{\mcp}{\textnormal{\textsf{MCP}}}
\newcommand{\tCD}{\textnormal{\textsf{TCD}}}
\newcommand{\CD}{\textnormal{\textsf{CD}}}
\newcommand{\RCD}{\textnormal{\textsf{RCD}}}

\newcommand{\met}{\textnormal{\textsf{d}}}
\renewcommand{\loc}{\mathrm{loc}}
\DeclareMathOperator{\spt}{spt}

\renewcommand{\labelenumi}{(\roman{enumi})}
\renewcommand\theenumi\labelenumi

\newtheorem{ithm}{Theorem}[section]

\newcommand{\ac}{\mathrm{ac}}
\newcommand{\opt}{\mathrm{opt}}
\newcommand{\e}{\textnormal{\textsf{e}}}

\DeclareMathOperator{\Geo}{Geo}
\newcommand{\TGeo}{\mathrm{TGeo}}
\newcommand{\OptGeo}{\mathrm{OptGeo}}
\newcommand{\OptTGeo}{\mathrm{OptTGeo}}
\newcommand{\wTCD}{\textnormal{\textsf{wTCD}}}
\newcommand{\bdpi}{\boldsymbol{\pi}}

\newcommand{\Ent}{\mathrm{Ent}}
\newcommand{\T}{\mathcal T}

\allowdisplaybreaks
\title{Generalized cones admitting a curvature-dimension condition}
\author{Matteo Calisti\thanks{{\tt matteo.calisti@univie.ac.at}, Faculty of Mathematics, University of Vienna,
Austria.}, Christian Ketterer\thanks{{\tt christian.ketterer@mu.ie}, Department of Mathematics and Statistics, Maynooth University}, Clemens S\"{a}mann\thanks{{ \tt{clemens.saemann@univie.ac.at}}, Faculty of Mathematics, University of Vienna,
Austria.}}

\begin{document}
\maketitle
\begin{abstract}
    We study (generalized) cones over metric spaces, both in Riemannian and Lorentzian signature. 
    In particular, we establish synthetic lower Ricci curvature bounds \`a la Lott--Villani--Sturm \cite{LV:09, Stu:06a, Stu:06b} and Ohta \cite{Oht:07} in the metric measure case, and \`a la Cavalletti--Mondino \cite{CM:24a} in Lorentzian signature. Here, a generalized cone is a warped product of a one-dimensional base space, which will be positive or negative definite, over a fiber that is a metric space. We prove that Riemannian or Lorentzian generalized cones over $\mathsf{CD}$-spaces satisfy the \emph{(timelike) measure contraction property} $\mathsf{(T)MCP}$ --- a weaker version of a \emph{(timelike) curvature-dimension condition} $\mathsf{(T)CD}$. Conversely, if the generalized cone is a $\mathsf{(T)CD}$-space, then the fiber is a $\mathsf{CD}$-space with the appropriate bounds on Ricci curvature and dimension.  In proving these results we develop a novel and powerful two-dimensional localization technique, which we expect to be interesting in its own right and useful in other circumstances. We conclude by giving several applications including synthetic singularity and splitting theorems for generalized cones. The final application is that we propose a new definition for lower curvature bounds for metric and metric measure spaces via lower curvature bounds for generalized cones over the given space.
\end{abstract}
\medskip


\tableofcontents
\section{Introduction}
Warped products are ubiquitous in geometry and applications. In particular, they appear as rigidity cases for (semi-)Riemannian manifolds and Ricci limit spaces, e.g.\ \cite{CC:96, coldingshape, wylie:warped}, and as major examples in general relativity as well as in cosmology e.g.\ the Schwarzschild solution to Einsteins equation of a single isolated mass \cite[Ch.\ 13]{ONe:83} or FLRW-spacetimes as models of our universe \cite[Ch.\ 12]{ONe:83}. In the smooth setting a warped product between two smooth semi-Riemannian manifolds $(B,g_B)$ and $(F,g_F)$ is the product manifold $B\times F$ equipped with the semi-Riemannian metric 
$$g_{_{B\times_f F}} = (P_B)^* g_B + (f\circ P_B)^2 (P_F)^* g_F\,,$$
where $f\colon B\rightarrow (0,\infty)$ is smooth, $(P_B)^*$ and $(P_F)^*$ are the pullbacks onto $B\times F$ with respect to the projections maps $P_B$ and $P_F$ from $B\times F$ onto $B$ and $F$, respectively. The spaces $B$ and $F$, and the function $f$ are respectively called the base, the fiber and the warping function and the semi-Riemannian manifold $(B\times F, g_{B\times_f F})$ has index (i.e., the dimension of the largest negative definite subspace) $\nu_B + \nu_F$, where $\nu_B$ and $\nu_F$ are the index of $g_B$ and $g_F$, respectively. 
Cones over metric spaces and warped products of metric spaces \cite{AB:98, chenwp, BBI:01} play a similar role in metric geometry \cite{AB:03, AB:04}. In the Lorentzian setting such warped products were the first major examples of \LLSn s \cite{AGKS:23}, which were not \smash{Lorentz(--Finsler)} spacetimes. \LLSn s \cite{KS:18} are the analogues to metric length spaces in Lorentzian signature.
\smallskip

This synthetic approach to Lorentzian geometry turned into a highly active field in recent years --- not only as a way to study low regularity phenomena and singularities in general relativity but also from a purely geometric point-of-view. As an example for the former, in \cite{GKS:19} it was shown for the first time that inextendibility implies a blow-up of curvature and in \cite{AGKS:23} the first synthetic singularity theorem was proven and its relation to big bang and big crunch singularities investigated. The synthetic timelike sectional curvature bounds introduced in \cite{KS:18} were built on triangle comparison, which was pioneered in the smooth setting by Harris \cite{Har:82}, Andersson--Howard \cite{AH:98} and Alexander--Bishop \cite{AB:08}. Building on a characterization of timelike Ricci curvature bounds for smooth spacetimes, independently by McCann \cite{McC:20} and Mondino--Suhr \cite{MS:23}, Cavalletti and Mondino introduced synthetic timelike Ricci curvature bounds in the setting of \LLSn s, established geometric consequences like a timelike Bishop--Gromov inequality or a timelike Bonnet--Myers inequality, and proved a synthetic Hawking singularity theorem \cite{CM:24a}. This led to another surge of activity in characterizing different notions of timelike Ricci curvature bounds \cite{Bra:23a, Bra:23b}, studying the null energy condition \cite{McC:24, Ket:24, CMM:24} and isoperimetric inequalities \cite{CM:24b}. Also, a first-order calculus for causally monotone functions (sometimes called time functions) was developed in \cite{BBCGMORS:24} and a $p$-D'Alembertian comparison result for spaces with lower timelike Ricci curvature bounds was proven (see also \cite{Bra:24} for an explicit representation of the $p$-D'Alembertian). This allows one to give a new and simplified proof of the splitting theorems of Lorentzian geometry \cite{BGMcCOS:24} and their extension to low regularity and with weighted volumes \cite{BGMcCOS:25}. In a more geometric direction gluing of \LLSn s has been studied in \cite{BR:24, Rot:23} --- another advantage of leaving the smooth framework --- and in a similar spirit causal completions now fit into the same nonsmooth framework \cite{ABS:22, BFH:23}. Yet another intense area of activity is defining convergence for Lorentzian spaces --- recent approaches include \cite{Mue:22, MS:24, BMS:24, SS:24, MS:25}. For recent reviews in the area of synthetic Lorentzian geometry see \cite{CM:22, Sae:24, McC:25, Bra:25}, and particularly relevant to our study of generalized cones as warped products with one-dimensional base see \cite{Ket:25}. 
\smallskip

The basic question we address in this article is: How do synthetic Ricci curvature bounds of the fiber relate to synthetic Ricci curvature bounds of the whole space and to convexity/concavity properties of the warping function? Note that the fiber will always be a metric space (i.e.\ positive definite) while for the entire generalized cone we will consider both the Riemannian and Lorentzian case (by considering both signs of the base space). We conclude the introduction by outlining the main results and the structure of the article.

\subsection*{Main results and structure of the article}
In section \ref{sec-prelim} we fix notations, recall basics from metric measure geometry, \LLSn s and in particular, recall the different notions of Ricci curvature bounds in Subsection \ref{subsec-curv}. Then, in Subsection \ref{subsec-loc} we recall the localization method on metric measure spaces. Next, in Section \ref{sec-gen-con-wpd-prod} we discuss semi-Riemannian warped products in the smooth and nonsmooth setting in detail. A particular focus will be on the case of a one-dimensional base, and also with one-dimensional fiber. In Subsection \ref{subsec-con-ms} we view metric generalized cones as metric length spaces and recall their properties. To be precise, given a metric space $(X,\met)$ as fiber, an interval $I\subseteq \R$ and a continuous warping function $f\colon I\rightarrow [0,\infty)$, we put a length structure on the metric generalized cone $I\times_f X$ via
$$L(\gamma)= \int_a^b \sqrt{ \dot \alpha^2 + (f\circ \alpha)^2 v_\beta^2}\,,$$
where $\gamma=(\alpha, \beta): [a,b]\rightarrow I \times X$ is an absolutely continuous  curve and has components $\alpha$ and $\beta$ (that are absolutely continuous) and $v_\beta$ is the metric derivative of $\beta$ (cf.\ e.g.\ \cite[Subsec.\ 2.7]{BBI:01}). Then, using the induced length distance (and taking a distance zero quotient) turns $I\times_f X$ into a metric length space. Similarly, in Subsection \ref{subsec-con-lls} we view Lorentzian generalized cones as \LLSn s and discuss their properties. Here the length structure on the Lorentzian generalized cone ${}^-I \times_f X$ is given by
\begin{equation*}
 L(\gamma):= \int_a^b \sqrt{\dot\alpha^2 - (f\circ\alpha)^2 v_\beta^2}\,,
\end{equation*}
for $\gamma$ a \emph{causal} curve, i.e., $-\dot\alpha^2 + (f\circ \alpha)^2 v_{\beta}^2\leq 0$ and $\dot\alpha>0$ a.e., in which case $\gamma$ is \emph{future directed}, or $\dot\alpha<0$ a.e., in which case it is called \emph{past directed}. Then, the \emph{time separation function} $\uptau\colon ({}^-I \times_f X)^2
\rightarrow [0,\infty]$ is defined (cf.\ e.g.\ \cite[Def.\ 14.15]{ONe:83} in the spacetime setting) as
\begin{equation*}
 \uptau(y,y'):=\sup\{L(\gamma):\gamma \text{ future directed causal curve from } y \text{ to } y'\}\cup\{0\}\,.
\end{equation*}
This turns ${}^-I \times_f X$ into a \LLS for fibers $X$ that are locally compact length spaces. We write ${}^-I$ to indicate that $I$ is equipped with the semi-Riemannian metric $-(\met t)^2$.
\smallskip

Then, the main part of the article starts with Section \ref{sec-2d-mod-spa} where we discuss the two-dimensional warped product model spaces. As a basis for later results we establish Ricci curvature bounds for two-dimensional warped products from convexity or concavity of the warping function. To be precise, for $f\in C^\infty(I)$ with
\begin{align*}
f'' \pm \kappa f\leq 0 \ \mbox{ and } \ \pm (f')^2 + \kappa f^2 \leq \eta\,.
\end{align*}
for some $\kappa\in\R$, we have
\setcounter{section}{4}
\setcounter{ithm}{1}
\begin{ithm}
Let $N>1$ and assume $\partial I= f^{-1}(\{0\})$.
If $h\in C^0([a,b], [0, \infty))$ satisfies $$\smash{\left(h^\frac{1}{N-1}\right)'' + \eta h^\frac{1}{N-1}\leq 0}$$ in the distributional sense, then the generalized cone $({}^\pm I\times_f [a,b], {f^{N-1}h}\,\vol_{^\pm I\times_f [a,b]})$ satisfies $\CD(\kappa N, N+1)$ or $\tCD_p(-\kappa N,N+1)$, respectively. 

Here $\vol_{^\pm I\times_f [a,b]}$ is the volume measure $\sqrt{|\det{(\pm \mathrm{d} t^2 + f^2 \mathrm{d} x^2)|}}\,\L^2 = f\,\L^2$, where $\L^2$ is the two-dimensional Lebesgue measure on $I\times [a,b]$.
\end{ithm}
Analogously, in Theorem \ref{th:lasttheorem} we establish that if the two-dimensional warped product has (timelike) Ricci curvature bounds then the fiber has Ricci curvature bounded below and the warping function is $\kappa$-concave (for some appropriate constant $\kappa$). In more detail, we have
\setcounter{ithm}{9}
\begin{ithm}
 Let $\kappa\in \R$ and $N>1$.
     Let  $([a,b], \m)$ be a metric measure space for a Radon measure $\m$ with $\spt \m=~[a,b]$ and let $f:{}^\pm\! I \rightarrow [0, \infty)$ be smooth with $\partial I= f^{-1}(\{0\})$. Assume 
     $\left(^\pm \!I\times_f[a,b], f(t)^N \dint t\otimes \m\right)$
     satisfies \begin{enumerate}\item[(+)] $\CD(\kappa N, N+1)$ \item[(--)] $\tCD_p(-\kappa N, N+1)$, $p\in (0,1).$
     \end{enumerate}  If we consider $I={^+}\!I$, then we assume additionally $\partial I\neq \emptyset$. Then it  follows that 
\begin{enumerate}
    \item $f''\pm \kappa f\leq 0$, 
    \item \begin{enumerate}
        \item[(+)]
    $([a,b], h(r)\,\dint r)$ satisfies $\CD((N-1)\eta, N)$ with $\eta= \sup_I\left\{ (f')^2 + \kappa f^2\right\}$.
    \item[(--)] 
    $([a,b], h(r)\,\dint r)$ satisfies $\CD((N-1)\eta, N)$ with $\eta= \sup_I\left\{ -(f')^2 + f'' f\right\}$.
    \end{enumerate}
\end{enumerate}
    In particular, $\m= h(r)\,\dint r$ where  $h$ is semi-concave on $[a,b]$ and continuous on $(a,b)$.
\end{ithm}

Next, the whole section \ref{sec-fib-Y} is devoted to extending Theorem \ref{Th:2 dimensional non-smooth case} to the general case, i.e., the fiber being a general (essentially nonbranching) $\mathsf{CD}$-space and the warping function satisfying the appropriate convexity/concavity conditions. Then, the generalized cone satisfies the (timelike) measure contraction property. That is, in more detail, we succeed in proving the major result

\setcounter{section}{5}
\setcounter{ithm}{1}
\begin{ithm}[From $X$ to ${}^-\!I\times^N_fX$]
Assume that $(X,\met,\m)$ is an essentially nonbranching metric measure space satisfying $\CD(\eta (N-1), N)$ for $N>1$, and $f: I \rightarrow [0, \infty)$ is smooth with $f''- \kappa f\leq 0$, $-(f')^2 + \kappa f^2 \leq \eta$. Then the generalized cone ${}^-\!I\times_f X$ equipped with the measure $f(t)^N \,\dint t \otimes\dint\m$, satisfies $\tmcp(-\kappa N, N+1)$. 
\end{ithm}
\setcounter{subsection}{5}
\setcounter{ithm}{5}
As the argument carries over verbatim, we get the corresponding result for $I\times_f^N X$.
\begin{ithm}[From $X$ to ${}^+I\times^N_fX$] Assume that $(X,\met,\m)$ is an essentially nonbranching metric measure space satisfying $\CD(\eta (N-1), N)$ for $N>1$. Assume $f:I\rightarrow[0,\infty)$ is smooth and satisfies
 $f''+ \kappa f\leq 0$ and  $(f')^2 + \kappa f^2 \leq \eta$. 
Then $I\times_f^N X$, i.e., the generalized cone $I\times_f X$ equipped with the measure $f(t)^N \,\dint t \otimes\dint\m$, satisfies $\mcp(\kappa N, N+1)$. 
\end{ithm}

\setcounter{section}{1}
\begin{rem}
We conjecture that the property $\mcp(\kappa N, N+1)$ can be improved to the condition $\CD(\kappa N, N+1)$ (Conjecture \ref{conj:wp}). Indeed, under the additional assumption that  the space $(X, \met, \m)$ is also infinitesimal Hilbertian, i.e., it satisfies the condition $\RCD(\eta(N-1), N)$,  the generalized cone satisfies the condition $\RCD(\kappa N, N+1)$. This was proved by the second author in \cite{ketterer:25}. The methods and ideas in \cite{ketterer:25} are almost completely separate from the current paper and heavily rely on nonsmooth calculus tools developed solely for $\RCD$ spaces.
\end{rem}
Section \ref{sec-Y-fib} is devoted to proving a converse statement of Theorem \ref{th:totmcp}, namely

\setcounter{section}{6}
\setcounter{ithm}{0}
\begin{ithm}[From $\I\times^N_fX$ to $X$]
Let $N\in (1,\infty)$, $\kappa\in \R$, and let $X$ be a proper, essentially nonbranching, complete, geodesic metric space with a Radon measure $\m$. Assume that $\smash{Y:=\I\times^N_f X}$ satisfies $\tCD_p(-\kappa N, N+1)$ and is {timelike} {$p$-essentially nonbranching}
where $f: I\rightarrow [0, \infty)$ is smooth with $\partial I=f^{-1}(\{0\})$. Then \begin{itemize}
    \item $f''- \kappa f\leq 0$, and
    \item $X$ satisfies $\CD(\eta(N-1), N)$ where $\eta:=\sup_I\{ -(f')^2 + \kappa f^2\}$.
\end{itemize}
\end{ithm}\setcounter{section}{1}

\begin{rem}\label{rem:metricanalog}
Under slightly stronger assumptions for $f$, i.e.\ $f^{-1}(\{0\})= \partial I$, also a metric analogue of this statement holds, and again the proof is verbatim the same. 
\end{rem}
\begin{rem}
Generalized cones in the sense of this paper are a special case of a class of generalized products studied by Soultanis \cite{soultanis}.  In particular, in combination with (iv) of Remark \ref{Re:properties of CD(K,N)-densities and logarithmic convolutions} the first bulled point of \Cref{Th: Y to X} also follows from \cite[Cor.\ 6.3]{soultanis}.
\end{rem}

As a direct corollary we can upgrade the $\mathsf{CD}$-condition to the $\mathsf{RCD}$-condition if the fiber is infinitesimally Hilbertian, cf.\ Corollary \ref{cor-Y-TCD-X-RCD}. This is also proved in \cite{ketterer:25} by a different method.
In proving these results we develop a novel and powerful two-dimensional localization method that we expect to be useful in its own right.

Finally, we conclude the article by providing lists of standard cone constructions in Subsection \ref{subsec-new-ex} and give four main applications. To be precise we prove a synthetic Hawking singularity theorem \ref{cor-hawking}, a volume singularity theorem \ref{cor-vol-sing}, and a splitting theorem for generalized cones \ref{cor-spl-thm}. As the final application, we propose a new definition for lower curvature bounds for metric and metric measure spaces via appropriate lower curvature bounds for the generalized cones over the given space \ref{def-new-cb}.

\setcounter{section}{1}

\section{Preliminaries}\label{sec-prelim}
In this section we fix notations, recall basics from metric measure geometry, \LLSn s and in particular, recall the different notions of Ricci curvature bounds in Subsection \ref{subsec-curv}. Then in Subsection \ref{subsec-loc} we recall the localization method on metric measure spaces.

\subsection{Optimal transport  and curvature-dimension conditions}\label{subsec-curv}
We start by recalling some basics of optimal transportation theory, both on metric measure spaces and on Lorentzian pre-length spaces, referring to \cite{Vil:09, Stu:06a, Stu:06b} for the metric case and to \cite{CM:24a, Bra:23b} for the Lorentzian case for more details.

\subsubsection{Curvature-dimension conditions for metric measure spaces}

Let $(X,\met)$ be a Polish space and let $\m$ be a Radon measure on $X$.  We assume that $\met$ is a proper metric, i.e., closed and bounded subsets are compact. The triple $(X, \met, \m)$ will be called a metric measure space.  

We denote with $L_\met$ the canonical length structure associated to $\met$, and with $\Geo(X)$ the set of constant speed geodesics $\gamma:[0,1]\rightarrow X$, i.e., $\gamma\in \Geo(X)$ if and only if $$\met(\gamma(s), \gamma(t))= (t-s)\, \met(\gamma(0), \gamma(1))$$ for $0\leq s\leq t\leq 1$. We call $(X,\met)$ a geodesic (metric) space if for all $x, y\in X$ there exists a geodesic $\gamma\in \Geo(X)$ such that $\gamma(0)=x$ and $\gamma(1)=y$. A geodesic is also a minimizer of  $L_\met$ subject to its  endpoints. We note that this notion of geodesic differs from the  one in the context of smooth Riemannian manifolds where geodesics are solutions of the geodesic equation $\nabla_{\dot\gamma} \dot\gamma=0$.

We denote by $\mathscr{P}(X)$ the set of Borel probability measures on $X$ and by $\mathscr{P}_\mathrm{c}(X)$ the subset of compactly supported ones. Given a metric measure space $(X,\met,\m)$, we also define the subset $\mathscr{P}^\ac(X)$  of $\m$-absolutely continuous probability measures and its subset $\mathscr{P}^\ac_\mathrm{c}(X)$  of measures with compact support. For $p\geq1$, we  define the set $\mathscr{P}_p(X)$ of probability measures $\mu\in \mathscr{P}(X)$ with $\int_X\met(x,x_0)^p\,\dint\mu(x)<\infty$ for some $x_0\in X$. Similarly we define $\mathscr P^\ac_p(X)$  and $\mathscr P^\ac_{p,\textnormal c}(X)$.

Given $\mu,\nu\in\mathscr{P}(X)$, let $\Pi(\mu,\nu)$ be the set of all their \emph{couplings}, i.e., measures $\pi\in\mathscr{P}(X^2)$ such that $\pi(\cdot\times X)=\mu$ and $\pi(X\times\cdot)=\nu$. Alternatively, given a Borel map $T:X\rightarrow Y$ from a polish space $X$ to a polish space $Y$, define the \emph{pushforward} $T_\sharp \mu$ of $\mu\in\mathscr{P}(X)$ by $T$ as a probability measure on $Y$ by the formula
$$T_\sharp \mu(E):=\mu(T^{-1}(E)), \ \ E\subseteq X \mbox{ Borel measurable}.$$
Then $\pi\in\Pi(\mu,\nu)$ if and only if $(P_1)_\sharp \pi=\mu$ and $(P_2)_\sharp \pi=\nu$, where $P_1$ and $P_2$ are the projections onto the first and second factors. 

For $\mu,\nu\in\mathscr{P}_p(X)$ the \emph{$p$-Wasserstein distance} $W_p:\mathscr{P}_p(X)\times\mathscr{P}_p(X)\rightarrow[0,\infty)$ is defined as
\begin{align}\label{id:pwasserstein}W_p(\mu,\nu):=\left(\inf_{\pi\in\Pi(\mu,\nu)}\int_{X^2}\met(x,y)^p\,\dint\pi(x,y)\right)^\frac{1}{p}.\end{align}
If $\pi\in\Pi(\mu,\nu)$ is a minimizer in  \eqref{id:pwasserstein}, we say that $\pi$ is $p$-optimal and write $\pi\in\Pi^\opt_p(\mu,\nu)$. The function $W_p$ is a (finite) distance on $\mathscr{P}_p(X)$ and turns the latter into a complete metric space \cite[Thm.\ 6.18]{Vil:09}. Moreover, if $(X,\met)$ is geodesic, then also $(\mathscr{P}_p(X),W_p)$ is geodesic \cite[Cor.\ 7.22]{Vil:09}. 

When $W_p(\mu,\nu)$ is finite it is known that $\pi\in\Pi(\mu,\nu)$ is optimal if and only if it is concentrated on a $\met^p$-cyclically monotone set: Let $c\colon X^2\rightarrow \R$, then  a set $\Lambda\subseteq X^2$ is \emph{$c$-cyclically monotone} if for every finite set of points $(x_i,y_i)_{i=1,\dots,N}$ in $\Lambda$ it holds that
\begin{align*}
    \sum_{i=1}^N c(x_i,y_i)\leq\sum_{i=1}^Nc(x_i,y_{i+1})\,,
\end{align*}
with the convention that $y_{N+1}:=y_1$.

The evaluation map $\e_t:\Geo(X)\rightarrow X$ for a fixed $t\in[0,1]$ is defined as
\begin{align*}
    \e_t(\gamma):=\gamma_t.
\end{align*}
A measure $\bdpi\in\mathscr{P}(\Geo(X))$ is called a \emph{$p$-optimal geodesic plan} if $(\e_0,\e_1)_\sharp \bdpi$ is a $p$-optimal coupling between $(\e_0)_\sharp \bdpi$ and $(\e_1)_\sharp \bdpi$. It follows that  the map $[0,1]\ni t\mapsto(\e_t)_\sharp \bdpi$ is a geodesic in $(\mathscr{P}_p(X),W_p)$. Vice-versa, it is known that any geodesic $(\mu_t)_{t\in[0,1]}$ can be lifted to an $p$-optimal geodesic plan $\bdpi$ so that $(\e_t)_\sharp \bdpi=\mu_t$, see for example \cite[Thm.\ 2.10]{AG:13}. We denote by $\OptGeo_p(\mu,\nu)$ the space of such optimal geodesic plans $\bdpi$ with $(\e_0)_\sharp \bdpi=\mu$ and $(\e_1)_\sharp \bdpi=\nu$.

\begin{defi}[(Essential) nonbranching]\label{De: ess non-branch}
   A subset $G\subseteq\Geo(X)$ of geodesics is called \emph{nonbranching} if for any $\gamma^1,\gamma^2\in G$ the following holds:
   \begin{align*}
       \exists t\in(0,1)\colon\gamma^1_s=\gamma^2_s\quad\forall s\in[0,t]\quad\Longrightarrow\quad\gamma_s^1=\gamma^2_s\quad\forall s\in[0,1].
   \end{align*}
   The metric measure space $(X,\met,\m)$ is said to be nonbranching if $\Geo(X)$ is nonbranching. We say that $(X,\met,\m)$ is \emph{$p$-essentially nonbranching} if for all $\mu,\nu\in\mathscr{P}_p^\ac(X)$ any $\bdpi\in\OptGeo_p(\mu,\nu)$ is concentrated on a Borel nonbranching set $G\subseteq\Geo(X)$. We call $2$-essentially nonbranching just essentially nonbranching. 
\end{defi}

We define the \emph{(Boltzmann-Shannon) relative entropy functional} ${\Ent_\m(\mu):\mathscr{P}(X)\rightarrow\R\cup\{+\infty\}}$ as
\begin{align}\label{Eq: boltzmann}
   \Ent_\m(\mu):=\begin{cases}\int_X\rho\log\rho\,\dint\m\quad&\mathrm{if}\quad\mu=\rho\m\quad\mbox{and}\quad(\rho\log\rho)^+\in L^1(X,\m),\\
   +\infty&\mbox{otherwise.}\end{cases}
\end{align}
The two following facts are well known \cite{Stu:06a}:
\begin{itemize}
   \item by Jensen's inequality, it holds that $\Ent_\m(\mu)\geq-\log\m(\supp\,\mu)>-\infty$ for every ${\mu\in\mathscr{P}_\mathrm{c}^\ac(X)}$;
   \item $\Ent_\m$ is weakly lower semicontinuous: if $\{\mu_n\}_{n\in\N}\subseteq\mathscr{P}(X)$ converges weakly  to $\mu\in\mathscr{P}(X)$ and there is a Borel subset $C\subseteq X$ with $\m(C)<\infty$ and $\supp\,\mu_n\subseteq C$ for all $n\in\N$, then
\begin{align*}
   \Ent_\m(\mu)\leq\liminf_{n\rightarrow\infty}\Ent_\m(\mu_n).
\end{align*}
\end{itemize} 

\begin{defi}[$\CD(K,\infty)$]\label{De: CD for inf K}
   Given $K\in\R$, we say that $(X,\met,\m)$ satisfies the  $\CD(K,\infty)$ condition if and only if for every $\mu_0,\mu_1\in\mathscr{P}^\ac_{2,\mathrm{c}}(X)$ such that $\Ent_\m(\mu_0), \Ent_\m(\mu_1)<\infty$ there exists $\bdpi\in\OptGeo_2(\mu_0,\mu_1)$ such that, denoting $\mu_t:=(\e_t)_\sharp \bdpi$, we have
   \begin{align}\label{Eq: CD for inf K}
       \Ent_\m(\mu_t)\leq(1-t)\Ent_\m(\mu_0)+t\Ent_\m(\mu_1)-\frac{K}{2}t(1-t) W_2(\mu_0, \mu_1)^2\,.
   \end{align}
\end{defi}

For $\kappa\in\R$ and $\theta\geq0$ let the \emph{generalized sine functions} be
\begin{equation*}
    \mathfrak{s}_\kappa(\theta):=\begin{cases}
        \frac{1}{\sqrt{\kappa}}\sin(\sqrt{\kappa}\theta)\quad&\mathrm{if}\,\kappa>0,\\
        0&\mathrm{if}\,\kappa=0,\\
        \frac{1}{\sqrt{-\kappa}}\sinh(\sqrt{-\kappa}\theta)&\mathrm{if}\,\kappa<0.
    \end{cases}
\end{equation*}
Then, for $t\in[0,1]$, we define the \emph{volume distortion coefficients} for $\kappa=\tfrac{K}{N}$ with $K\in\R$ and $N\geq1$ as
\begin{equation*}
    \sigma_{\kappa}^{(t)}(\theta):=\begin{cases}
        \tfrac{\mathfrak{s}_\kappa(t\theta)}{\mathfrak{s}_\kappa(\theta)}\quad&\mathrm{if}\,\kappa\theta^2\neq0\,\mathrm{and}\,\kappa\theta^2<\pi^2,\\
        t&\mathrm{if}\,\kappa\theta^2=0,\\
        +\infty&\mathrm{if}\,\kappa\theta^2>+\infty,
    \end{cases}
\end{equation*}
and set $\sigma_{K,N}^{(t)}(0)=t$. Define then
\begin{align*}
    \sigma_{K,N}^{(t)}(\theta):=\sigma_{\frac{K}{N}}^{(t)}(\theta), \ \ \ \ \ \ \ 
    \tau_{K,N}^{(t)}(\theta):=(t\cdot\sigma_{K,N-1}^{(t)}(\theta)^{N-1})^{\frac{1}{N}}.
\end{align*}
When $N=1$ we set $\smash{\tau_{K,1}^{(t)}(\theta)=t}$ if $K\leq0$ and $\smash{\tau_{K,1}^{(t)}(\theta)=+\infty}$ if $K>0$. Note that for every $t\in(0,1)$ and every $\theta>0$, $\smash{\sigma^{(t)}_{K,N}}(\theta)$ is continuous in $(K,N)\in\R\times[1,\infty)$, nondecreasing in $K$ and nonincreasing in $N$ \cite[Rem.\ 2.2]{BS:10}. Analogous
claims apply to the quantity $\smash{\tau_{K,N}^{(t)}(\theta)}$. Furthermore, for every $t\in[0,1]$,
every $\theta\geq0$, and every $\kappa\in(-\infty,\pi^2/\theta^2)$,
$$\sigma_\kappa^{(t)}(\theta)=\sigma_{\kappa\theta^2}^{(t)}(1).$$

The \emph{Rényi relative $N$-entropy} $S_N: \mathscr{P}(X)\rightarrow (-\infty, 0]$ on a metric measure space $(X, \met, \m)$ is defined as
$$ S_N(\mu|\m)= - \int_X \rho^{-\frac{1}{N}} \dint \mu  \mbox{ where } \mu= \rho \dint \m + \mu^\perp,$$
where $\mu^\perp \perp \m$. 
The two following facts are well known \cite{Stu:06b, gsm15}:
\begin{itemize}
   \item By Jensen's inequality, it holds that $S_N(\mu|\m)\geq-\m(\supp\,\mu)^{\frac{1}{N}}$ for every $\mu\in\mathscr{P}_\mathrm{c}(X,\m)$.  
    \item $S_N$ is  weakly lower semicontinuous: if $\{\mu_n\}_{n\in\N}\subseteq\mathscr{P}(X)$ converges weakly  to $\mu\in\mathscr{P}(X)$ and there is a Borel subset $C\subseteq X$ with $\m(C)<\infty$ and $\supp\,\mu_n\subseteq C$ for all $n\in\N$, then
\begin{align*}
S_N(\mu|\m)\leq\liminf_{n\rightarrow\infty}S_N(\mu_n|\m).
\end{align*}
\end{itemize}

We can now recall the definition of the curvature-dimension condition, introduced by Sturm \cite{Stu:06a, Stu:06b} and Lott--Villani \cite{LV:09} for $p=2$. The case for general $p$ is studied in \cite{Kel:17}. The {\it reduced} curvature dimension condition for $p=2$ was introduced later in \cite{BS:10}.

\begin{defi}[$\CD_p(K,N)$ and $\CD^*_p(K,N)$]\label{De:CD and CDstar}
    Given real numbers $K\in\R$ and $N\geq1$, we say that a metric measure space $(X,\met,\m)$ satisfies the \emph{curvature-dimension condition} $\CD_p(K,N)$ if for every $\mu_0,\mu_1\in\mathscr{P}_{p,\mathrm{c}}(X,\m)$ there exists $\bdpi\in\OptGeo_p(\mu_0,\mu_1)$ so that for all $t\in[0,1]$, $\mu_t:=(\e_t)_\sharp \bdpi\ll\m$ and for all $N'\geq N$ it holds that
    \begin{equation}
        \begin{aligned}\label{Eq:CD}
            S_{N'}(\mu_t|\m)\leq-\int_{X^2}\Big[\tau_{K,N'}^{(1-t)}(&\met(x_0,x_1))\rho_0(x_0)^{-\frac{1}{N'}}+\tau_{K,N'}^{(t)}(\met(x_0,x_1))\rho_1(x_1)^{-\frac{1}{N'}}\Big]\,\dint\pi(x_0,x_1),
        \end{aligned}
    \end{equation}
    where $\pi=(\e_0,\e_1)_\sharp \bdpi$ and $\mu_i=\rho_i\m$, $i=0,1$. 
    
    If \eqref{Eq:CD} holds with the coefficients $\sigma^{(t)}_{K,N'}(\theta)$ in place of \smash{$\tau^{(t)}_{K,N'}(\theta)$}, we say that $(X,\met,\m)$ satisfies the \emph{reduced curvature-dimension condition} $\CD_p^*(K,N)$.
\end{defi}

\begin{notation}[Suppressing $p$]\label{Not: CDp}
    The condition $\CD_p(K,N)$ is independent of $p\in[1,\infty)$ thanks to \cite{ACMcCCF:21}, so we will notationally suppress the subscript $p$ everywhere in this case, and work with  $p=2$.
\end{notation}
The measure contraction property was introduced independently by Ohta \cite{Oht:07} and Sturm \cite{Stu:06a, Stu:06b} (see also \cite{kush01}).
We will make use of the version from \cite{Oht:07}.

\begin{defi}[$\mcp^O(K,N)$]\label{De: MCP}
    A metric measure space is said to satisfy the \emph{measure contraction property (in the sense of Ohta)} $\mcp^O(K,N)$ for $K\in\R$ and $N\geq1$ if for every $o\in\supp\,\m$ and $\mu_0\in\mathscr{P}(X)$ of the form $\mu_0=\frac{1}{\m(A)}\m\llcorner A$ for some Borel set $A\subseteq X$ with $0<\m(A)<+\infty$ and $A\subset B_{\pi\sqrt{\frac{N-1}{K}}}(o)$ there exists $\bdpi\in\OptGeo_p(\mu_0,\delta_o)$ such that
    \begin{align}\label{Eq: MCPO}
        \frac{1}{\m(A)}\m\geq(\e_t)_\sharp\left(\tau_{K,N}^{(1-t)}(\met(\gamma_0,\gamma_1))^N\bdpi(\dint\gamma)\right)\qquad\forall t\in[0,1].
    \end{align}
\end{defi}

\begin{rem}[Relations among the curvature-dimension conditions]\label{Rem: Cd rel}
Here we collect some useful properties of the curvature-dimension conditions that we will need later on.
\begin{itemize}
    \item[(i)]\label{Rem: Cd rel 1}\textbf{(Globalization)} Since in general $\tau_{K,N}^{(t)}(\theta)\geq\sigma_{K,N}^{(t)}(\theta)$, the condition $\CD(K,N)$ always implies $\CD^*(K,N)$. The latter is equivalent to its local version $\CD^*_\loc(K',N)$ for all $K'<K$ thanks to \cite[Thm.\ 5.1]{BS:10}. When the metric measure space is essentially nonbranching, due to \cite{CM:21} $\CD(K,N)$ for $N\in (1, \infty)$ is equivalent to its local version. The result for the case $K=\infty$ was previously already known \cite{Stu:06a}. The case $N=1$ is excluded to avoid pathological behaviours,
    cf. \cite[Rem.\ 2.4]{CM:17a}.

    \item[(ii)]\label{Rem: Cd rel 2}\textbf{($\CD(K,N)$ and $\CD(K,\infty)$)} If $(X,\met,\m)$ satisfies $\CD(K,N)$  and $\m(X)=1$, then it satisfies also $\CD(K,\infty)$ by taking the limit as $N\rightarrow\infty$, cf. \cite[Prop.\ 1.6 (ii)]{Stu:06b}.
    
    \item[(iii)]\label{Rem: Cd rel 3}\textbf{($\CD$ and $\mcp^O$)} It is known thanks to \cite{Raj:12} that for $N\in [1, \infty)$ 
    \begin{align*}
        \CD(K,N)\quad\Rightarrow\quad\mcp^O(K,N)
    \end{align*}
    but the converse is in general false \cite{Stu:06b}.

    \item[(iv)]\label{Rem: Cd rel 5}\textbf{(Equivalent pointwise inequality)} If $(X,\met,\m)$ is also essentially nonbranching, thanks to \cite[Prop. 9.1]{CM:21} $\CD(K,N)$ is equivalent to the following pointwise inequality, which holds for $\bdpi$-a.e.\ $\gamma$ (retaining the notation of \eqref{Eq:CD}):
    \begin{align}\label{Eq: CD pointwise}
        \rho_t(\gamma_t)^{-\frac{1}{N}}\geq\tau_{K,N}^{(1-t)}(\met(\gamma_0,\gamma_1))\rho_0(\gamma_0)^{-\frac{1}{N}}+\tau_{K,N}^{(t)}(\met(\gamma_0,\gamma_1))\rho_1(\gamma_1)^{-\frac{1}{N}}.
    \end{align}
    \item[(v)]\label{Rem: Cd rel 4}\textbf{($\mcp^O$ in the smooth case)} If $X$ is a smooth Riemannian manifold with smooth metric tensor $g$, it is not true that $\mcp^O(K,N)$ implies $\Ric_g\geq K$ unless $N=\dim X$, cf. \cite[Cor.\ 5.5 (ii) and Rem.\ 5.6]{Stu:06b}. In general one only has that $N\geq \dim X$.
    
\end{itemize}
\end{rem}
\subsubsection{Lorentzian (pre-)length spaces}
{\LpLSn s are  Lorentzian analogues of metric spaces, while \LLSn s are analogues of metric length spaces. They have been introduced in \cite{KS:18} (after earlier works of Busemann \cite{Bus:67} and Kronheimer--Penrose \cite{KP:67}). There are different variants of the basic axiomatization of these spaces, cf.\ \cite{McC:24, BMcC:23, BBCGMORS:24, MS:25} (and approaches \cite{SV:16, SS:24} based on the null distance and the bounded Lorentzian metric spaces \cite{MS:24, BMS:24}). We opted for sticking to the original setting of \cite{KS:18} as, for example, the topological and metric structure of the underlying space is always given by the cone or product structure.}
\smallskip

Given a set $Y$ we call the triple $(Y,\ll, \leq)$ a \emph{causal space} where $\leq$ is a pre-order, i.e., a reflexive and transitive relation, and $\ll$ is a transitive relation contained in $\leq$ \cite[Def.\ 2.1]{KS:18}. We write $x<y$ when $x\leq y$ and $x\neq y$. We say $x$ and $y$ are \emph{timelike} (resp.\ \emph{causally}) \emph{related} if $x\ll y$ (resp.\ $x\leq y$). 

For $A\subseteq Y$ we define the \emph{chronological} and \emph{causal future} of $A$ as
\begin{align}
&I^+(A):=\{y'\in Y\colon\exists y\in A, y\ll y'\}\,,\qquad  J^+(A):=\{y'\in Y\colon\exists y\in A, y\leq y'\}\,,
\end{align}
respectively. Analogously we define the chronological $I^-(A)$ and the causal past $J^-(A)$ of $A$. In the case $A=\{y\}$ for $y\in Y$ we will write $I^+(x):= I^+(\{x\})$ and $J^+(x)=J^+(\{x\})$. 
The \emph{chronological} and \emph{causal emeralds}  of $A,B\subseteq Y$ are defined as
\begin{align}
&I(A,B):=I^+(A)\cap I^-(B)\,,\qquad J(A,B):=J^+(A)\cap J^-(B)\,.
\end{align}
Moreover, for notational clarity, we define the subsets of timelike (resp.\ causal) pairs as
$$Y^2_{\ll} =\{(x,y)\in Y^2: x\ll y\}\,,\qquad Y^2_{\leq}= \{(x,y)\in Y^2 : x\leq y\}.$$

\begin{defi}[Lorentzian pre-length space {\cite[Def.\ 2.8]{KS:18}}]\label{Def: LpLS}
A \emph{Lorentzian pre-length space} $(Y, \met, \ll, \leq, \uptau)$ is a causal space $(Y, \ll, \leq)$  equipped  with a (proper) metric $\met$ and a lower semi-continuous function $\uptau: Y^2 \rightarrow [0,\infty]$ called \emph{time separation function} that satisfies for all ${x,y, z\in Y}$ 
\begin{itemize}
    \item $\uptau(x,y)=0$ if $x\nleq y$, 
    \item $\uptau(x,y)>0$ if and only if $x\ll y$, 
    \item $\uptau(x,z)\geq \uptau(x,y)+ \uptau(y,z)$ if $x\leq y\leq z$.
\end{itemize}
The set $Y$ is endowed with the metric topology induced by $\met$. The lower semi-continuity of $\uptau$ implies that $I^{\pm}(x)$ is open for all $x\in Y$. We also refer to the function $\uptau(x,\cdot)$ \textnormal{(}and $\uptau(\cdot,y)$\textnormal{)} as
\emph{Lorentzian distance} from a fixed point $x\in Y$ \textnormal{(}to a fixed point $y\in Y$\textnormal{)}.
\end{defi}
If $I\subseteq\R$ is an interval, a curve $\gamma:I\rightarrow Y$ is called (future-directed) timelike (resp.\ causal) if $\gamma$ is locally Lipschitz continuous (w.r.t.\ $\met$) and if for all $s\leq t\in I$, it holds $\gamma(s)\ll \gamma(t)$ (resp.\ $\gamma(s)\leq \gamma(t)$). We say that $\gamma$ is a null
curve if, in addition to being causal, no two points on $\gamma(I)$ are timelike related.

The ($\uptau$-)length $L_\uptau(\gamma)$ of a (future directed) causal curve $\gamma\colon[a,b]\rightarrow Y$ is defined via the time separation function, in analogy to the theory of length metric spaces \cite{KS:18}, i.e.,
\begin{equation*}
    L_\uptau(\gamma):=\inf \left\{\sum_{i=0}^{N-1} \uptau(\gamma(t_i),\gamma(t_{i+1})): N\in\N,\, a=t_0 < t_1 < \ldots < t_N=b\right\}\,.
\end{equation*}
Under natural assumptions, these definitions coincide with the classical ones in the smooth setting \cite{KS:18}.

A future-directed causal curve $\gamma:[a,b]\rightarrow Y$ is called \emph{maximal} (or a \emph{maximizer}) if its length realizes the time separation function between $\gamma(a)$ and $\gamma(b)$, i.e., $L_\uptau(\gamma)= \uptau(\gamma(a), \gamma(b))$. 
A curve $\gamma:[0,1]\rightarrow Y$ will be called
a (causal) \emph{geodesic} if it is maximal and continuous when parametrized proportional by $\uptau$-arc-length. In other words, the set of (causal) geodesics is
$$\Geo(Y)= \{\gamma\in {C}([0,1], Y): \uptau(\gamma(s), \gamma(t))= (t-s)\uptau(\gamma(0), \gamma(1)), \forall s< t\}.$$
Its subset of timelike geodesics is defined as 
$$\TGeo(Y)= \{\gamma\in \Geo(Y): \uptau(\gamma(0), \gamma(1))>0\}.$$

   A subset $G\subseteq\TGeo(X)$  is called \emph{forward timelike nonbranching} if for any $\gamma^1,\gamma^2\in G$ the following holds:
   \begin{align*}
       \exists t\in(0,1)\colon\gamma^1_s=\gamma^2_s\quad\forall s\in[0,t]\quad\Longrightarrow\quad\gamma_s^1=\gamma^2_s\quad\forall s\in[0,1].
   \end{align*}
If $G=\TGeo(Y)$ is forward timelike nonbranching, then we say the Lorentzian pre-length space $(Y,\met, \ll, \leq, \uptau)$ is forward timelike nonbranching. Similarly, one defines \emph{backward timelike nonbranching}, and we just call the Lorentzian pre-length space \emph{timelike nonbranching} if it is forward and backward timelike nonbranching.

Our main results will only deal with \emph{Lorentzian geodesic spaces}. We will use the streamlined definition of \cite[Subsec.\ 1.1]{CM:24a}.
\begin{defi}[Lorentzian geodesic space] A Lorentzian pre-length space $(Y,\met, \ll, \leq, \uptau)$ is termed \emph{Lorentzian geodesic space} if additionally it is: 
\begin{itemize}
\item \emph{$\met$-compatible:} every $x\in Y$ admits a neighbourhood $U$ and a constant $C$ such that $L_{\met}(\gamma)\leq C$ for every future or past directed causal curve $\gamma$ contained in $U$;
\item  \emph{geodesic:} for all $x,y\in Y$ with $x<y$ there is a {\it maximal} future-directed causal curve $\gamma$ from $x$ to $y$, i.e.,  $\uptau(x,y)= L_{\uptau}(\gamma)$.
\end{itemize}
\begin{rem}
In particular, a Lorentzian geodesic space is \emph{(strictly) intrinsic}, i.e., $\forall x\leq y$
\begin{equation*}
    \uptau(x,y) = \sup\{L_\uptau(\gamma): \gamma \text{ future directed causal curve from } x \text{ to } y\}\,,
\end{equation*}
where the supremum is actually a maximum. 
Moreover, a Lorentian geodesic space is a {\it \LLS} in the sense of \cite{KS:18} if it is additionally locally causally closed, $I^\pm(x)\neq\emptyset$ for all $x\in Y$ and timelike path connected, see  \cite[Def.\ 3.22]{KS:18}. 
\end{rem}
\end{defi}

We consider the following version of global hyperbolicity that is consistent with the previous literature \cite[Cor.\ 3.8]{Min:23}. 
\begin{defi}
A Lorentzian geodesic space $(Y,\met, \ll, \leq,\tau)$ is called
\begin{itemize}
\item \emph{causal}: if $\leq$  is also antisymmetric, i.e., $\leq$ is an order; 
\item \emph{globally hyperbolic}: if it is causal and for every $K_1, K_2\subseteq Y$ compact, the \emph{causal emerald} $J^{+}(K_1)\cap J^{-}(K_2)$ is compact in $Y$.
\end{itemize}
\end{defi}

From \cite[Thm.\ 3.7]{Min:23} this definition of global hyperbolicity is equivalent with the one adopted in \cite{KS:18} (i.e., non-total imprisonment and compactness of causal diamonds). 
Global hyperbolicity implies that the relation $\leq$ is a closed subset of $Y\times Y$.
It was proved in \cite[Thm.\ 3.28]{KS:18} that for a globally hyperbolic  Lorentzian length space $(Y,\met, \ll, \leq,\uptau)$, the time-separation function $\uptau$ is finite and continuous: in particular, for a maximal causal curve there exists a constant $\uptau$-speed parametrizations, thus any two distinct causally related points are joined by a causal geodesic, hence the space is a Lorentzian geodesic space. Finally, let us emphasize that for the setting of this article (i.e., generalized cones, see Subsection \ref{subsec-con-ms}) there are no differences: Every generalized cone over a locally compact and complete length metric space $(X,\met)$ is a globally hyperbolic \LLS \cite[Cor.\ 4.11]{AGKS:23}, hence a Lorentzian geodesic space.

\subsubsection{Timelike curvature-dimensions conditions for Lorentzian spaces}
Let $(X,\met,\ll,\leq,\uptau)$ be a {globally hyperbolic}  {\it geodesic} Lorentzian space. 

Given $\mu, \nu\in \mathscr{P}(X)$ we define \emph{chronological} and \emph{causal} couplings: the sets $\Pi_\ll(\mu,\nu)$ and $\Pi_\leq(\mu,\nu)$, respectively, are the sets of $\pi\in\Pi(\mu,\nu)$ such that $\pi(X^2_\ll)=1$ and $\pi(X^2_\leq)=1$, respectively.

Similarly to the non-Lorentzian case, one defines a notion of ``Lorentzian distance'' between probability measures: fixed $p\in(0,1)$, the \emph{$p$-Lorentz-Wasserstein distance} \cite{EM:17} $\ell_p$ of $\mu,\nu\in\mathscr{P}(X)$ is
\begin{align}\label{Eq:ellq}
    \ell_p(\mu,\nu) :=  \left[\!\sup_{\pi\in \Pi_{}(\mu,\nu)}\int_{M\times M} \ell(x,y)^p\,\dint \pi(x,y)\right]^\frac{1}{p}
     =\left[\!\sup_{\pi\in \Pi_{\leq}(\mu,\nu)}\int_{M\times M} \uptau(x,y)^p\,\dint \pi(x,y)\right]^\frac{1}{p}
\end{align}
where $\sup\emptyset=-\infty$ and $\ell:X^2\rightarrow[0,\infty]\cup\{-\infty\}$ is defined by
\begin{align*}
    \ell(x,y):=\begin{cases}
        \uptau(x,y)\quad&\mbox{if } x\leq y,\\
        -\infty&\mbox{otherwise}.
        \end{cases}
\end{align*}

The $p$-Lorentz-Wasserstein distance is interpreted as a sort of time separation on $\mathscr{P}(X)$ since it satisfies the reverse triangle inequality (\cite[Thm.\ 13(ii)]{EM:17}, \cite[Prop.\ 2.5]{CM:24a}). We call  a maximizing $\pi\in\Pi_\leq(\mu,\nu)$ \emph{optimal} and write $\pi\in\Pi^{p,\opt}_\leq(\mu,\nu)$.  If $(X, \met, \ll, \leq , \uptau)$ is a globally hyperbolic geodesic Lorentzian space and $\mu,\nu\in\mathscr{P}_\mathrm{c}(X)$ (satisfying an integrability condition) with $\Pi_\leq(\mu,\nu)\neq\emptyset$, then $\smash{\Pi^{p, \opt}_{\leq}(\mu, \nu)\neq \emptyset}$ \cite[Prop.\ 2.3]{CM:24a}.

We call a pair $\mu,\nu\in\mathscr{P}_\mathrm{c}(X)$ \emph{timelike $p$-dualizable} if $\Pi^{p,\opt}_\leq(\mu,\nu)\cap\Pi_\ll(\mu,\nu)\neq\emptyset$, 
and in this case we say that such an element $\pi$ \emph{$p$-dualizes} $(\mu,\nu)$. Notice that if for $\mu,\nu\in\mathscr{P}_\mathrm{c}(X)$ it holds that $\supp\,\mu\times\supp\,\nu\subseteq X^2_\ll$ then the pair $(\mu,\nu)$ is timelike $p$-dualizable also in the following stronger sense \cite[Def.\ 2.27, Cor.\ 2.29]{CM:24a}.  We say that $(\mu,\nu)\in\mathscr{P}(X)$ is \emph{strongly timelike $p$-dualizable by} $\pi\in\Pi_\ll(\mu,\nu)$ if the pair $(\mu,\nu)$ is timelike $p$-dualizable by $\pi$, and there is an $\ell^p$-cyclically monotone Borel set $\Gamma\subseteq X^2_\ll\cap(\supp\mu\times\supp\nu)$ such that every given $\sigma\in\Pi_\leq(\mu,\nu)$ is $p$-optimal if and only if $\sigma(\Gamma)=1$ (in particular, $\sigma$ is chronological). Similarly, in this case we say that such a $\pi$ \emph{strongly $p$-dualizes} $(\mu,\nu)$. {Here, an \emph{$\ell^p$-cyclically monotone set}, is a set $\Lambda\subseteq X^2_\leq$ such that for every finite set of points $(x_i,y_i)_{i=1,\ldots,N}$ in $\Lambda$  it holds that
\begin{equation*}
    \sum_{i=1}^N \ell^p(x_i,y_i) \geq \sum_{i=1}^N\ell^p(x_{i+1},y_i)\,,
\end{equation*}
with the convention that $x_{N+1}:=x_1$.
}

A curve $(\mu_t)_{t\in[0,1]}\subseteq\mathscr{P}(X)$ is called \emph{timelike (proper-time parametrized) $\ell_p$-geodesic} if there exists a $\bdpi\in\mathscr{P}(\TGeo(X))$ with $(\e_0,\e_1)_\sharp \bdpi$ optimal and $(\e_0,\e_1)_\sharp \bdpi(X^2_\ll)=1$ such that
$$\mu_t=(\e_t)_\sharp \bdpi\qquad\forall t\in[0,1].$$
We call such $\bdpi$ a \emph{timelike $p$-optimal geodesic plan} and write $\bdpi\in\OptTGeo_p(\mu,\nu)$.

We call $(Y, \met, \m, \ll, \leq , \uptau)$ a \emph{measured Lorentzian geodesic space} if $(Y, \met, \ll, \leq , \uptau)$ is a Lorentzian geodesic space and $\m$ is a Radon measure on $X$. 

The following definition is the Lorentzian analogue of Definition \ref{De: ess non-branch}.
\begin{defi}[Timelike $p$-essential nonbranching]
   We say that $(X,\met,\m, \ll, \leq , \uptau)$ is \emph{timelike $p$-essentially nonbranching} if for all $\mu,\nu\in\mathscr{P}_p(X, \m)$ and every $\bdpi\in \OptTGeo_p(\mu,\nu)$ is induced by an $\ell_p$-geodesic that is concentrated on a Borel timelike nonbranching set $G\subseteq \TGeo(X)$.
\end{defi}
\begin{rem}[{\cite{CM:24a,Bra:23b}}]\label{Re:reparametrizations}
    The elements of $\TGeo(X)$ are $\uptau$-arclength parametrized, i.e.,  $\gamma\in\TGeo(X)$ satisfies
    \begin{align*}
        \uptau(\gamma_s,\gamma_t)=(t-s)\,\uptau(\gamma_0,\gamma_1)\qquad\forall s,t\in[0,1],\quad s<t.
    \end{align*}
    As a consequence of this and the reverse triangle inequality, every timelike $\ell_p$-geodesic $(\mu_t)_{t\in[0,1]}$ with $\ell_p(\mu_0,\mu_1)<\infty$ similarly satisfies
    \begin{align*}
        \ell_p(\mu_s,\mu_t)=(t-s)\,\ell_p(\mu_0,\mu_1)\qquad\forall s,t\in[0,1],\quad s<t.
    \end{align*}
\end{rem}
Though we only use it in the context of measured Lorentzian geodesic spaces, the following definition makes sense in the context measured Lorentzian pre-length spaces.
\begin{defi}[Timelike curvature-dimension conditions]\label{Def: TCD}
    Fix $p\in(0,1)$ and let $K\in\R$ and $N\geq1$.  A measured Lorentzian pre-length space $(X, \met, \m, \ll, \leq , \uptau)$ satisfies the \emph{timelike curvature-dimension condition} $\tCD_p(K,N)$ if for every timelike $p$-dualizable pair $(\mu_0,\mu_1)=(\rho_0\,\m,\rho_1\,\m)\in\mathscr{P}_\mathrm{c}(X,\m)$ there exists $\bdpi\in\OptTGeo_p(\mu_0,\mu_1)$ such that for all $t\in[0,1]$ and $N'\geq N$ we have
    \begin{equation}
        \begin{aligned}\label{Eq:TCD}
           S_{N'}(\mu|\m)\leq -\int_{X\times X}\Big[\tau_{K,N'}^{(1-t)}(&\met(x_0,x_1))\rho_0(x_0)^{-\frac{1}{N'}}+\tau_{K,N'}^{(t)}(\met(x_0,x_1))\rho_1(x_1)^{-\frac{1}{N'}}\Big]\,\dint\pi(x_0,x_1),
        \end{aligned}
    \end{equation}
    where $\pi=(\e_0,\e_1)_\sharp \bdpi$ and $\mu_i=\rho_i\m$, $i=0,1$. 

    If \eqref{Eq:TCD} holds with the coefficients $\sigma^{(t)}_{K,N'}(\theta)$ in place of \smash{$\tau^{(t)}_{K,N'}(\theta)$}, we say that $X$ satisfies the \emph{reduced timelike curvature-dimension condition} $\tCD^*_p(K,N)$.

    If the previous statements hold for every strongly timelike $p$-dualizable $(\mu_0,\mu_1)\in\mathscr{P}_\mathrm{c}(X, \m)^2$, we say that $X$ obeys the \emph{weak timelike curvature-dimension condition} $\wTCD_p(K,N)$ and the \emph{weak reduced timelike curvature-dimension condition} $\wTCD^*_p(K,N)$ respectively.
\end{defi}

Notice in the definition above the dependence on the transport exponent $p\in(0,1)$. It is  conjectured that the condition $\tCD_p(K,N)$  is independent of $p$, as it is for $\CD(K,N)$ (see \Cref{Not: CDp}).

Next, we introduce the timelike measure contraction properties. The idea behind is the same as in the metric case: one of the two measures in the definition of $\tCD_p(K,N)$ becomes a Dirac measure at a point that in this case is causally related to the support of the other measure. 

\begin{defi}[Timelike measure-contraction properties]\label{Def: TMCP}
    Fix $p\in(0,1)$, let $K\in\R$ and $N\geq1$. We say that $(X, \met, \m, \ll, \leq, \uptau)$ satisfies the \emph{future timelike measure contraction property} $\tmcp^+(K,N)$ if for every $\mu_0=\rho_0\,\m\in\mathscr{P}^\ac_c(X)$ and every $x_1\in X$ with $\mu_0(I^-(x_1))=1$ there exists a timelike $\ell_p$-geodesic $(\mu_t)_{t\in[0,1]}$ from $\mu_0$ to $\mu_1:=\delta_{x_1}$ such that for all $t\in[0,1]$ and $N'\geq N$ we have
    \begin{equation}
        \begin{aligned}\label{Eq:TMCP+}
            \int_X\rho_t^{1-\frac{1}{N'}}\,\dint\m\geq\int_{X^2}\tau_{K,N'}^{(1-t)}(\uptau(x,x_1))\rho_0(x^0)^{-\frac{1}{N'}}\,\dint\mu(x),
        \end{aligned}
    \end{equation}
    where $\pi=(\e_0,\e_1)_\sharp \bdpi$ and $\mu_i=\rho_i\m$, $i=0,1$.

    We say that $X$ satisfies the \emph{past} $\tmcp^-(K,N)$ if \eqref{Eq:TMCP+} holds with $\mu_1=\rho_1\,\m$ and $\mu_0=\delta_{x_0}$ for some $x_0\in X$ with $\mu_1(I^+(x_0))=1$, i.e.,
      \begin{equation}
        \begin{aligned}\label{Eq:TMCP-}
            \int_X\rho_t^{1-\frac{1}{N'}}\,\dint\m\geq\int_{X^2}\tau_{K,N'}^{(t)}(\uptau(x_0,x))\rho_1(x)^{-\frac{1}{N'}}\,\dint\mu(x).
        \end{aligned}
    \end{equation}
    
     Finally, we say that $X$ satisfies $\tmcp(K,N)$ if it satisfies both the conditions $\tmcp^+(K,N)$ and $\tmcp^-(K,N)$.

     If the same statements hold with the coefficients $\smash{\sigma_{K,N}^{(t)}(\theta)}$ in place of $\smash{\tau_{K,N}^{(t)}(\theta)}$, we just put a prefix \emph{reduced} to the respective properties $\smash{\tmcp^{*,\pm}(K,N)}$ and $\tmcp^*(K,N)$.
\end{defi}

\begin{rem}[Relations among the timelike curvature-dimension conditions]\label{Rem: TCD rel}
    As in the case of $\CD(K,N)$, we collect some properties of its timelike counterpart:
    \begin{itemize}
        \item [(i)]\label{Rem: TCD rel 1}\textbf{(Globalization)} Under strong assumptions, the condition $\wTCD^*_{p,\loc}(K,N)$ is equivalent to $\wTCD^*_p(K,N)$, see \cite[Thm.\ 3.45]{Bra:23b} for details. The same statement for $\wTCD_p$ is not known at the moment.
        \item[(ii)] \textbf{($\wTCD_p(K,N)$ and $\wTCD_p(K, \infty)$)}
    Combining \cite[Prop.\ 3.6 (i) and Cor.\ 3.10]{Bra:23b} and \cite[Rem.\ 4.2]{Bra:23a} we have that $\wTCD_p(K,N)$ implies $\wTCD_p(K,\infty)$ in the sense of \cite[Def.\ 4.1]{Bra:23a} which is the Lorentzian version of \Cref{De: CD for inf K}. 

        \item[(iii)]\label{Rem: TCD rel 2}\textbf{($\wTCD_p$ and $\tmcp$)} The condition $\wTCD_p(K,N)$ implies $\tmcp(K,N)$ and the same holds true for its reduced version, cf.\ \cite[Prop.\ 4.9]{Bra:23b}.
        
        \item[(iv)]\label{Rem: TCD rel 3}\textbf{(Equivalent pointwise inequality)} If $(X,\met,\ll,\leq,\uptau,\m)$ is also timelike $p$-essentially nonbranching, thanks to \cite[Thm.\ 3.41]{Bra:23b} $\tCD_p(K,N)$ and $\wTCD_p(K,N)$ are equivalent to the following pointwise inequality,  for $\bdpi$-a.e.\ $\gamma$ (retaining the notation of \eqref{Eq:TCD}):
        \begin{align}\label{Eq: TCD pointwise}
            \rho_t(\gamma_t)^{-\frac{1}{N}}\geq\tau_{K,N}^{(1-t)}(\uptau(\gamma_0,\gamma_1))\rho_0(\gamma_0)^{-\frac{1}{N}}+\tau_{K,N}^{(t)}(\uptau(\gamma_0,\gamma_1))\rho_1(\gamma_1)^{-\frac{1}{N}}.
        \end{align}
        Evident adaptations hold for $\wTCD_p^*(K,N)$ and for $\tmcp(K,N)$ and their variants \cite{Bra:23b}.
    \end{itemize}
\end{rem}

\begin{rem}[$\tmcp$ and $\tmcp^e$]\label{rem:entropicmcp} Under the assumption of $p$-essentially nonbranching, the measure contraction property $\tmcp^+$ also implies the entropic measure contraction property $\tmcp^e$ of \cite{CM:24a}.  In particular, by \cite[Prop.\ 3.19]{CM:24a}, it follows for every pair $(\mu_0, \mu_1)\in \mathscr P_c(X)\times \mathscr P_c(X)$ of timelike $p$-dualizable measures there exists $\pi \in \Pi_{\leq}^{p, \opt}(\mu_0, \mu_1)$  that 
$$U_N(\mu_t)\geq \sigma_{K,N}^{(t)}(\left\|\uptau\right\|_{L^2(\pi)})U_N(\mu_0).$$
In particular, if $\mu_0= \frac{1}{\m(A_0)} \m|_{A_0}$ for a set $A_0$ of positive measure, by Jensen's inequality one has 
$$\m(A_t)^{\frac{1}{N}}\geq \sigma_{K,N}^{(t)}(\left\|\uptau\right\|_{L^2(\pi)}) \m(A_0)^{\frac{1}{N}}.$$
\end{rem}    

\subsubsection{Weighted Lorentzian manifolds and the Bakry--\'Emery condition}
Let $(M,g)$ be a smooth Lorentzian spacetime of dimension $n\in \N_{\geq 2}$, i.e.,  a smooth, connected, time-oriented Lorentzian manifold. For this subsubsection we choose the signature convention to be $(+, -, \dots, -)$, but later it will be $(-,+,+,\ldots,+)$, see Remark \ref{rem-sign-conv} below. Also, we assume $(M,g)$ is globally hyperbolic. 
\smallskip

For $N>n$ and $\Phi\in C^\infty(M)$ we define the Bakry--\'Emery $N$-Ricci tensor as 
\begin{align}
\Ric^{\Phi, N}_g= \Ric_g-\nabla^2 \Phi - \frac{1}{N-n} \dint\Phi \otimes \dint\Phi 
.
\end{align}
Here $\nabla^2 \Phi$ is the Hessian form w.r.t.\ $g$ of $\Phi$. If $N=n$, we set 
\begin{align}\Ric_g^{N, \Phi}(v,v)= \begin{cases} \Ric_g(v,v)- \nabla^2 \Phi(v,v) & \mbox{ if } \langle \nabla \Phi, v\rangle =0,\\
-\infty & \mbox{ otherwise.}
\end{cases}
\end{align}
\begin{thm}\label{thm:smoothcd} Let $p\in (0,1)$, $K\in \R$ and $N\in [n, \infty)$. The following statements are equivalent. 
\begin{itemize}
\item[(i)] $\Ric_g^{N, \Phi}\geq Kg$, i.e., $\Ric_g^{N, \Phi}(v,v)\geq Kg(v,v)$ for every timelike vector $v\in TM$.
\item[(ii)] The condition $\tCD_{p}(K,N)$ holds {(w.r.t.\ the reference measure $\m= e^\Phi\,\vol_g$).}
\end{itemize}
Moreover, if $\Ric_g\geq K$ on timelike vectors and $h= e^{\Phi}$ is measurable and satisfies
\begin{align*}
h^{\frac{1}{N}} \circ \gamma(t) \geq \sigma_{K'/N}^{(1-t)}(\tau(\gamma(0), \gamma(1))) h^\frac{1}{N}\circ \gamma(0)+  \sigma_{K'/N}^{(t)}(\tau(\gamma(0), \gamma(1))) h^\frac{1}{N}\circ \gamma(1)
\end{align*} 
for each timelike, constant-speed geodesic $\gamma: [0,1]\rightarrow M$, then $(M,\met,\leq,\ll,\uptau, e^\Phi\mathrm{vol}_g)$ satisfies  $\tCD_{p}(K+K', N+n)$. 
\end{thm}
\begin{proof} For the equivalence between (i) and (ii) we refer to Theorem A.1 in the Appendix A in \cite{Bra:23b}, where the equivalence was established for $\tCD^*_p$-spaces \cite{McC:20}.

For the second statement we refer to \cite[Thm.\ 1.7]{Stu:06b} and its proof.
\end{proof}
\begin{rem}\label{rem-sign-conv} We recall that 
$\Ric_g(v,v)= \sum_{i=1}^n \epsilon_i g(R(e_i, v)v, e_i)$
for a timelike vector $v\in TM$, an orthonormal basis $(e_i)_{i=1, \dots, n}$, and $\epsilon_i = g(e_i, e_i)$. Since the curvature tensor $R$ only depends on the Levi-Civita connection that is the same for $g$ and $-g$, we have $\Ric_{g}= \Ric_{-g}$. It follows that 
a Ricci curvature bound $\Ric_g^{N, \Phi}\geq Kg$ for a metric $g$ with signature $(+, -, \dots, -)$ as in {\it (i)} of \Cref{thm:smoothcd} holds if and only if there is  a  bound $\Ric_{g^-}^{N, \Phi}\geq -Kg^-$ for the metric $g^-=-g$ with signature $(-,+, \dots, +)$. Hence, for a metric $g$ with signature $(-, +, \dots, +)$ the statement of the previous theorem is that $\Ric_g^{N, \Phi}\geq Kg$ holds if and only if $\tCD(-K, N)$ holds. We emphasize this since there is no unified signature convention in the literature and Theorem \ref{thm:smoothcd} will be crucial in the proof of our main results.
\end{rem}

\subsection{The localisation method on metric measure spaces}\label{subsec-loc}
\subsubsection{Disintegration of measures} 
In this subsection we follow the presentation in \cite{fremlin} (see also \cite{CM:21, zhenhao24}).

We consider  a measurable space $(R,\mathcal R)$, a set $Q$ and a map $\mathfrak Q:R\rightarrow Q$. The set $Q$ is equipped with the smallest $\sigma$-algebra $\mathcal Q$ such that $B\in \mathcal Q$  if $\mathfrak Q^{-1}(B)\in \mathcal{R}$. Given a measure $\m$ on $(R,\mathcal R)$, one can define {its quotient} measure $\mathfrak q$ on $Q$ via the pushforward $\mathfrak Q_{\sharp}\, \m=: \mathfrak q$.

A disintegration of $\m$ over  $\mathfrak q$ is a map $\mathcal R \times Q\ni (B,q)\mapsto \m_{q}(B)\in [0,\infty]$ such that the following holds
\begin{itemize}
\item $\m_{q}$ is a measure on $(R,\mathcal R)$ for every $q\in Q$,
\item $q\mapsto \m_{q}(B)$ is $\mathfrak q$-measurable for every $B\in \mathcal R$,
\item $\m(B)= \int_Q\m_q(B)\,\dint\mathfrak q$ for every $B\in \mathcal R$.
\end{itemize}We use the notation $\{\m_{q}\}_{q\in Q}$ for such a disintegration. We call the measures $\m_{q}$ {\it conditional  measures}.
A disintegration $\{\m_q\}_{q \in Q}$ over $\mathfrak q$ is called \emph{consistent w.r.t.\ $\mathfrak Q$}  if for all $B\in \mathcal R$ and $C\in \mathcal Q$ the consistency condition
$$
\m(B\cap \mathfrak Q^{-1}(C))=\int_C \m_{q}(B) \,\dint\mathfrak q(q)
$$
holds.  
A 
disintegration $\{\m_{q}\}_{q\in Q}$  
 is called {\it strongly consistent }if it is consistent w.r.t.\ its quotient map and for $\mathfrak q$-a.e.\ $q$ the conditional measure $\m_q$ is concentrated in $\mathfrak Q^{-1}(\{q\})$.
 
Finally, we say that a consistent disintegration $\{\m_q\}_{q \in Q}$ over $\mathfrak q$ is \emph{$\mathfrak q$-unique} if the following holds. Let $ \{\m'_{q}\}_{q\in Q}$ be another consistent disintegration of $\m$ over $\mathfrak q$. Then $\m_{q}=\m'_{q}$ for $\mathfrak q$-a.e.\ $q \in Q$. If $\mathcal R$ is countably generated and $\m$ is $\sigma$-finite, then a consistent disintegration of $\m$ over $\mathfrak q$ is $\mathfrak q$-unique.
\subsubsection{$1D$ localisation}  
Let $(X,d,\m)$ be a proper metric measure space that is essentially nonbranching. We  assume that $\supp\m =X$.
Let $u:X\rightarrow \mathbb{R}$ be a $1$-Lipschitz function. Then 
\begin{align*}
\Gamma_u:=\{(x,y)\in X\times X : u(x)-u(y)=d(x,y)\}
\end{align*}
is a  $d$-cyclically monotone set. One also defines $$\Gamma_u^{-1}=\{(x,y)\in X\times X: (y,x)\in \Gamma_u\}.$$
If $\gamma:[a,b]\rightarrow X$ is a geodesic such that $(\gamma(a),\gamma(b))\in \Gamma_u$ then $(\gamma(t),\gamma(s))\in \Gamma_u$ for $a<t\leq s<b$.  It is therefore natural to consider the set $G$ of unit speed geodesics $\gamma:[a,b]\rightarrow \R$ such that $(\gamma(t),\gamma(s))\in \Gamma_u$ for all $a\leq t\leq s\leq b$. 
The union $\Gamma_u\cup \Gamma_u^{-1}$ defines a relation $R_u$ on $X\times X$, and $R_u$ induces a {\it transport set with endpoints} 
$$\mathcal T^b_u:= P_1(R_u\backslash \{(x,y):x=y\})\subseteq X,$$ 
where $P_1(x,y)=x$ is the projection onto the first factor. For $x\in \T^b_u$ we define the sections $\Gamma_u(x):=\{y\in X:(x,y)\in \Gamma_u\}, $
and similarly $\Gamma_u^{-1}(x)$, as well as $R_u(x)=\Gamma_u(x)\cup \Gamma_u^{-1}(x)$. Since $u$ is $1$-Lipschitz, 
 $\Gamma_u, \Gamma_u^{-1}, R_u\subseteq X^2$ and  $\Gamma_u(x), \Gamma_u^{-1}(x), R_u(x)\subseteq X$ are closed.

The {\it transport set without branching} $\mathcal T_u$ associated to $u$ is  defined as 
\begin{align*}
\mathcal T_u=\left\{ x\in \mathcal T^b_u: \forall  y,z\in R_u(x) \Rightarrow  (y,z)\in R_u\right\}.
\end{align*}
The sets $\T^b_u$ and $\T^b_u\setminus \T_u$ are $\sigma$-compact, and $\T_u$ and $R_u\cap \T_u\times \T_u$ are Borel sets.
Moreover, $\T_u$ is  $\sigma$-compact. 

In \cite{Cav:17} Cavalletti shows that $R_u$ restricted to $\T_u\times \T_u$ is an equivalence relation. 
Hence, from $R_u$ one obtains a partition of $\mathcal T_u$ into a disjoint family of equivalence classes $\{X_{q}\}_{q\in Q}$. 

Let $\mathfrak Q: \mathcal T_u\rightarrow Q$ be the induced quotient map. We  equip
 $Q$ with the   measurable structure $\mathcal Q$ such that $\mathfrak Q: \T_u\rightarrow Q$ is a measurable map. 
We set $\smash{\mathfrak q:= \mathfrak Q_{\sharp}\,[\m|_{\mathcal T_u}]}$. 

Every $X_{q}$ is isometric to an interval  $I_q\subseteq\mathbb{R}$ via a distance preserving map $\gamma_q:I_q \rightarrow X_{q}$. The curve $\gamma_q:I_q\rightarrow X$ extends to an arclength parametrized geodesic  on $\overline I_{q}$  that we also denote $\gamma_q$. We denote the  open interior of $\overline I_q$ with $(a_q,b_q)$ where $a_q, b_q\in \R\cup \{\pm \infty\}.$

A \emph{measurable section} of the equivalence relation $R_u$ on $\T_u$ is a measurable map $s: \T_u\rightarrow \T_u$ such that $R_u(s(x))=R_u(x)$, and $(x,y)\in R_u$ implies $s(x)=s(y)$. In \cite[Prop.\ 5.2]{Cav:17} Cavalletti shows that there exists a measurable section $s$ of $R_u$ on $\T_u$. Therefore, one can identify the measurable space $Q$ with the measurable set $\{x\in \T_u: x=s(x)\}\subseteq X$ equipped with the induced measurable structure. Hence $\mathfrak q$ is a Borel measure on $X$. By inner regularity there exists a $\sigma$-compact set $Q'\subseteq X$ such that $\mathfrak q(Q\setminus Q')=0$. We have that the Borel $\sigma$-algebra induced on $Q'$ is contained in $\mathcal Q$. In the following we will replace $Q$ with $Q'$ without further comments.  The map $\gamma_q$ with {$q\in Q$ is reparametrized such that $\gamma_q(0)=s(x)$.} The functions $Q\ni q \mapsto a_{q}, b_{q}\in \R$ are measurable. 
%
\begin{lem}[{\cite[Lem.\ 3.4]{CM:20}}]
Let $(X,d,\m)$ be an essentially nonbranching $\CD^*(K,N)$ space for $K\in \R$ and $N\in (1,\infty)$ with $\supp \m=X$ and $\m(X)<\infty$.
Then, for any $1$-Lipschitz function $u:X\rightarrow \R$, it holds that $\m(\T^b_u\setminus \T_u)=0$.
\end{lem}
For $\mathfrak q$-a.e.\ $q\in Q$ it was proved in \cite[Thm.\ 7.10]{CM:21} that 
\begin{align}\label{int}
R_u(x)=\overline{X_q}\supset X_q \supset (R_u(x))^{\circ} \ \ \forall x\in \mathfrak Q^{-1}(q).
\end{align}
where $(R_u(x))^\circ$ denotes the relative interior of the closed set $R_u(x)$. We replace $Q$ with a set, that is also denoted $Q$, such that \eqref{int} holds for every $q\in Q$.
%
\begin{thm}[{\cite[Thm.\ 3.3]{CM:20}}]\label{Th:1Dloc}
Let $(X,d,\m)$ be
a complete, proper, geodesic metric measure space 
with $\supp\m =X$ and $\m$ $\sigma$-finite. Let $u:X\rightarrow \mathbb{R}$ be a $1$-Lipschitz function,  let $(X_q)_{q\in Q}$ be the induced partition of $\mathcal T_u$ via $R_u$, and let $\mathfrak Q: \T_u\rightarrow Q$ be the induced quotient map as above.
Then, there exists a strongly consistent disintegration $\{\m_q\}_{q\in Q}$ of $\m|_{\T_u}$ w.r.t.\ $\mathfrak Q$. 
\end{thm}
Define the ray map $\mathfrak{G}$ as 
\begin{align*}
\mathfrak G:  \mathcal V\subseteq Q\times \R\rightarrow X\ \mbox{ via }\
\mbox{graph}(\mathfrak G)=\{ (q, t,x) \in Q\times \R\times X: t\in I_q, \gamma_q(t)=x\}.  
\end{align*}
We have the following:
\begin{itemize}
\item 
the map $\mathfrak G$ is Borel measurable,  and 
by definition $\mathcal V= \mathfrak G^{-1}(\T_u)$;
\item $\mathfrak G(q,\cdot)=\gamma_q: I_q\rightarrow X$; 
\item $\mathfrak G:\mathcal V\rightarrow  \T_u$ is bijective and its inverse is given by $\mathfrak G^{-1}(x)=\big( \mathfrak Q(x), \pm d(x,\mathfrak Q(x))\big)$.
\end{itemize}
\begin{thm}\label{Th:CD1}
Let $(X,d,\m)$ be an essentially nonbranching $\CD(K,N)$  space with $\supp\m=X$, $\m$ $\sigma$-finite, $K\in \R$ and $N\in (1,\infty)$.

Then, for any $1$-Lipschitz function $u:X\rightarrow \R$ there exists a disintegration $\{\m_q\}_{q\in Q}$ of $\m$ that is strongly consistent with $R^b_u$. 

Moreover, for $\mathfrak q$-a.e.\ $q\in Q$, $\m_{q}$ is a Radon measure with $\m_q=h_q\mathcal{H}^1|_{X_q}$ and $(X_q, d|_{X_q\times X_q}, \m_q)$ verifies the condition $\CD(K,N)$. 
\end{thm}
\begin{rem}[$\CD(K,N)$-density]\label{rem:thefollowingremark}
More precisely, for $\mathfrak q$-a.e.\ $q\in Q$ it holds that
\begin{align}\label{kuconcave}
h_q(c_t)^{\frac{1}{N-1}}\geq \sigma_{K/N-1}^{(1-t)}(|\dot c_t|)h_q(c_0)^{\frac{1}{N-1}}+\sigma_{K/N-1}^{(t)}(|\dot c_t|)h_q(c_1)^{\frac{1}{N-1}}
\end{align}
for every geodesic $c:[0,1]\rightarrow (a_q,b_q)$.

The property
\eqref{kuconcave} yields that $h_q$ is locally Lipschitz continuous on $(a_q,b_q)$, and  $h_q:\R \rightarrow (0,\infty)$ satisfies
\begin{align}\label{kuode}
\frac{d^2}{dr^2}h_q^{\frac{1}{N-1}}+ \frac{K}{N-1}h_q^{\frac{1}{N-1}}\leq 0 \mbox{ on $(a_q, b_q)$} \mbox{ in distributional sense.}
\end{align}Moreover, $h_q$ is continuous on $\overline{(a_q, b_q)}$ for $\mathfrak q$-a.e.\ $q\in Q$. It follows that \eqref{kuconcave} holds for every geodesic $c:[0,1]\rightarrow \overline{(a_q, b_q)}$ for $\mathfrak q$-a.e.\ $q\in Q$.

Hence, $h_q$ is a $\CD(K,N)$-density on $\overline{(a_q, b_q)}$ in the sense of Cavalletti-Milman \cite{CM:21}. Such functions have been studied also in \cite{EKS:15}.\end{rem}

\begin{rem}[Properties of $\CD(K,N)$-densities]\label{Re:properties of CD(K,N)-densities and logarithmic convolutions}
We  collect  some properties of $\CD(K,N)$-densities from \cite{CM:21} that we will need later.
\begin{enumerate}
    \item $h$ is strictly positive in  $(a,b)$ or it identically vanishes.
    \item $h$ is locally semi-concave in the interior of its definition, i.e., for all $r_0\in(a,b)$ there exists a constant $C_{r_0}\in\R$ such that $h(r)-C_{r_0}r^2$ is concave in a neighbourhood of $r_0$.
    \item $h$ is locally Lipschitz continuous in the interior $(a,b)$.
\item A  metric measure space $(I, \m)$ for an interval $I\subset \R$ and a measure $\m$ on $I$ satisfies the $\CD(K,N)$ condition if and only if $\m= h \dint \mathcal L^1$ and $h$ is a $\CD(K,N)$-density. 
\end{enumerate}
\end{rem}
\begin{rem}
If a function $h: \overline{(a,b)}\rightarrow \R$ is upper semi-continuous and satisfies the differential inequality \eqref{kuode} in $(a,b)$, then $h$  satisfies \eqref{kuconcave} for every geodesic $c$ in $\overline{(a,b)}$ \cite[Lem.\ 2.8]{EKS:15}.
\end{rem}

\subsubsection{Characterization of curvature bounds via $1D$ localisation}\label{subsec:CD1}
Given a continuous function $\phi: X\rightarrow \R$ such that $\phi^{-1}(\{0\})=: S\neq \emptyset$, we define the signed distance function w.r.t.\ $\phi$ as 
$$\met_\phi: X\rightarrow \R, \ \met_\phi(x) =\sgn(\phi(x)) \cdot\inf_{y\in S} \met(y,x) .$$
When $(X,\met)$ is a length space $\met_\phi$ is $1$-Lipschitz. 

Let $K\in \R$ and $N\geq 1$.
\begin{defi} 
Let $(X,\met_X,\m_X)$ be an essentially nonbranching metric measure length space with $\supp \m_X= X$, and let $u=\met_\phi$ be a signed distance function with the associated partition $\{X_q\}_{q \in Q}$ of $\mathcal T_u$ where $Q$ is equipped with the measurable structure as before such that the quotient map $\mathfrak Q$ is measurable. We say that $(X,\met_X,\m_X)$ satisfies the condition $\CD_u^1(K,N)$  if 
\begin{itemize}
\item[(i)] There exists a $\mathfrak q$-unique disintegration $\{\m_q\}_{q\in Q}$ of $\m|_{\mathcal T_u}$ over $\mathfrak q$.
\item[(ii)] $Q$ is the image of a measurable section $s: \mathcal T_u\rightarrow \mathcal T_u$, and there exists a Borel set $Q'\subseteq Q$, a subset of $\mathcal T_u$, such that $\mathfrak q(Q\setminus Q')=0$.
\item[(iii)] For $\mathfrak q$-a.e.\ $q\in Q$ the set $R_u(x)=\overline X_q$ for every $x\in \mathfrak Q^{-1}(\{q\})$, and $X_q$ is the image $\mbox{Im}(\gamma_q)$ of a geodesic $\gamma_q:I_q\rightarrow X$ for an interval $I_q\subseteq\R$.
\item[(iv)] For $\mathfrak q$-a.e.\ $q\in Q$ $\m_q$ is supported on $\overline X_q$ and  the metric measure space $(X_q, \met_X|_{X_q\times X_q},\m_q)$ satisfies the condition $\CD(K,N)$.
\item[(v)] For every bounded set $K\subseteq X$ there exists $C_K\in (0, \infty)$ such that $\m_q(K)\leq C_K$ for $\mathfrak q$-a.e.\ $q\in Q$. 
\end{itemize}
\noindent
The metric measure space $(X,d_X,\m_X)$ satisfies the condition $\CD^1(K,N)$ if it satisfies the condition $\CD^1_u(K,N)$ for any $1$-Lipschitz function $u:X\rightarrow \R$.
\end{defi}
\begin{rem}
From the previous sub-subsection it is immediately clear that the condition $\CD(K,N)$ implies the condition $\CD^1(K,N)$. 
\end{rem}
\begin{thm}[{\cite{CM:21, zhenhao24}}]\label{thm:cavmil}
If an essentially nonbranching, locally finite metric measure space $(X,\met_X,\m_X)$, such that $\m_X$ has full support, satisfies the condition $\CD^1(K,N)$ for $K\in \R$ and $N\in [1,\infty)$ then it satisfies the condition $\CD(K,N)$.
\end{thm}

\section{Generalized cones and warped products}\label{sec-gen-con-wpd-prod}
Here we discuss semi-Riemannian warped products in the smooth and nonsmooth setting with a particular focus on the case of a one-dimensional base, and also with one-dimensional fiber.

\subsection{Warped products over semi-Riemannian manifolds}\label{sec:wpsemiriem}
Let $(B,g_B)$ be a $d$-dimensional semi-Riemannian manifold and let $f: B\rightarrow (0, \infty)$ be a smooth function. Let $(F,g_F)$ be an $n$-dimensional semi-Riemannian manifold. The \emph{semi-Riemannian warped product} of $B$, $F$  and $f$ is defined as the product manifold
$C= B\times F$
equipped with the semi-Riemannian metric 
$$g_{_{B\times_f F}} = (P_B)^* g_B + (f\circ P_B)^2 (P_F)^* g_F\,,$$
where $(P_B)^*$ and $(P_F)^*$ are the pullbacks onto $B\times F$ w.r.t.\ the projections maps from $B\times F$ onto $B$ and $F$. The spaces $B$, $F$ and the function $f$ are called the base, the fiber and the warping function, respectively. The semi-Riemannian manifold $(B\times F, g_{B\times_f F})$ has index (i.e., the dimension of the largest negative definite subspace) $\nu_B + \nu_F$, where $\nu_B$ and $\nu_F$ are the index of $g_B$ and $g_F$, respectively. 

We call $\smash{P_F^{-1}(\{x\})}$, for $x\in F$, the leaves and $\smash{P_B^{-1}(\{t\})}$, for $t\in B$, the fibers. Vectors tangent to leaves are called \emph{horizontal}, vectors tangent to fibers are called \emph{vertical}.  We write $(\cdot)^{\perp}$ and $(\cdot)^\top$ for the horizontal and the vertical projections. One can lift vector fields on $B$  (respectively on $F$) as horizontal (respectively vertical) vector fields to $B\times F$ and the family of all such lifts is denoted with $\mathcal L(B)$ (respectively $\mathcal L(F)$). We use the same notation for vector fields on $B$ (respectively on $F$) and their lifts on $B\times F$.  If $h$ is a smooth function on $B$ then the  lift of the gradient $\nabla h$ is the gradient of $h\circ P_B$.

Recall the following formulas for the Levi-Civita connection $\nabla$ of $B\times_fF$ (see \cite[Prop. 7.35]{ONe:83}). Let $X, Y\in \mathcal L(B)$ and $V,W\in \mathcal L(F)$, then 
\begin{enumerate}
\item $\nabla_XY$ is the lift of $\nabla^B_X Y$. 
\item $\nabla_XV= \nabla_VX= (X\log f)V= g_B(X, \nabla \log f) V$. 
\item $(\nabla_VW)^\perp= - g_F (V,W) \nabla \log f$.
\item $(\nabla_VW)^\top$ is the lift of $\nabla^F_VW$. 
\end{enumerate}
\noindent
{\bf Assumption:} 
From now on we assume that both $B$ and $F$ are Riemannian manifolds with Riemannian metrics $g_B$ and $g_F$, but we consider $B$ as a semi-Riemannian manifold $^\pm\! B$ that  is either equipped with the Riemannian metric $g_B$ or with the semi-Riemannian $\smash{-g_B=g_{^{{-}}\!B}}$. The warped product metric is given then by
$$g_{^\pm\! B\times_f F} = (P_B)^* g_{^\pm B} + (f\circ P_B)^2 (P_F)^* g_F= \pm (P_B)^* g_{B} + (f\circ P_B)^2 (P_F)^* g_F\,.$$

\begin{rem}[Hessian and Laplacian]\label{Rem: hessian and laplacian}
Recall that the Hessian of  $h\in C^\infty(B)$ w.r.t.\ $g_B$ is 
$$\nabla^2 h(X,Y) = X(Y h) - \nabla^B_XY h= g_B(\nabla^B_X \nabla h, Y)\,,$$
where $\nabla^B$ is the Levi-Civita connection of $g_B$.

We note also that by Koszul's formula the Levi-Civita connections w.r.t.\ $g_B$ and w.r.t.\ $-g_B$ are the same.  In particular, the Hessians $\nabla^2$ w.r.t.\ $g_B$ and w.r.t.\ $-g_B$ are the same. On the other hand, for the gradient of a function $h$ w.r.t.\ $g_{^- B}$ we have  $\smash{\nabla^{^-\!B} h= - \nabla^{B} h}$ because by definition of the gradient
$$ g_B(\nabla^B h, X)= g_{^-B}(- \nabla^B h, X).$$ 
Moreover $-g_{B}(\nabla^B h, \nabla^B h)= g_{^- B}(\nabla^{^- B} h, \nabla^{^- B} h)$. In the following we drop the superscript in $\nabla^{^\pm B}h$ and write $\nabla h$ if there is no risk of confusion.

For the Laplace operator we have  $\smash{\Delta^{\!^-\!B}= - \Delta^B}$. To see  this we recall that the Laplace operator $\Delta^M$ on a semi-Riemannian manifold $M$ is defined as 
$$\Delta^M f(p)= \sum_{i=1}^n \epsilon_i g_M(\nabla_{e_i} \nabla f|_p, e_i)$$ 
where $p\in M$, $(e_i)_{i=1, \dots, n}$ is an orthonormal basis at $T_pM$ and $\epsilon_i = g_M(e_i, e_i)$. Equivalently, $\Delta^M f$ is the divergence of $\nabla f$.

\end{rem}

Given $\kappa\in \R$ we assume for $f\in C^2(B)$ that 
\begin{align}\label{Eq: f diff ineq}
\nabla^2 f + \kappa f g_{^{\pm}\! B}=\nabla^2 f \pm \kappa f g_{ B}\leq 0.
\end{align}
For instance, if $B=I$ is an interval this differential inequality translates to 
$$ f'' \pm \kappa f\leq 0.$$
Hence, $f$ is either $\kappa$-concave or $(-\kappa)$-concave, depending on whether we consider $I$ or $^-\! I$. 
\smallskip

We will consider measures of the form $f^N\vol_B \otimes \vol_F$ on ${}^\pm B\times_f F$. The volume  that is induced by $\smash{g_{{}^\pm B\times_f F}}$ is of the form $f^n\vol_B\otimes \vol_F$ where $n$ is the dimension of the manifold $F$. This is why we will study now the convexity properties of the function $\smash{\Phi:= \log f^{N-n}}$. For $N\in (n, \infty)$ a straightforward calculation from \eqref{Eq: f diff ineq} for $\Phi= (N-n) \log f$  yields
\begin{align}\label{Eq: Phi}
    \nabla^2 \Phi + \frac{1}{N-n} \dint\Phi \otimes \dint\Phi  +  (N-n)\kappa\, g_{^\pm\! B}\leq 0.
\end{align}
Moreover, given $\eta\in \R$ we consider $H\in C^{2}(F)$ such that 
$$\nabla^2 H^{\frac{1}{N-n}}+ \eta H^\frac{1}{N-n} g_F\leq 0.$$
Then, setting $\Psi:= \log H$ it follows that
\begin{align}\label{Eq: Psi}
    \nabla^2 \Psi + \frac{1}{N-n} \dint\Psi\otimes \dint\Psi +  (N-n) \eta\, g_F\leq 0.
\end{align}
\begin{prop}\label{Pr:G is kappa concave}
Let $\Phi\in C^{\infty}(B)$ and $\Psi\in C^{\infty}(F)$  satisfy \eqref{Eq: Phi} and \eqref{Eq: Psi}.  We set $\Theta= \Phi\circ P_B  + \Psi\circ P_F \in C^{\infty}(B\times F)$ and assume additionally that
\begin{align}\label{ineq:forf} g_{^\pm B}(\nabla f, \nabla f) + \kappa f^2 \leq \eta \mbox{ on } B.\end{align}
Then
\begin{align*}
\nabla^2 \Theta+ \frac{1}{N-n} \dint\Theta\otimes \dint\Theta + ((d+N)-(d+n))\kappa \,g_{\pm B\times_f F} \leq 0.
\end{align*}
In particular $G= e^\Theta$ satisfies $\nabla^2 G^\frac{1}{N-n}+ \kappa\, G^\frac{1}{N-n} g_{\pm B\times_f F}\leq 0$.
\end{prop}
\smallskip

If $B=I$, then \Cref{ineq:forf} reads as $\pm (f')^2 + \kappa f^2 \leq 0$.
\begin{proof}  We  write  $\langle \cdot, \cdot \rangle$ for $\smash{g_{_{^\pm B\times_f F}}}$. We pick vector fields $X\in \mathcal L(B)$ and $V\in \mathcal L(F)$ and set $W=X+V\in\mathfrak{X}(B\times F)$. Using the formulas of the Levi-Civita connection of $B\times_fF$ we compute that 
$$\nabla^2 (\Phi\circ P_B)(W,W)= \langle \nabla_W \nabla( \Phi\circ P_B), W\rangle = \nabla^2(\Phi\circ P_B)(X,X)+ \nabla^2 (\Phi\circ P_B)(V,V)$$
and 
$$\nabla^2 (\Psi\circ P_F)(W,W)= \nabla^2(\Psi\circ P_F)(V,V)+2 \nabla^2(\Psi\circ P_F)(X,V).$$
Moreover we compute
\begin{align}
\label{eq:Phi1}\nabla^2(\Phi\circ P_B)(X,X)&= X(X(\Phi\circ P_B)) - (\nabla_X X)(\Phi\circ P_B)\\\nonumber
&= X(X\Phi)\circ P_B- (\nabla^B_XX)\Phi\circ P_B= \nabla^2 \Phi(X,X)\circ P_B\,,\\
\label{eq:Phi2}\nabla^2(\Phi\circ P_B)(V,V)&= \langle V, V\rangle \langle \nabla \log f, \nabla \Phi\rangle\\\nonumber
&= (N-n)\langle V, V\rangle \langle \nabla\log f, \nabla \log f \rangle\,,
\end{align}
as well as 
\begin{align*}
\nabla^2(\Psi\circ P_F)(V,V)= \nabla^2 \Psi(V,V)\circ P_F,\ \
\nabla^2(\Psi\circ P_F)(X,V)= -\langle X, \nabla\log f\rangle \langle V, \nabla \Psi\rangle\,.
\end{align*}
We also compute 
\begin{align}\label{eq-theta}
\nabla \Theta \otimes \nabla \Theta(W,W)
&=(N-n)^2 \langle \nabla \log f, X\rangle^2 + 2 (N-n)\langle \nabla \log f, X\rangle \langle \nabla \Psi, V\rangle+ \langle \nabla \Psi, V\rangle^2\,.
\end{align}
Hence,
\begin{align*}
&\nabla^2 \Theta(W,W) + \frac{1}{N-n} \nabla \Theta\otimes \nabla \Theta(W,W)\\
&\overset{\eqref{eq-theta}}{=} \nabla^2 \Phi(X,X)\circ P_B + \nabla^2 \Psi(V,V)\circ P_F\\
&\ \ \ \  \ + (N-n)\langle V,V\rangle \langle \nabla \log f, \nabla \log f \rangle \\
&\ \ \ \ \ -2 (...)+  (N-n) \langle \nabla \log f, X\rangle^2+2(...) + \frac{1}{N-n} \nabla \Psi\otimes \nabla \Psi (V,V)\\
&\leq  \nabla^2 \Phi(X,X)\circ P_B + \nabla^2 \Psi(V,V)\circ P_F\\
&\ \ \ \  +(N-n) \langle V, V\rangle \frac{\eta- \kappa f^2}{f^2}+ \frac{1}{N-n} \langle \nabla \Phi, X\rangle^2+ \frac{1}{N-n} \nabla \Psi\otimes \nabla \Psi (V,V)
\\
&=  \nabla^2 \Phi(X,X)\circ P_B + \nabla^2 \Psi(V,V)\circ P_F\\
&\ \ \ \  + (N-n) \big(\eta g_F(V,V)  - \kappa\langle V, V\rangle \big)+ \frac{1}{N-n} \langle \nabla \Phi, X\rangle^2+ \frac{1}{N-n} \nabla \Psi\otimes \nabla \Psi (V,V)\\
&\overset{\eqref{Eq: Phi}, \eqref{Eq: Psi}}{\leq} -(N-n)\kappa \big( g_{^\pm \! B}(X,X) + f^2 g_F(V,V)\big).
\end{align*}
and this finishes the proof. 
\end{proof}
\subsubsection{$1$-dimensional base and fiber}
We now assume that $B= I$ is an interval. 
The interval $I$ is equipped with either one of the two standard semi-Riemannian metrics, that is $(\dint t)^2$ or $-(\dint t)^2$.  In the former case we  write $^+\!I$ for $I$, in the latter case we write $^-\!I$ for $I$ as before. 

We recall that the inequality \eqref{Eq: f diff ineq} for a smooth function $f: {}^\pm I \rightarrow (0, \infty)$ becomes 
\begin{align}\label{inequ:kappa} f'' \pm \kappa f\leq 0\,,\end{align}
where the sign $\pm$ in front of $\kappa$ depends on whether we consider $^+ I$ or ${}^- I$. 


Moreover, we consider $F=(a,b)\subseteq\R$,  equipped with the Riemannian metric $(\dint t)^2$, and let $h\in C^\infty( (a,b),(0,\infty))$ satisfy 
$$\left(h^{\frac{1}{N-1}}\right)'' + \eta h^\frac{1}{N-1}\leq 0.$$
\begin{rem}
Consider the semi-Riemannian warped product $^\pm I\times_f (a,b)$ between $B={}^\pm\! I$ and $F=(a,b)$. If $I$ is equipped with the standard Riemannian structure, i.e., $I= ^+\!\!\!I$, then $^+\! I\times_f (a,b)$ is a Riemannian manifold and the Riemannian metric is $(\dint t)^2 + f^2 (\dint r)^2$. If $I$ is equipped with the semi-Riemannian metric $-(\dint t)^2$, i.e., ${}^-\!I$, then $^-\! I\times_f (a,b)$ is a Lorentzian manifold and the Lorentzian metric is $-(\dint t)^2+ f^2 (\dint r)^2$. 
\end{rem}
\begin{rem}[Signature convention]
We could use a different signature convention, i.e.\ $(\dint t)^2- f^2 (\dint r)^2$ instead of $-(\dint t)^2+f^2 (\dint r)^2$. Then, $\kappa$  should to be replace with $-\kappa$ (i.e.\ $f''+\kappa f\leq 0$ instead of $f''- \kappa f\leq 0$). In the literature there is no canonical signature convention.  For instance, in \cite{AB:08}, \cite{CM:24a} and in \cite{AGKS:23} the signature convention is $(-, +, \dots, +)$, whereas in \cite{McC:20} the signature convention is $(+, -, \dots, -).$
\end{rem}
Applying  \Cref{Pr:G is kappa concave} to $G= f^{N-1} h$ we get the following result:
\begin{cor}\label{cor:fct}
 Let $f,h$ and $G$ be as above. Then $G$ satisfies
\begin{equation}\label{eq:fct} \smash{\nabla^2 G^\frac{1}{N-1}+ \kappa G^\frac{1}{N-1}g_{\pm I\times_f(a,b)}\leq 0\,},
\end{equation} provided $\pm \left(\frac{\dint f}{\dint t}\right)^2 + \kappa f^2 \leq \eta$,  where the sign depends on whether $I$ is equipped with $(\dint t)^2$ or with $-(\dint t)^2$.
\end{cor}
In particular, if $c: [0, 1]\rightarrow{}^-I \times_f (a,b)$ is a timelike, constant speed geodesic, then we compute for $\phi =G \circ c$ that 
$$\left( \phi^\frac{1}{N-1}\right)'' = \nabla^2 G^{\frac{1}{N-1}}(c', c')\leq - \kappa \left(G^{\frac{1}{N-1}}\circ c\right) g_{{}^-I\times_f F}(c', c')= + (L(c))^2\kappa G^{\frac{1}{N-1}}\circ c.$$
Hence, $(\phi^{\frac{1}{N-1}})'' - \kappa L(c)^2 \phi^{\frac{1}{N-1}}\leq 0$, and that is equivalent to 
$$\phi((1-s) t_0 + st_1)^\frac{1}{N-1} \geq \sigma_{-\kappa}^{(1-s)}((t_1-t_0)L(c)) \phi(t_0)^{\frac{1}{N-1}} +  \sigma_{-\kappa}^{(s)}((t_1-t_0)L(c)) \phi(t_1)^{\frac{1}{N-1}} $$
for all $t_0, t_1\in [0,1]$ and $s\in [0,1]$.
\begin{prop}\label{prop:concave} Consider $I$, $f$ and $(a,b)$ as before, as well as the semi-Riemannian warped product $^-\!I\times_f (a,b)$. 
Let  $h$ be a continuous function on $[a,b]$ such that $\smash{(h^\frac{1}{N-1})'' + \eta h^\frac{1}{N-1}\leq 0}$ in the distributional sense. Let $c: [0,1]\rightarrow\,^-\!I\times_f (a,b)$ be a timelike, constant-speed geodesic w.r.t.\ $g_{^-\!I\times_f(a,b)}$, and set $\phi= G\circ c$. Then 
$$\left(\phi^{\frac{1}{N-1}}\right)'' - \kappa L(c)^2 \phi^{\frac{1}{N-1}} \leq 0 \ \ 
\mbox{ in the distributional sense. }$$
\end{prop}
\begin{proof}
We can approximate $h$ uniformly by smooth functions $h_\varepsilon$  such that $(h_\varepsilon^\frac{1}{N-1})''+ \eta  h_\varepsilon^\frac{1}{N-1}\leq 0$. More precisely,  we pick a non-negative mollifier function $\varphi\in C^{\infty}_c((0,1))$ with $\int_0^1 \varphi(t)\,\dint t =1$ and define $\varphi_\varepsilon(t):= \frac{1}{\varepsilon} \varphi\left(\frac{t}{\varepsilon}\right)$.  Then, for $s\in (a+\varepsilon, b-\varepsilon)$, we set
$$g_\varepsilon(s):=h^{\frac{1}{N-1}}*\varphi_\eps(s) = \int_a^b \varphi_\varepsilon(t-s)h^{\frac{1}{N-1}}(t)\,\dint t= \int_{-\varepsilon}^0 \varphi_\varepsilon(-r) h^{\frac{1}{N-1}}(s-r)\,\dint r\,.$$
It follows that $h_\varepsilon =g_\varepsilon^{N-1} \in C^{\infty}((a+\varepsilon, b-\varepsilon))$ and $h_\varepsilon$ satisfies the desired differential inequality. For details we refer to the proof of \cite[Prop.\ 3.6]{Ket:17}.  

Since $h$ is continuous, one can check that $h_\varepsilon$ converges uniformly to $h$ on $(a+\bar \varepsilon, b-\bar \varepsilon)$ as $\varepsilon\downarrow 0$ for all $\bar\varepsilon$ small enough. We can replace $f$ with $f_\alpha=\alpha\cdot f$ such that $-(f_\alpha')^2 + \kappa f_\alpha^2 \leq \eta-\varepsilon$ for a constant $\alpha=\alpha(\varepsilon)$ with $\alpha(\varepsilon)\rightarrow 1$ if $\varepsilon \rightarrow 0$. Then, it follows that $G_{\alpha, \varepsilon}=f^{N-1}_\alpha h_\varepsilon$ satisfies 
\begin{align}\label{ineq:concave}\nabla^2 G_{\alpha, \varepsilon}^\frac{1}{N-1} + \kappa G_{\alpha,\varepsilon}^\frac{1}{N-1} g_{I\times_{f_\alpha} (a,b)} \leq 0 \ \mbox{ on } I\times (a+\bar\varepsilon, b-\bar \varepsilon). \end{align}
The semi-Riemannian metric $-\dint t^2 +f^2_\alpha \dint r^2$ converges smoothly to $-\dint t^2 + f^2 \dint r^2$ on $I\times (a,b)$. Hence, every timelike geodesic $c:[0,1] \rightarrow I\times (a,b)$ w.r.t.\ $-\dint t^2+f^2 \dint r^2$ is the limit of timelike geodesics $c_\alpha$ in $I\times (a+\bar\varepsilon, b-\bar \varepsilon)$ w.r.t.\ $-\dint t^2 + f^2_\alpha \dint r^2$ for $\alpha\rightarrow 1$ and $\bar \varepsilon>0$ sufficiently small.  At the same time we have that $G_{\alpha(\varepsilon), \varepsilon}$ converges uniformly to $\smash{G=f^{\frac{1}{N-1}} h}$ as $\varepsilon\downarrow 0$ on $I\times (a+\bar \varepsilon, b+\bar\varepsilon)$.  

By the previous computations, it follows for $\phi_\varepsilon= G_{\alpha(\varepsilon), \varepsilon}$ and $t_0, t_1\in [0,1]$ that 
$$\phi_\epsilon((1-s) t_0 + st_1)^\frac{1}{N-1} \geq \sigma_{-\kappa}^{(1-s)}((t_1-t_0)L(c_\varepsilon)) \phi_\epsilon(t_0)^{\frac{1}{N-1}} +  \sigma_{-\kappa}^{(s)}((t_1-t_0)L(c_\varepsilon)) \phi_\epsilon(t_1)^{\frac{1}{N-1}} .$$
This inequality is stable under uniform convergence  and hence it is also satisfied for $\phi$. This yields that distributional inequality for $\phi$ by \cite[Prop.\ 3.6]{Ket:17}. 
\end{proof}

\subsection{Generalized cones as metric spaces}\label{subsec-con-ms}
We briefly recall the construction of warped products with one-dimensional base space, a.k.a.\ generalized cones, of metric spaces, c.f.\ \cite{AB:04, AB:98, BBI:01}.

Let $(X, \met)$ be a metric space, let $I\subseteq\R$ be a closed interval and suppose $f:I\rightarrow [0,\infty)$ is continuous. The length of an absolutely continuous  curve $\gamma=(\alpha, \beta): [a,b]\rightarrow I \times X$ is given by 
$$L(\gamma)= \int_a^b \sqrt{ \dot \alpha^2 + (f\circ \alpha)^2 v_\beta^2} \ \dint t$$
where both components $\alpha$ and $\beta$ are absolutely continuous. In particular, the metric speed $v_\beta$ of $\beta$ is well-defined almost everywhere. Moreover, we have
\begin{equation}\label{eq:varlength}
L(\gamma) = \sup \left\{\sum_{i=1}^{\scriptscriptstyle N}\left( |\alpha(t_{i})- \alpha(t_{i-1})|^2 + (f\circ \alpha(t_i))^2 \met(\beta(t_{i-1}), \beta(t_i))^2\right)^{\frac{1}{2}}\right\}
\end{equation}
where $a=t_0\leq t_1 \leq \dots \leq t_N=b$.

The warped product distance on $I \times_f X$ between points $(s,x)$ and $(t, y)$ is defined by the infimum of the length of paths connecting  $(s,x)$ and $(t,y)$. First, this is a semi-distance on $I\times X$ that turns into a distance function on the quotient space $(I\times X)/\sim$. Here $(s,x)\sim (t,y)$ if and only if the warped distance between $(s,x)$ and $(t,y)$ is $0$. We write $\met_{I\times_f X}=\met_Y$ for the warped product metric on $(I\times X)/\sim$ and write $I\times_f X$ for the corresponding metric space that we also call generalized (metric) cone.

\begin{rem} If $f>0$, the topology induced by $\met_Y$ coincides with the product topology. If we allow that $f$ vanishes at certain points this might not be the case. But in the following we will always impose  conditions that  imply that $f^{-1}(\{0\})\subseteq\partial I$. It follows that the topology in a neighbourhood of $[(s,x)]$ for $s\in f^{-1}(\{0\})$ is the same as in a neighbourhood of the tip of the topological cone $([0, \infty)\times X)/\{(0,x)\sim (0,y)\}$. 
\end{rem}
\begin{thm}[Properties of minimizers, {\cite[Thm.\ 3.1]{AB:98}}]
    Assume $(X,\met)$ is a geodesic space, and let $\gamma=(\alpha, \beta)$ be a minimizer of $L$ in $I\times_f X$.
    Then 
    \begin{enumerate}
        \item $\beta$ is a length minimizer in $X$.
        \item $\gamma$ is independent of $X$, except for the total height, i.e., the length of $\beta$. Precisely: for another geodesic metric space $\tilde X$ and a minimizer $\tilde \beta$ in $\tilde X$ with the same length and speed as $\beta$, $(\alpha, \tilde \beta)$ is a minimizer in $I\times_f \tilde X$. This property is called \emph{fiber independence}.
        \item $\alpha$ has speed $\frac{c}{f^2}$ for a constant $c>0$. 
        \item For some parametrization of $\gamma$ proportional to arclength, $\alpha$ satisfies $\frac{1}{2} (\dot \alpha)^2 + \frac{1}{2}\frac{1}{f^2}=E$ for a constant $E>0$. 
        \end{enumerate}
\end{thm}

Two standard examples that illustrate this definition are Euclidean metric cones where $I=[0,\infty)$ and $f=\id$, and spherical suspensions, where $I=[0,\pi]$ and $f=\sin$. The Euclidean cone is given as the quotient space $([0,\infty)\times X)/ \sim$ with the metric $\met_Y$ that has the explicit form
$$\met_Y([s,x], [t,y])= \begin{cases} \sqrt{s^2 + t^2 - 2st \cos \met_X(x,y)} & \met_X(x,y)<\pi\,,\\
s+t& \met_X(x,y)\geq \pi\,. \end{cases}$$
\subsection{Generalized cones as Lorentzian length spaces}\label{subsec-con-lls}
In this section we review the construction of a generalized cone as Lorentzian pre-length space from \cite{AGKS:23}. We slightly generalize their definition by allowing arbitrary intervals and a warping function that might vanishes at the endpoints of the interval as in the metric case. To be precise, let $(X,\met)$ be a metric space, $I$ be an interval and $f: I\rightarrow [0, \infty)$ be continuous such that $f^{-1}(\{0\})\subseteq \{\inf I, \sup I\}$. If $f$ vanishes at $\inf I$ or $\sup I$, we demand that
$$\int_{\inf I}^t \frac{1}{f}=\infty\,,\quad\quad\int_t^{\sup I}\frac{1}{f}=\infty,\quad\quad\forall t\in(\inf I,\sup I)\,.$$
This generalization has been discussed in \cite[Rem.\ 3.32]{AGKS:23}.


To put a Lorentzian structure on ${}^-\!I\times X$ we define causal and timelike curves.
\begin{defi}[Timelike and causal curves]
 Let $Y={}^-\!I\times X$  and let $\gamma\colon J\rightarrow Y$ be an absolutely continuous curve with respect to the product metric on $Y$. Such a curve has components $\gamma=(\alpha,\beta)$, where $\alpha\colon J\rightarrow I$ and $\beta\colon 
J\rightarrow X$ are both absolutely continuous, and the metric derivative of $\beta$, $v_\beta$, exists 
almost everywhere \cite[Thm.\ 2]{AGS:05}. We additionally require that $\alpha$ is strictly monotonous. The curve $\gamma$ is called
\begin{equation}
\begin{cases}
 \text{\emph{timelike}}\\
 \text{\emph{null}}\\
 \text{\emph{causal}}
 \end{cases}
 \text{\quad if \qquad}
 -\dot\alpha^2 + (f\circ \alpha)^2 v_{\beta}^2\quad
 \begin{cases}
  < 0\,,\\
  = 0\,,\\
  \leq 0\,.
 \end{cases}
\end{equation}
almost everywhere. It is called \emph{future/past directed causal} if $\alpha$ is strictly monotonically 
increasing/decreasing, i.e., $\dot\alpha>0$ or $\dot\alpha<0$ almost everywhere.
\end{defi}

Following \cite{AGKS:23} we can define causal relations on $Y$ as follows.
 Let $y, y'\in Y$, then $y$ and $y'$ are \emph{chronologically} (or \emph{timelike}) related, denoted by $y\ll y'$, if there exists a future 
directed timelike curve from $y$ to $y'$. Moreover, $y$ and $y'$ are \emph{causally} related, denoted by $y\leq y'$ 
if there exists a future directed causal curve from $y$ to $y'$ or $y=y'$.

\begin{defi}[Length of a causal curve]
Let $\gamma=(\alpha,\beta)\colon [a,b]\rightarrow Y$ be a causal curve. Its \emph{length} $L(\gamma)$ is defined as 
\begin{equation}
 L(\gamma):= \int_a^b \sqrt{\dot\alpha^2 - (f\circ\alpha)^2 v_\beta^2}\,.
\end{equation}
\end{defi}
In \cite{AGKS:23} the authors prove a variational formula for $L$ that resembles the formula \eqref{eq:varlength} in the case of metric warped products. As consequence they show that $L$ is upper semi-continuous w.r.t.\ pointwise convergence of $\gamma_n$ to $\gamma$. The latter is understood w.r.t.\ the product topology on $I\times X$. As noted before, if $f$ vanishes in $\partial I$, the topology of $Y$ should be the corresponding quotient topology, i.e., the topology induced by $\met_Y$. It is straightforward to see that the upper semi-continuity still holds in this case. 
\begin{prop}[Push-up {\cite[Prop.\ 3.22]{AGKS:23}}]
Every generalized cone ${}^-I\times_f X$ such that $X$ is a length space has the property that $p\ll q$ if
and only if there exists a future directed causal curve from $p$ to $q$ of positive
length, i.e., push-up holds. Moreover, $I^{\pm}(p)$ is open for any $p\in Y$.
\end{prop}
\begin{defi}[Time separation function]
The \emph{time separation function} $\uptau\colon Y\times Y \rightarrow [0,\infty]$ on $Y={}^-I \times_f X$ is defined as
\begin{equation}
 \uptau(y,y'):=\sup\{L(\gamma):\gamma \text{ future directed causal curve from } y \text{ to } y'\},
\end{equation}
if this set is non-empty, and $\uptau(y,y')=0$ otherwise.
\end{defi}
\begin{rem}\label{rem-tau-pro} One can deduce that the time separation function $\uptau$ has the following properties:
 \begin{enumerate}
  \item $\uptau(y,y')=0$ if $y'\not\leq y$.
  \item $\uptau(y,y')>0$ if $y\ll y'$.
\item \cite[Lem.\ 3.21]{AGKS:23} $\uptau$ satisfies the \emph{reverse triangle inequality}. More precisely, let $y_1,y_2,y_3\in Y$ with $y_1\leq y_2\leq y_3$. Then
 \begin{equation}\label{lem-rev-tri}
  \uptau(y_1,y_2) + \uptau(y_2,y_3)\leq \uptau(y_1,y_3)\,.
 \end{equation}\item \cite[Lem.\ 3.25]{AGKS:23}
  Let $(X, \met)$ be a  length space. Then $\tau$ is 
lower semi-continuous with respect to the product metric on $I\times X$. 
It is again straightforward to see that, if $f$ vanished in $\partial I$, $\uptau$ is still lower semi-continuous w.r.t.\ to the quotient topology, i.e., the topology induced by $\met_Y$. 
  \end{enumerate}
\end{rem}
\begin{cor} Let $X$ be a length space and let $Y={}^-I \times_f X$ be the generalized cone. Then $(Y, \ll, \leq, \met_Y, \uptau)$ is a Lorentzian pre-length space. 
\end{cor}

Next we collect various properties of generalized cones that we will need from \cite{AGKS:23}. Among them an essential feature of them is the so-called fiber independence.
\begin{thm}[Fiber independence {\textnormal{\cite[Thm.\ 3.29]{AGKS:23}}}]\label{th:fiber}
 Let $(X,\met)$ be a geodesic length space and let $\gamma=(\alpha,\beta)\colon[0,b]$ $\rightarrow Y = {}^-\! I\times_f X$ be future directed causal and maximal. 
 \begin{enumerate}
  \item The fiber component $\beta$ is a length minimizer in $X$.
  \item\label{thm-structure-of-geod-fib-ind} Fiber independence holds. The base component $\alpha$ depends only on the length of $\beta$. More precisely, let $(X',\met')$ be another geodesic length space, $\beta'$ minimizing in $X'$ with 
$L^{\met'}(\beta')=L^\met(\beta)$ and the same speed as $\beta$, i.e, $v_\beta=v_{\beta'}$. Then $\gamma':=(\alpha,\beta')$ is a 
future directed maximal causal curve in $Y':=I\times_f X'$, which is timelike if $\gamma$ is timelike in $Y$.
  \item If $\gamma$ is timelike, then it has an (absolutely continuous) parametrization with respect to arclength, i.e., 
$-\dot\alpha^2 + (f\circ\alpha)^2 v_\beta^2 = -1$ almost everywhere. 
 \item If $\gamma$ is timelike and parametrized with respect to arclength, then $v_\beta$ is proportional to 
$\frac{1}{(f\circ\alpha)^2}$.
\item If $\gamma$ is timelike, it has an (absolutely continuous) parametrization proportional to arclength such that 
$-\dot\alpha^2 + \frac{1}{(f\circ\alpha)^2}$ is constant.
 \end{enumerate}
\end{thm}
\begin{rem}[Consequences of fiber independence]\label{subsec:con_fiber}
The fiber independence has the following consequence. Let $(X,\met)$ be a geodesic length space and let $\gamma=(\alpha,\beta)\colon[0,b]$ $\rightarrow {}^-\! I\times_f X=
Y$ be future directed, causal and maximal.  We can define the generalized cone $Y_\beta$ of $I$, $X_\beta$ and $f$ where $X_\beta=\beta([0,b])$.
The subset $X_\beta$ in $X$ is isometric
to $[0,L(\beta)]$ where $L(\beta)$ is the length of $\beta$. We obtain an isometry by the $\met$-arc-length parametrization $\bar \beta$ of $\beta$.  The fiber independence implies that $Y_\beta=\I\times_f X_\beta\simeq \I\times_f [0, L(\beta)]$ isometrically embeds into $Y$ via the map $$F_\beta: \I\times_f [0,L(\beta)] \rightarrow \I\times_f X, \quad F(r,t)= (r,  \bar\beta(t))\,. $$
\end{rem}
Here we use the following: Consider two Lorentzian pre-length spaces $(M_i, d_i, \tau_i, \ll_i, \leq_i)$. We call a map $\iota: M_0\rightarrow M_1$ a Lorentzian isometric embedding if for all $x,y\in M_0$
\begin{enumerate}
\item $x\leq_0 y$ if and only if $ \iota(x) \leq_1 \iota(y)$, 
\item $\tau_0(x, y)= \tau_1(\iota(x), \iota(y))$. 
\end{enumerate}

\begin{thm}[Properties of generalized cones]\label{Pr: various properties}
    Let $Y=\!{}^-\!I\times_fX$ be a generalized cone as described above and $(X,\met)$ a metric space.
    \begin{itemize}
        \item[(i)] If $X$ is a geodesic length space 
        every maximizing causal curve $\gamma = (\al,\beta): [-b,b]\to Y$ has a causal character,
        i.e., $\gamma$ is either timelike or null \textnormal{\cite[Cor.\ 3.30]{AGKS:23}}.
        \item[(ii)] If $X$ is locally compact  the length $L$ coincides with the $\uptau$-length $L^{\uptau}$ on future-directed causal curves  \textnormal{\cite[Prop.\ 4.7]{AGKS:23}}.
        \item[(iii)] If $X$ is a geodesic and locally compact length space then $Y$ is a strongly causal and regular Lorentzian length space \textnormal{\cite[Cor.\ 4.9]{AGKS:23}}.
        \item[(iv)] If $X$ is a locally compact, complete length space then $Y$ is globally hyperbolic \textnormal{\cite[Cor.\ 4.11]{AGKS:23}}.
        \item[(v)] If $X$ is geodesic, then $Y$ is geodesic too. Furthermore, any two timelike related points can be connected by a timelike geodesic \textnormal{\cite[Cor.\ 4.11]{AGKS:23}}.
    \end{itemize}
\end{thm}

\section{Two-dimensional model spaces}\label{sec-2d-mod-spa}
First, we recall the formulas of the Ricci tensor of warped products of semi-Riemannian manifolds. Let $B, F$ and $f$ be as in Section \ref{sec:wpsemiriem}, especially $\dim_B=d$ and $\dim_F=n$, and let $X,Y\in\mathcal{L}(B)$ and $V,W\in\mathcal{L}(F)$.
The following formulas for the Ricci curvature of $^\pm\!B\times_f F$ hold, cf.\ \cite[Prop.\ 7.43]{ONe:83}:
\begin{align}
\label{eq-ric1}\Ric_{^\pm\!B\times_f F}(X,Y)&= \Ric_B(X,Y)- n \frac{\nabla^2 f(X,Y)}{f}.\\
\label{eq-ric2}\Ric_{^\pm\! B\times_f F}(X,V)&= 0.\\
\label{eq-ric3}\Ric_{\pm\! B\times_f F}(V,W)&= \Ric_F(V,W) - \left( \frac{\Delta^{^\pm\!B} f}{f} + (n-1) \frac{g_{^\pm B}(\nabla f, \nabla f)}{f^2} \right) \langle V,W\rangle.
\end{align}


In this section we assume $B=I$, where $I\subseteq\R$ is an interval, and $F=(a,b)$. We have the following corollary. 

\begin{cor}\label{Co: 2-dim smooth warped product} Let $I\subseteq\R$, $[a,b]\subseteq\R$ and $f:I\rightarrow [0,\infty)$ be smooth. We assume that $\partial I= f^{-1}(\{0\})$.
\begin{enumerate}
\item[(+)]
 $I\times_f [a,b]$ satisfies 
$
\CD(\kappa , 2) 
$
if and only if $f''+ \kappa f \leq 0$ and
\begin{enumerate} \item $\partial I=\emptyset$, or
\item if $\partial I \neq \emptyset$, then $b-a\leq \frac{\pi}{\sqrt \eta}$ for  $\eta=\sup_{\partial I}(f')^2$ . 
\end{enumerate}
    \item[(--)]
$\I\times_f [a,b]$ satisfies 
$
\tCD_p(-\kappa, 2) 
$
if and only if 
$
f''- \kappa f\leq 0$.
\end{enumerate}
\end{cor} \begin{proof} {\bf {\it  (--)}} Let us first check the {\it if} implication for the case $I=\I$.

By equations \eqref{eq-ric1}, \eqref{eq-ric2}, \eqref{eq-ric3} for the Ricci curvature of $\I\times_f (a,b)$  for a timelike vector $v=a\frac{\partial}{\partial t}+b\frac{\partial}{\partial r}$, i.e.\ $g_{^- I\times_f [a,b]}(v,v)<0$, we have
\begin{align*}
    \Ric_{^- I\times_f(a,b)}\left(v,v\right)&= -a^2\frac{f''}{f}\left\vert\frac{\partial}{\partial t}\right\vert^2_{^{^+}\! I}{+}\, b^2\frac{f''}{f} f^2\left\vert\frac{\partial}{\partial r}\right\vert^2
    = \frac{f''}{f}\underbrace{\left[-a^2\left\vert\frac{\partial}{\partial t}\right\vert_{^{^+}\!I}^2 + b^2\,f^2 \left\vert\frac{\partial}{\partial r}\right\vert^2\right]}_{<0\,\text{(}v\text{ is timelike)}}\\ &\geq \kappa\left[-a^2\left\vert\frac{\partial}{\partial t}\right\vert^2_{{}^{^+}\!I} + b^2f^2\left\vert\frac{\partial}{\partial r}\right\vert^2\right]= \kappa\cdot g_{^-I\times_f(a,b)}(v,v)
    \end{align*}
where $\left| \frac{\partial}{\partial t}\right|_{^+ I}^2= g_{^+I}\left( \frac{\partial}{\partial t}, \frac{\partial}{\partial t}\right).$
Hence $\Ric_{^- I\times_f (a,b)}\geq \kappa g_{^-I\times_f(a,b)} $ in timelike directions whenever $f>0$. Then \Cref{thm:smoothcd} yields the condition $\tCD_p(-\kappa,2)$.

Note that $\I\times_f [a,b]$ is not a smooth space if $\partial I=f^{-1}(\{0\})\neq \emptyset$. However it is clear that one can prove the inequality for the $\tCD_p$ condition on $\I\times_f (a,b)$ for any admissible optimal transport, in the same way as for \Cref{thm:smoothcd}, and that the singular points  $\partial I\times [a,b]$, in case $\partial I$ is non-empty, as well as the boundary $I\times( \{a\}\cup \{b\})$ do not affect the validity of the $\tCD_p$ condition on $\I\times_f[a,b]$. 

For the \emph{only if} direction we refer to (i) of \Cref{Le:from Y to X 2 dim and h smooth} below.
\smallskip

{\it (+)} By \cite{LySt:23} the $\CD(\kappa, 2)$ for $I\times_f [a,b]$ holds if and only if $I\times_f [a,b]$ has Alexandrov curvature bounded from below by $\kappa$. Then the claim follows from \cite{AB:04} and \cite{AB:16} where the corresponding statement was shown for finite dimensional Alexandrov spaces.
%
%
%
\end{proof}
Let $h\in C^{\infty}((a,b))$ satisfy $\smash{\big( h^\frac{1}{N-1}\big)'' + \eta h^\frac{1}{N-1}\leq 0}$, where $\eta\in\R$ and $N>1$, and  define \begin{center}$\Theta= \log G$,\end{center}where {$G= f^{N-1} h:I\times(a,b)\rightarrow\R$.  For $f\in C^\infty(I)$ we still assume \begin{align}\label{ass:forf} f'' \pm \kappa f\leq 0 \ \mbox{ and } \ \pm (f')^2 + \kappa f^2 \leq \eta.\end{align} 

Thanks to \Cref{cor:fct}} (the one-dimensional case of \Cref{Pr:G is kappa concave}) it follows that the $(N+1)$-Bakry--\'Emery Ricci tensor of $({}^\pm I\times_f (a,b), e^{\Theta}\vol_{{}^\pm I\times_f (a,b)})$ has a lower bound
\begin{align}\label{Eq: ricbxff h smooth}
\Ric^{\Theta, N+1}_{{}^\pm I\times_f (a,b)}= \Ric_{{}^\pm I\times_f (a,b)} -\nabla^2 \Theta - \frac{1}{(N+1)-2}\,\dint\Theta \otimes \dint\Theta\geq N\kappa g_{^\pm I \times_f (a,b)}
\end{align}
whenever $f>0$, i.e., the generalized cone and the reference measure $G\vol_{{}^\pm I \times_f (a,b)}$ is smooth. 
\smallskip

In fact, under this assumption the space $({}^\pm I\times_f [a,b], e^{\Theta}\vol_{{}^\pm I\times_f (a,b)})$ satisfies a curvature-dimension condition. We will show this in the next proposition, where  we will also lower the regularity of $h$ from smooth to continuous.
\begin{thm}\label{Th:2 dimensional non-smooth case} Let $f$ as before, $p\in (0,1)$, $N>1$ and assume $\partial I= f^{-1}(\{0\})$.
If $h\in C^0([a,b], [0, \infty))$ satisfies $${\left(h^\frac{1}{N-1}\right)'' + \eta h^\frac{1}{N-1}\leq 0\,,}$$ in the distributional sense, then the generalized cone $({}^\pm I\times_f [a,b], {G}\,\vol_{^\pm I\times_f [a,b]})$ satisfies\\$\CD(\kappa N, N+~1)$ or $\tCD_p({{-}}\kappa N,N+1)$, respectively. 
\end{thm}
\begin{proof} We only prove the statement in the case $I=\I$.

If $h$ is smooth, then  $G$ satisfies  \eqref{eq:fct} in Corollary \ref{cor:fct}. 
Then \Cref{Co: 2-dim smooth warped product} and the second part in  Theorem \ref{thm:smoothcd} yield the claim. 

If $h$ is not smooth, the conclusion follows from \Cref{Co: 2-dim smooth warped product}, \Cref{prop:concave}, and  again the second part in Theorem \ref{thm:smoothcd}.
\end{proof}
{In the following we will prove  a converse of \Cref{Th:2 dimensional non-smooth case} above, i.e., the assumptions on $h$ and $f$ are in general not only sufficient for $(^\pm I \times_f [a,b], G\vol_{I\times_f [a,b]})$ to satisfy a curvature-dimension condition but also necessary.  

We start with the case where $h$ is smooth in \Cref{Le:from Y to X 2 dim and h smooth} below and then decrease the regularity of $h$ to $L^\infty_\loc((a,b))$ in \Cref{Pr:from Y to X 2 dim and h semi-concave}.} 

\begin{lem}\label{Le:from Y to X 2 dim and h smooth}
Assume that $h:[a,b]\rightarrow [0, \infty)$  and $f:{}^\pm\! I \rightarrow [0, \infty)$ are smooth, $\partial I= f^{-1}(\{0\})$ (if $\partial I\neq \emptyset$). Moreover $\left(^\pm \!I\times_f[a,b], f(t)^N h(r)\,\dint t\otimes\dint r\right)$
satisfies \begin{enumerate}
    \item[(+)] $\CD(N\kappa, N+1)$
    \item[(--)] $\tCD_p({{-}}\kappa N, N+1)$ for $0<p<1$.
    \end{enumerate}
If we consider $I={^+}\!I$,  we assume  $\partial I\neq \emptyset$. 
Then it follows that 
\begin{enumerate}
    \item $f''\pm \kappa f\leq 0$, 
    \item \begin{enumerate}
        \item[(+)]
    $([a,b], h(r)\,\dint r)$ satisfies $\CD((N-1)\eta, N)$ with $\eta= \sup_I\left\{ (f')^2 + \kappa f^2\right\}$.
    \item[(--)] 
    $([a,b], h(r)\,\dint r)$ satisfies $\CD((N-1)\eta, N)$ with $\eta= \sup_I\left\{ -(f')^2 + f'' f\right\}$.
    \end{enumerate}
\end{enumerate}
\end{lem}

\begin{proof}[{Proof of Lemma \ref{Le:from Y to X 2 dim and h smooth}}] {\bf 1.} We first consider the case of $I={^-\!I}$.

Since the generalized cone under consideration is smooth the $(N+1)$-Ricci tensor is well-defined and is given explicitly by
\begin{equation}\label{eq:2DRic}
\Ric^{\Theta,N+1}_{^-\!I\times^N_f [a,b]}:=\Ric_{^-\!I\times_f [a,b]} - (N-1)\frac{\nabla^2 (f^{N-1} h)^{\frac{1}{N-1}}}{(f^{N-1} h)^{\frac{1}{N-1}}}.
\end{equation}
In combination with the formulas \eqref{eq-ric1}, \eqref{eq-ric2}, \eqref{eq-ric3}   for $\Ric_{{}^-\!I \times_f [a,b]}$ and the Hessian $\nabla^2 G$ with $G=f^{N-1} h$, this can equivalently be written as follows
\begin{align}
\label{eq-ric-theta-1}\Ric^{\Theta,N+1}_{^-\!I\times^N_f [a,b]}(\partial_t,\partial_t)&= \Ric_I(\partial_t,\partial_t)- N\frac{f''}{f}\langle\partial_t,\partial_t\rangle_I= - N \frac{f''}{f}.\\
\label{eq-ric-theta-2}\Ric^{\Theta,N+1}_{^-\!I\times^N_f [a,b]}(\partial_t,\partial_r)&= 0.\\
\nonumber\Ric^{\Theta,N+1}_{^-\!I\times^N_f [a,b]}(\partial_r,\partial_r)&= \Ric_{[a,b]}(\partial_r,\partial_r) - (N-1)\smash{\frac{\left(h^\frac{1}{N-1}\right)''}{h^{\frac{1}{N-1}}}}\langle \partial_r, \partial_r\rangle_{[a,b]}\\
&\ \ \ \ \ \ \ \ \ \ \ \ \ \ \ \ \ \ \ \ \ \  - \left( \frac{\smash{\Delta^{\!^-\!I}} f}{f} + (N-1) \frac{- (f')^2}{f^2} \right) \langle \partial_r,\partial_r\rangle_{I\times_f [a,b]}\\
 &=\label{eq-ric-theta-3} -(N-1)\frac{\left(h^\frac{1}{N-1}\right)''}{h^{\frac{1}{N-1}}}-\left(\frac{- f''}{f}+ (N-1)\frac{-(f')^2}{f^2}\right) f^2.
\end{align}
By {\Cref{thm:smoothcd}} the condition $\textsf{TCD}_p({{-}}\kappa N,N+1)$ on $\smash{{}^-I\times^N_f [a,b]}$ is equivalent to the bound $\smash{\Ric_{^-\!I\times^N_f [a,b]}}\geq \kappa N g_{{}^-I\times_f^{N} [a,b]}$ in timelike directions (keeping in mind our signature convention). Hence, by \eqref{eq-ric-theta-1} it follows that $\smash{-N\frac{f''}{f}\geq   -N\kappa}$, i.e.,  $f''- \kappa f\leq 0$. 
\medskip

We consider a timelike tangent vector  of the form $a \frac{\partial}{\partial t} + \frac{\partial}{\partial r}=:v\in T_{(t,r)}(I\times (a,b))$ with $t\notin f^{-1}(\{0\})$. It being timelike implies $-a^2 + f^2(t) <0$, or equivalently $f(t)< |a|$. Let us assume $a>0$.

By \eqref{eq-ric-theta-1}, \eqref{eq-ric-theta-2}, \eqref{eq-ric-theta-3} and  the curvature assumption $\smash{\Ric_{^-\!I\times^N_f [a,b]}}\geq \kappa N g_{{}^-I\times_f^{N} [a,b]}$ in timelike directions we have that
\begin{align*}
    -(N-1)\frac{\left(h^\frac{1}{N-1}\right)''\!(r)}{h^{\frac{1}{N-1}}(r)}- \left(-\frac{ f''(t)}{f(t)}+ (N-1)\frac{-(f'(t))^2}{f^2(t)}\right) f^2(t) -a^2N\frac{f''(t)}{f(t)}\geq \kappa N \left(-a^2 + f^2(t)\right)
    \end{align*}
    Letting $a\downarrow f(t)$  in the last inequality yields
    \begin{align*}
        -(N-1)\frac{\left(h^\frac{1}{N-1}\right)''(r)}{h^{\frac{1}{N-1}}(r)}+{f(t)\cdot f''(t)}+ (N-1){(f'(t))^2} - N{f(t)\cdot f''(t)}\geq 0.
    \end{align*}
    Hence 
    \begin{align*}
        (N-1)\frac{\left(h^\frac{1}{N-1}\right)''(r)}{h^{\frac{1}{N-1}}(r)}\leq  (N-1)\left({(f'(t))^2} - {f(t)\cdot f''(t)}\right).
    \end{align*}
Since $r\in (a,b)$ and $t\in I\backslash \partial I$ are arbitrary, it follows that
$$(N-1)\frac{\left(h^\frac{1}{N-1}\right)''}{h^{\frac{1}{N-1}}}\leq -(N-1) \sup_{ I}\{-{(f')^2} +{f f''}\}.$$
\noindent
{\bf 2.} Now we consider the case $I=  {^+\!I}$ where we assume $\partial I\neq \emptyset.$
Similarly as before,  we obtain
%
\begin{align*}
    -(N-1)\frac{\left(h^\frac{1}{N-1}\right)''}{h^{\frac{1}{N-1}}}\geq (N-1)((f')^2+\kappa f^2)
    +f{( f''+\kappa f)}.
\end{align*} 
Evaluating at points $(t,r)$ with $t\in \partial I=f^{-1}(\{0\})$ we get  that 
$$
  (N-1)\frac{\left(h^\frac{1}{N-1}\right)''}{h^{\frac{1}{N-1}}}\leq -(N-1) \sup_{\partial I} (f')^2.
$$
Finally, by Lemma 3.1 in \cite{AB:04} the claim follows. 
\end{proof}

\begin{rem}
The curvature bound in (ii) (--) is given by $(N-1)\eta$ where $\eta= \sup_I\{-(f')^2 + f'' f\}$. This is weaker than for the positive signature case since we  only use a Bakry--\'Emery Ricci curvature bound in timelike directions. If we assume a Bakry--\'Emery Ricci curvature lower bound of the form $\Ric_{^-\!I\times^N_f [a,b]}\geq \kappa N g_{^-\!I\times^N_f [a,b]}$ in space directions, we obtain the following estimate: 
By (iii) and a Bakry--\'Emery Ricci lower bound in space-like directions it follows
\begin{align}
    -(N-1)\frac{\left(h^\frac{1}{N-1}\right)''}{h^{\frac{1}{N-1}}}&\geq N\kappa f^2 + \left(-\frac{ f''}{f}+ (N-1)\frac{-(f')^2}{f^2}\right) f^2\nonumber\\
    &=(N-1)(-(f')^2+\kappa f^2)
    +f\underbrace{(- f''+\kappa f)}_{\geq 0 }\geq (N-1)(-(f')^2+\kappa f^2).
\end{align} 
On the other hand, if $f$  satisfies $f''-\kappa f=0$, then 
    $([a,b], h(r)\,\dint r)$ satisfies $\CD((N-1)\eta, N)$ with $\eta= \sup_I\left\{ -(f')^2 + \kappa f^2\right\}$.
 In this case $\eta=\sup_I \{-(f')^2+\kappa f^2\}= -(f')^2 + \kappa f^2$.
\end{rem}
We now want to extend \Cref{Le:from Y to X 2 dim and h smooth} to the case where $h$ is just continuous.
For this we introduce the following regularisations of $h$ (cf.\ \cite{CM:17b}).
\begin{defi}[Power-like convolutions]\label{Le:power-convolutions}
   Let $h:[a,b]\rightarrow [0,\infty)$ be a continuous function and $\phi\geq 0$ a standard mollifier as in the proof of Proposition \ref{prop:concave}. Define the following ``power-like" convolutions $h^\eps$ as
   \begin{align}\label{Eq:power-like convolutions}
       h^\eps(r):=\left[h^{\frac{1}{N+1}}*\phi_\eps(r)\right]^{N+1} =\left[\int_\R h(r-y)^{\frac{1}{N+1}}\phi_\eps(y)\,\dint y\right]^{N+1}.
   \end{align}
   \end{defi}
As in \cite[Lem.\ 6.2]{CM:17b} these convolutions have the following properties:
 \begin{itemize}
       \item [(i)] $h^\eps$ is a non-negative smooth function on $[a+\eps, b-\eps]$.
       \item[(ii)] $h^\eps\rightarrow h$ uniformly on compact subsets of $(a,b)$ as $\eps\rightarrow0$.
   \end{itemize}

In the following we denote by $Y^\eps$ the  generalized cone given by $\smash{{^\pm}\!I\times_f[a+\eps,b-\eps]}$  endowed  with the volume measure $\n^\eps(\dint t,\dint r):=f(t)^{{N-1}}h^\eps(r)\,\dint t\otimes\dint r$, where $h:[a,b]\rightarrow [0,\infty)$ is continuous. Recall also  the function $\Phi(t,r):=\log(f(t)^{{N-1}}h(r))$ from \Cref{Co: 2-dim smooth warped product} and let $\Phi^\eps$ be analogously defined with $h^\eps$ in place of $h$.
\begin{lem}[Curvature-dimension bounds for $Y^\eps$]\label{Le: Yeps}
    Assume that the $N$-generalized cone $Y$ satisfies $\tCD_p(\kappa N,N+1)$ or $\CD(\kappa N,N+1)$, respectively. Then also $Y^\eps$ does.
\end{lem}
\begin{proof} We only present the proof for the Lorentzian case, i.e., $Y^\varepsilon={^-}\!I\times^{N}_f (a+\eps,b-\eps)$, because the Riemannian case works in an analogous manner (in fact, even simpler). 

Considering $Y^\varepsilon$, the underlying Lorentzian manifold ${^-}\!I\times_f (a+\eps,b-\eps)$  is smooth and $Y^{\eps'}\subseteq Y^\eps$ if $\eps\in (0, \eps')$.  The Lorentzian reference measure is $f(t)\,\dint t \otimes\dint r$. Given $(\mu_0,\mu_1)\in\mathscr{P}^\mathrm{ac}_{\textnormal{c}}(Y)$ $p$-dualizable by an optimal geodesic plan $\boldsymbol{\pi}\in\mathrm{OptTGeo}(Y)$, denote the $\ell_p$-geodesic between $\mu_0$ and $\mu_1$ by $(\mu_s)_{s\in[0,1]}$. We assume $\spt \mu_s\subseteq I\times_f (a+\eps, b-\eps)$. Then by \cite{McC:20} for $\mu_s\ll\n$ with $\mu_s= \zeta_s \n$ for $\boldsymbol{\pi}$-a.e.\ $\gamma=(\alpha,\beta)$ it holds that
\begin{align*}
    \zeta_s(\gamma_s)=f(\alpha_s)^{-(N-1)}h(\beta_s)^{-1}\rho_s(\alpha_s,\beta_s)
\end{align*}
where $\rho_s(T_s(x))=\det DT_s(x)^{-1} \rho_0(x)$ ($s\in [0,1]$), is the density of $\mu_s$ w.r.t.\ $f(t)\,\dint t\otimes \dint r$, and $T_s$ is the optimal transport map between $\mu_0$ and $\mu_s$, i.e., $(T_s)_\sharp\mu_0=\mu_s$, which exists due to \cite{McC:20}.

Now the key observation is that the $\ell_p$-geodesic between $\mu_0$ and $\mu_1$  in $Y^\eps$ is  $\mu_s$, independently of $\varepsilon$. Replacing $h$ with $h^\eps$ only changes the density {$\zeta_s$}, since the metric structure of $Y^\eps$ is the same as the one of $Y$. More precisely, the density $\zeta^\eps_s$ of $\mu_s$ w.r.t.\ $\n^\eps$ is
\begin{align}
    \zeta_s^\eps(\gamma_s)=f(\alpha_s)^{-(N-1)}h^\eps(\beta_s)^{-1}\rho_s(\alpha_s,\beta_s).
\end{align}
Then we are left to prove that $Y^\eps$ is $\tCD_p(\kappa N,N+1)$.

Recall that for timelike essentially nonbranching spaces, due to \cite[Thm.\ 3.35, Rem.\ 3.36 and Thm.\ 3.41]{Bra:23b}, for every $\n$-absolutely continuous pair $(\mu_0,\mu_1)=(\zeta_0\n,\zeta_1\n)$ timelike $p$-dualized by $\boldsymbol{\pi}$,
and denoted by $\mu_s=\zeta_s\n$ their optimal $\ell_p$-geodesic, $\tCD_p(\kappa N,N+1)$ is equivalent for $\boldsymbol{\pi}$-a.e.\ $\gamma$ to
$$\zeta_s(\gamma_s)^{-\frac{1}{N+1}}\geq\tau_{\kappa N,N+1}^{(1-s)}(\uptau(\gamma_0,\gamma_1))\zeta_0(\gamma_0)^{-\frac{1}{N+1}}+\tau_{\kappa N,N+1}^{(s)}(\uptau(\gamma_0,\gamma_1))\zeta_1(\gamma_1)^{-\frac{1}{N+1}}.$$
From our assumption
that $Y$ is $\tCD_p(\kappa N,N+1)$ we get for $\boldsymbol{\pi}$-a.e.\ $\gamma_s=(\alpha_s,\beta_s)$
\begin{align*}
    \bigg[h^\eps(\beta_s)^{-1}&f(\alpha_s)^{-(N-1)}\rho_s(\alpha_s,\beta_s)\bigg]^{-\frac{1}{N+1}}\\
    &=\left[\left(\int_\R h(\beta_s-y)^{\frac{1}{N+1}}\phi_\eps(y)\,\dint y\right)^{-(N+1)}f(\alpha_s)^{-(N-1)}\rho_s(\alpha_s,\beta_s)\right]^{-\frac{1}{N+1}}\\
    &=\left[\left(\int_\R\left(h(\beta_s-y)^{-1}f(\alpha_s)^{-(N-1)}\rho_s(\alpha_s,\beta_s)\right)^{-\frac{1}{N+1}}\phi_\eps(y)\,\dint y\right)^{-(N+1)}\right]^{-\frac{1}{N+1}}\\
    &\geq\int_\R\Big[\tau_{\kappa N,N+1}^{(1-s)}(\uptau(\gamma_0,\gamma_1))\left(h(\beta_0-y)^{-1}f(\alpha_0)^{-(N-1)}\rho_0(\alpha_0,\beta_0)\right)^{-\frac{1}{N+1}}\\
    &\hspace{2cm}+\tau_{\kappa N,N+1}^{(s)}(\uptau(\gamma_0,\gamma_1))\left(h(\beta_1-y)^{-1}f(\alpha_1)^{-(N-1)}\rho_1(\alpha_1,\beta_1)\right)^{-\frac{1}{N+1}}\Big]\phi_\eps(y)\,\dint y\\
    &=\tau_{\kappa N,N+1}^{(1-s)}(\uptau(\gamma_0,\gamma_1))f(\alpha_0)^\frac{N-1}{N+1}\rho_0(\alpha_0,\beta_0)^{-\frac{1}{N+1}}\int_\R h(\beta_0-y)^{\frac{1}{N+1}}\phi_\eps(y)\,\dint y\\
    &\hspace{2cm}+\tau_{\kappa N,N+1}^{(s)}(\uptau(\gamma_0,\gamma_1))f(\alpha_1)^\frac{N-1}{N+1}\rho_1(\alpha_1,\beta_1)^{-\frac{1}{N+1}}\int_\R h(\beta_1-y)^{\frac{1}{N+1}}\phi_\eps(y)\,\dint y\\
    &=\tau_{\kappa N,N+1}^{(1-s)}(\uptau(\gamma_0,\gamma_1))\left(h^\eps(\beta_0)^{-1}f(\alpha_0)^{-(N-1)}\rho_0(\alpha_0,\beta_0)\right)^{-\frac{1}{N+1}}\\
    &\hspace{2cm}+\tau_{\kappa N,N+1}^{(s)}(\uptau(\gamma_0,\gamma_1))\left(h^\eps(\beta_1)^{-1}f(\alpha_1)^{-(N-1)}\rho_1(\alpha_1,\beta_1)\right)^{-\frac{1}{N+1}},
\end{align*}
which shows the displacement convexity for  $(\mu^\eps_s)_{s\in[0,1]}$ in $Y^\eps$. Since $\mu_0, \mu_1\in \mathscr{P}^{\mathrm{ac}}_{\textnormal{c}}(Y)$ with $\spt\mu_s$ in $Y^\eps$ are arbitrary, $\tCD_p(\kappa N, N+1)$ for $Y^\eps$ follows. 

If we instead assume that $Y$ satisfies $\CD(\kappa N,N+1)$ in place of its timelike counterpart, up to replacing $\uptau(\gamma_0,\gamma_1)$ with $\met(\gamma_0,\gamma_1)$ and not assuming timelike $p$-dualizability of $(\mu_0,\mu_1)\in\mathscr{P}^\mathrm{ac}_{\textnormal{c}}(X)$ because of the absence of causality, we get that $Y^\eps$ satisfies $\CD(\kappa N,N+1)$ too.
\end{proof}

\begin{cor}[2-dimensional case, from $Y$ to $(a,b)$ with {$h\in C^0([a,b])$}]\label{Pr:from Y to X 2 dim and h semi-concave}
     We assume that $h:~[a,b]\rightarrow [0, \infty)$ is in $C^0([a,b])$, $f:{}^\pm\! I \rightarrow [0, \infty)$ is smooth, and 
     $\left(^\pm \!I\times_f[a,b], f(t)^N h(r)\,\dint t\otimes\dint r\right)$
     satisfies 
     \begin{enumerate}
         \item[(+)] $\CD(\kappa N, N+1)$ \item[(--)] $\tCD_p(-\kappa N, N+1)$, $p\in (0,1)$.
         \end{enumerate} 
         If we consider $I={^+}\!I$, then we assume additionally $\partial I\neq \emptyset$. Then it still follows that 
\begin{enumerate}
    \item $f''\pm \kappa f\leq 0$, 
    \item \begin{enumerate}
        \item[(+)]
    $([a,b], h(r)\,\dint r)$ satisfies $\CD((N-1)\eta, N)$ with $\eta= \sup_I\left\{ (f')^2 + \kappa f^2\right\}$.
    \item[(--)] 
    $([a,b], h(r)\,\dint r)$ satisfies $\CD((N-1)\eta, N)$ with $\eta= \sup_I\left\{ -(f')^2 + f'' f\right\}$.
    \end{enumerate}
\end{enumerate}
    In particular $h$ is semi-concave on $[a,b]$.
\end{cor}
\begin{proof}
    From \Cref{Le: Yeps} we know that $Y^\eps$ satisfies $\CD(\kappa N,N+1)$, respectively $\tCD_p(\kappa N,N+1)$. We can consider the two cases simultaneously. 
    
    By construction $Y^\eps$ is smooth, so we can apply Lemma \ref{Le:from Y to X 2 dim and h smooth} directly. 
    More specifically, from  \Cref{Le:from Y to X 2 dim and h smooth}, formula (i) there, we  get
    $$f''\pm\kappa f\leq0$$
    which does not depend on $\varepsilon$. 

    Also by Lemma \ref{Le:from Y to X 2 dim and h smooth}
    the densities $h^\eps$ are $\CD(\eta(N-1),N)$-densities {on $(a+\eps, b-\eps)$}, i.e.,
    \begin{align}\label{aboveabove}
    \left(\left(h^\eps\right)^{\frac{1}{N-1}}\right)''+\eta\left(h^\eps\right)^\frac{1}{N-1}\leq0
    \end{align} 
    where $\eta$ is given as in Lemma \ref{Le:from Y to X 2 dim and h smooth} depending on the signature of ${}^\pm I$.

In particular, thanks to Remark \ref{Re:properties of CD(K,N)-densities and logarithmic convolutions},(ii), $h^\eps$ is semi-concave. The remaining claims follow from \cite[Lem.\ 6.2]{CM:17b}.
\end{proof}
We want to further generalize the previous result. First we prove the following lemma. 
\begin{lem}\label{lem:2Dmcp}
     Let  $([a,b], \m)$ be a metric measure space where $\m$ is a Radon measure with $\spt \m=~[a,b]$ and assume $f:{}^\pm\! I \rightarrow [0, \infty)$ is smooth with $\partial I = f^{-1}(\{0\})$. Assume that 
     $$\left(^\pm \!I\times_f[a,b], f(t)^N \dint t\otimes \m\right)$$
     satisfies $\mcp(\kappa N, N+1)$ or $\tmcp(N\kappa, N+1)$, respectively.  Then it  follows that $\m\ll\mathcal{L}^1$ with
  $\m= h(r)\,\dint r$, where $h \in C^0([a,b])$.
\end{lem}
We note that in the positive signature case  \cite[Prop.\ 3.1]{Han:20}  already gives the following: since  $I\times_f [a,b]$  equipped with the measure $f(t)^N \dint t \otimes \m$  satisfies a measure contraction property, it follows from \cite[Prop.\ 3.1]{Han:20} that there exists $\Phi\in L^\infty_{\loc}(I\times_f (a,b))$ such that 
$$f(t)^N\dint t \otimes \m= \Phi(t,r) f(t)\,\dint t \otimes \dint r. $$
In particular $\m= h(r)\,\dint r$ for $\smash{h\in L^\infty_{\loc}((a,b))}$.
\begin{proof} 
We only prove the lemma for  the negative signature case. The positive signature case works along the same lines.


We pick any point $(t,r)=p\in\I\times_f [a,b]$ and consider {the function} $u=\uptau(p,\cdot)$ on $I^+(p)$, the timelike future of $p$. The spacetime $\I\times_f [a,b]$ can be endowed with the measure $f(t)^N\,\dint t \otimes \m$ or with the measure $f(t)\,\dint t\otimes\dint r$, that is the Lorentzian volume of $^-I\times_f[a,b]$. In both cases  a measure contraction property is satisfied. 

The time separation function is the same in both cases. It gives a measurable decomposition $\{X_q\}_{q\in Q}$ of $I^+(p)$ into timelike geodesics w.r.t.\ $u$. 
This  yields  disintegrations

\begin{align}\label{Eq:emms}
    (f(t)^N\,\dint t\otimes \m)\rvert_{I^+(p)}= \int_Q \m_q \,\dint \mathfrak q(q), \qquad (f(t)\,\dint t\otimes \dint r)\rvert_{I^+(p)}= \int_Q  \tilde \m_q \,\dint \tilde{\mathfrak q}(q),
\end{align}
where $\mathfrak q$ and $\tilde{\mathfrak q}$ are measures on the quotient space $Q$. Moreover, since a measure contraction property holds,  we have $\m_q=h_q(r) \,\dint r$ and $\tilde \m_q= \tilde h_q(r)\,\dint r$ with densities in $C^0([a_q, b_q])$ such that $h_q, \tilde h_q>0$ on $(a_q, b_q)$, see
\cite[Thm.\ 4.18]{CM:24a}. The set $X_q$ is the image of a maximal $1$-speed geodesic $\gamma_q:[a_q, b_q]\rightarrow I^+(p)$ that starts in $p$.  It is also worth noting that $Q$ is a smooth manifold since the underlying space is smooth and $u$ is the distance function to a point, but this fact does not play a role. Because of the same reason $\tilde h_q$ is smooth for all $q$. 

We consider the Lebesgue decomposition of $\m$ w.r.t.\ the Lebesgue measure, that is\\ $\m=~h\mathcal L^1 +~\m^s$.  There exists a measurable set $S\subseteq[a,b]$ such that $\mathcal L^1(S)=0$ and $\m^s= \m|_S$. We claim that $\m(S)=0$.

Assume by contradiction that $\m(S)>0$ and set $C=I\times S$. Then 
$$\int_C f({t})^N\,\dint {t}\, \dint\m>0 \quad \mbox{ but } \quad \int_Cf({t})\,\dint {t} \,\dint \mathcal L^1=0.$$

Recall that up to $\mathfrak q$-null sets the minimal geodesic segment $X_q$ is isometric to an interval with endpoints $a_q, b_q\in \R\cup\{\pm \infty\}$. We denote the relative interior of $X_q$ with ${X}^\circ_q$, i.e., $X_q^\circ \simeq (a_q, b_q)$. 

 It follows that there exists a set $P\subseteq Q$ with $\mathfrak q(P)>0$ and a set $N\subseteq Q$ with $\mathfrak q(N)=0$ such that $\m_q(C)=\m_q(C\cap X_q)>0$, $\m_q=h_q\,\mathcal L^1$  and $X^\circ_q\simeq (a_q, b_q)$ for for all $q\in P\setminus N$. If we pick such $q\in P\setminus N$, then the set $C_q= C\cap X_q^\circ\subseteq(a_q, b_q)$ satisfies $\mathcal L^1(C_q)>0$. The latter property holds since ${\int_{C_q}h_q \,\dint\mathcal L^1= \m_q(C_q)>0}$. 

Since $\mathfrak q(P\setminus N)>0$, we can pick such a $q\in P\setminus N$. 

Moreover, there can be only at most one $q\in P\setminus N$ such that $X_q\subseteq(I\times\{\tilde r\})$ for some $\tilde r\in [a,b]$ since every segment $X_q$ starts in the point $p$.

Since $\m_q(C_q)>0$, it holds necessarily that $C_q= (I\times S) \cap X^\circ_q\neq \emptyset$.  
We can pick one $s\in S$ such that $(I\times\{s\})\cap X^\circ_q\neq \emptyset$. 
A neighbourhood $U(s)= (s-\eta, s+ \eta)\subseteq(a,b)$ of $s$ has the property, that $I\times\{\tilde s\}\cap X_q^\circ\neq \emptyset$ for every $\tilde s\in U(s)$. We distinguish two cases: either $(1)$ each $I\times \{\tilde s\}$ intersects with $X_q^\circ$ only once, or $(2)$ $X_q^\circ \subseteq I\times\{s\}$. These properties pertain to geodesics in smooth spaces.

Let us assume $(1)$. We can define a map $F_s: \tilde s\in U(s)\mapsto X_q^\circ \cap( I\times\{\tilde s\})$.  Notice that the image $F_s(S\cap (s-\eta, s+\eta))$ of $S$ under $F_s$ is exactly the set $(I\times (S\cap (s-\eta, s+\eta))) \cap X_q^\circ$. This map is clearly smooth on $(s-\eta, s+\eta)$. Hence $F_s(S\cap (s-\eta, s+\eta))$ is a $\mathcal L^1$-null set in $X_q^\circ$. We can cover $X_q\cap C$ with a countable number of such sets. Consequently $C_q$ is a $\mathcal L^1$-null set. But this contradicts $\mathcal L^1(C_q)>0$. 

Hence, only $(2)$ happens, i.e., $X_q^\circ \subseteq I\times \{s\}$. In particular, $p\in I\times \{s\}$. But this is a contradiction if we choose $p\notin C$.

Hence there exists a measurable function $h: [a,b]\rightarrow [0,\infty)$ such that $\m=h\,\dint \mathcal L^1$. In particular, by \eqref{Eq:emms} we have $\m_q= f^{N-1}h\, \tilde \m_q$ for $\mathfrak q$-a.e.\ $q$. Therefore $h_q(r)= f^{N-1}(\alpha_q(r))h(\beta_q(r)) \tilde h_q(r)$ for $\mathfrak q$-a.e.\ $q$ where $\gamma_q(r)=(\alpha_q(r), \beta_q(r))$ is a timelike geodesic such that $\Im\gamma_q=X_q$. This is true for an arbitrary disintegration. We also recall that $\beta_q(r)$ is reparametrized geodesic in $[a,b]$. Since $h_q, \tilde h_q\in C^0([a_q, b_q])$, it follows that  $h\in C^0([a,b])$.
\end{proof}
\begin{thm}\label{th:lasttheorem} Let $\kappa\in \R$ and $N>1$.
     Let  $([a,b], \m)$ be a metric measure space for a Radon measure $\m$ with $\spt \m=~[a,b]$ and let $f:{}^\pm\! I \rightarrow [0, \infty)$ be smooth with $\partial I= f^{-1}(\{0\})$. Assume 
     $\left(^\pm \!I\times_f[a,b], f(t)^N \dint t\otimes \m\right)$
     satisfies \begin{enumerate}\item[(+)] $\CD(\kappa N, N+1)$ \item[(--)] $\tCD_p(-\kappa N, N+1)$, $p\in (0,1).$
     \end{enumerate}  If we consider $I={^+}\!I$, then we assume additionally $\partial I\neq \emptyset$. Then it  follows that 
\begin{enumerate}
    \item $f''\pm \kappa f\leq 0$, 
    \item \begin{enumerate}
        \item[(+)]
    $([a,b], h(r)\,\dint r)$ satisfies $\CD((N-1)\eta, N)$ with $\eta= \sup_I\left\{ (f')^2 + \kappa f^2\right\}$.
    \item[(--)] 
    $([a,b], h(r)\,\dint r)$ satisfies $\CD((N-1)\eta, N)$ with $\eta= \sup_I\left\{ -(f')^2 + f'' f\right\}$.
    \end{enumerate}
\end{enumerate}
    In particular, $\m= h(r)\,\dint r$ where  $h$ is semi-concave on $[a,b]$ and continuous on $(a,b)$.
\end{thm}
\begin{proof}
    Since under the assumption of the theorem a curvature-dimension condition always implies a measure contraction property, the corollary is a consequence of the previous lemma and Corollary \ref{Pr:from Y to X 2 dim and h semi-concave}.
\end{proof}

\section{From the fiber to the generalized cone \texorpdfstring{$Y$}{Y}}\label{sec-fib-Y}
Throughout this section we assume that $f:I\rightarrow[0,\infty)$ is smooth and satisfies
\begin{itemize}
    \item $f''- \kappa f\leq 0$, 
    \item $-(f')^2 + \kappa f^2 \leq \eta$. 
\end{itemize}
\begin{rem}
We note that we could replace $\kappa\in \R$ by $-\kappa\in \R$. Then the assumptions would take the 
 equivalent form\begin{itemize}
    \item $f''+\kappa f\leq 0$, 
    \item $(f')^2 + \kappa f^2 \geq -\eta$. 
\end{itemize}
\end{rem}
The main purpose of this section is to prove the following result.

\begin{thm}[From $X$ to ${}^-\!I\times^N_fX$] \label{th:totmcp}
Assume that $(X,\met,\m)$ is an essentially nonbranching metric measure space satisfying $\CD(\eta (N-1), N)$ for $N>1$, and $f: I \rightarrow [0, \infty)$ is smooth as above. Then ${}^-\!I\times_f^N X$, i.e., the generalized cone ${}^-\!I\times_f X$ equipped with the measure $\n:=f(t)^N \,\dint t \otimes\dint\m$, satisfies $\tmcp(-\kappa N, N+1)$. 
\end{thm}

We first prove a preliminary lemma.
Recall the following set
$$\Geo(X)= \{\bar \beta:[0,1]\rightarrow X: \bar \beta \mbox{ is a constant speed geodesic}\}.$$
The set $\Geo(X)$ is endowed with the topology given by the uniform distance, $\met_\infty$, which is defined as $\smash{\met_\infty(\bar \beta_0, \bar \beta_1)= \sup_{ s\in [0,1]}\met_X(\bar \beta_0(s), \bar \beta_1(s))}$.
\begin{lem}[{{Measurable selection}}]\label{Lem:selection}
Under the same assumptions as \Cref{th:totmcp}, there exists a  measurable map $$\Psi^+: I^+((\bar t, \bar x))\subseteq \I\times_f X\longrightarrow \TGeo(\I\times_fX)$$ such that $\Psi^+(t,x)$ is the unique future-directed timelike geodesic between $(\bar t, \bar x)$ and $(t, x)$ with speed $\uptau((\bar t, \bar x), (t,x))$. 

Similarly, there exists an $\n$-a.e.\ defined measurable map  $$\Psi^-: I^-((\bar t, \bar x))\subseteq \I\times_f X\longrightarrow \TGeo(\I\times_fX)$$ for time reversed geodesics.


\end{lem}
\begin{proof}[Proof of Lemma \ref{Lem:selection}] We prove the statement for $\Psi^+=:\Psi$. 

Since $X$ satisfies the condition $\CD(\eta (N-1), N)$ {and is essentially nonbranching}, for a fixed $\bar x\in X$ there exists an $\m$-almost everywhere defined, measurable map $$\Phi: X\rightarrow \mbox{Geo}(X)$$
 such that $\Phi(x)=: \bar\beta_x$ satisfies $\bar \beta_x(0)=x$ and $\bar \beta_x(1)=\bar x$, and $\bar \beta_x:[0,1]\rightarrow X$ is the unique constant speed geodesic between $x$ and $\bar x$. 
 \smallskip\\
 {\it Remark:}
 We note that by use of a measurable selection theorem   there always exist a measurable map $\Phi$ that assigns to $x\in X$ a geodesic between $\bar x$ and $x$ {\cite[Lem.\ 2.11]{AG:13}, \cite[Thm.\ 6.9.2]{Bog:07b}.} Moreover, provided $X$ is an essentially nonbranching $\CD$ space it  follows from \cite[Corollary 5.4]{CM:17a} that for $\m_X$-a.e.\ $x$ there is exactly one geodesic between $\bar x$ and $x$. This yields $\Phi$ with the claimed properties.
\smallskip

For  $(t,x)\in I^+((\bar t, \bar x))\subseteq I\times_f X$, such that $\Phi(x)=\bar\beta_x$ is defined, we  consider a future-directed timelike geodesic $\gamma=(\alpha, \beta):[0,1]\rightarrow \I\times_f X$ between $(\bar t, \bar x)$ and $(t, x)$, parametrized proportional to {$\uptau$-}arclength, i.e., $\sqrt{\dot \alpha^2 - (f^2 \circ \alpha) v_\beta^2}\equiv \uptau((\bar t, \bar x), (t,x))$.  According to Theorem \ref{th:fiber}  $\beta: [0, 1] \rightarrow X$ is a length minimizer and therefore a pre-geodesic between $\bar x$ and $x$, hence a reparametrization of $\bar \beta_x$. Moreover, $\gamma$ is the unique geodesic between $(\bar t, \bar x)$ and $(t,x)$, since $\bar \beta_x$ is the unique geodesic between $\bar x$ and $x$.

By Theorem \ref{th:fiber},(ii) $\alpha$ only depends on $\bar t, t$ and $d_X(\bar x, x)=L(\beta)$. In particular, if  we consider the generalized cone $\I\times_f [0, \infty)$ and the geodesic $\tilde \gamma=(\tilde \alpha, \tilde \beta)$ between $(\bar t, 0)$ and $(t, d)$, $d:=\met(\bar x, x)$, where $\gamma$ is defined on $[0,1]$ and parametrized proportional to {$\uptau$-}arclength, we have that $\tilde \alpha= \alpha$. The curve $\tilde \beta:[0,1] \rightarrow [0,d]$ is a reparametrization of $s\in [0,1]\mapsto s\cdot d$, and it holds that $\smash{\beta= \bar \beta_x\circ \frac{1}{d} \tilde \beta}$. Moreover the map $\tilde \Psi$ that assigns to the tuple $(t, d)$ the geodesic $\tilde \gamma$ connecting $(\bar t, 0)$ and $(t, d)$ in $\I\times_f [0, \infty)$, is measurable.  We also recall that the length functional $\L$ on $\mbox{Geo}(X)$ is a measurable map. 

Now we define $\Psi: I\times X\rightarrow {\mathrm{T}}\mbox{Geo}({Y})$ as $$\Psi(r,x):=  \left( P_1\Big[\tilde \Psi\Big(r,  \mathcal L(\Phi(x))\Big)\Big], \Phi(x)\circ \frac{1}{d}\tilde \beta\right)$$
where $P_1$ is the projection to the first component. By definition $\Psi(t,x)$ is the geodesic that connects $(\bar t, \bar x)$ with $(t, x)$, defined on $[0,1]$ and parametrized by {$\uptau$-}arclength.  Since $\Psi$ is a composition of measurable maps, together with a reparametrization of $\Phi(x)$ that depends measurable on the parameters, it is a measurable map. 
%
\end{proof}
\begin{cor}\label{cor:plans} Let $(\bar t, \bar x)\in \I\times_f X$ and $\mu$ be an $\n$-absolutely continuous measure with $\spt \mu\subseteq I^+((\bar t, \bar x))$. Then there exists a unique measure $\Pi \in \mathscr P({\mathrm{T}}\mathrm{Geo}(Y))$ such that $(e_1)_\sharp \Pi= \mu$ and $(e_0)_\sharp\Pi= \delta_{(\bar t, \bar x)}$. 

A similar statement holds if $\mu\in I^-((\bar t, \bar x))$. 
\end{cor}
\begin{proof} Let $\Psi$ be the map constructed in the previous lemma. 
We can define the measure $(\Psi)_\sharp \mu=: \Pi$ which does the job. 
\end{proof}
\begin{prop}\label{prop:prop}
Under the same assumptions as \Cref{th:totmcp}, let $(\bar t, \bar x)= \bar p\in  \I\times_f X=Y$ and let $\mu_0$ be an $\n$-absolutely continuous probability measure {with $\supp\mu_0\subseteq I^-((\bar t,\bar p))$}. Let $(\mu_{{s}})_{{s}\in [0,1]}$ be the $\ell_p$-geodesic between $\mu_0$ and  $\delta_{\bar p}=\mu_1$. Then $\mu_s$ is {$\n$-}absolutely continuous with density $\rho_s$ for each ${s}\in [0,1)$ and $\rho_{(1-\eta) s_0 + \eta s_1}$, with $s_0<s_1\in [0,1)$ and $\eta\in [0,1]$, satisfies 
\begin{equation}
    \begin{aligned}\label{ineq:concave}  
         \rho_{\sigma({\eta})}(\gamma_{\sigma({\eta})})^{-\frac{1}{N}}  \geq \tau_{-\kappa N, N+1}^{({1-\eta})}\big({ (s_1-s_0)L(\gamma)}\big) \rho_{s_0}(\gamma_{s_0})^{-\frac{1}{N}}+\tau_{-\kappa N, N+1}^{({\eta})}\big(((s_1-s_0)L(\gamma)\big) \rho_{s_1}(\gamma_{s_1})^{-\frac{1}{N}}
    \end{aligned}
\end{equation}
for $\Pi$-almost every $\gamma\in{\mathrm T\!}\Geo(Y)$ where $\sigma(\eta)=(1-\eta) s_0 + \eta s_1$. The analogous result holds for the time reversed case.
\end{prop}

\begin{proof}
{\bf 1.} 
We consider $u= \met_{\bar x}:= \met(\cdot, \bar x)$ and perform a $1D$ localization   w.r.t.\ $u$. It follows from Theorem \ref{Th:1Dloc} that there exists a measurable partition $\{X_q\}_{q\in Q}$ of  $X$ (more precisely of the  set $\mathcal T^b_u$, but we have $\smash{\m(X\setminus\mathcal T^b_u)=0}$ since  $u=\met_{\bar x}$). 

This partition yields a disintegration $\{\m_q\}_{q \in Q}$ of $\m$ with measures $\m_q$ supported on $X_q$, as well as a corresponding quotient measure $\mathfrak q$ on $Q$. 

Since $X$ satisfies $\CD(\eta(N-1), N)$ we have the the following properties. 
\begin{itemize}
    \item
For $\mathfrak q$-a.e.\ $q$ the closure $\overline X_q$ is the image of a distance preserving map $ \beta_q: [a_q, b_q]\rightarrow X$. In other words $\overline X_q$ is a geodesically convex subset of $X$ that is isometric to $[a_q,b_q]$ via the distance preserving map $ \beta_q: [a_q,b_q]\rightarrow \overline X_q$.  

\item
For $\mathfrak q$-a.e.\ $q$ there exists $ h_q: [a_q,b_q]\rightarrow [0,\infty)$ such that 
$$\m_q= (\beta_q)_{\sharp}\left[ h_q \, \mathcal L^1|_{\scriptscriptstyle [a_q, b_q]}\right].$$
\item
 For $\mathfrak q$-a.e.\ $q\in Q$ the density $h_q: [a_q, b_q]\rightarrow [0, \infty)$ is a continuous function satisfying 
$$\frac{\dint^2}{\dint s^2} h_q^{\frac{1}{N-1}} + \eta h_q^{\frac{1}{N-1}}\leq 0 \ \mbox{ in the distributional sense in } (a_q,b_q).$$ 
In other words $(X_q, \met|_{\scriptscriptstyle X_q\times X_q}, \m_q)$ satisfies the condition $\CD(\eta(N-1), N)$. This follows from Theorem \ref{Th:CD1} and Remark \eqref{rem:thefollowingremark}.
\end{itemize}
\smallskip
{\bf 2.}
By fiber independence $Y_q:= \I \times_f X_q$ embeds isometrically into $\I\times_f X$ and $\{ Y_q\}_{q\in Q}$ is a measurable decomposition of $Y:=\I\times_f X$. Moreover there is a disintegration of the measure $\n=f(t)^N\mathcal\dint t \otimes \m$ w.r.t.\ $\{Y_q\}_{q\in Q}$ into 
$$f(t)^N\dint t\otimes \m_q=: \n_q.$$
The quotient measure is $\mathfrak q$, and  $\mathfrak{n}_q$ is supported on $Y_q$ for $\mathfrak{q}$-almost every $q\in Q$.  Also, $\n_q$ has a continuous density w.r.t.\ $\smash{f(t)\,\dint t\otimes{\mathcal{L}^1|_{[a_q,b_q]}}}$ that is $f(t)^{N-1}h_q({r})= G(t,{\mathfrak{G}(q,r)})$ where $\mathfrak G$ is the ray map that we introduced after Theorem \ref{Th:1Dloc}.

Hence, for $\mathfrak q$-a.e.\ $q$ the subset $Y_q$ equipped with $\n_q$ is the generalized $N$-cone over $([a,b], \m_q)$ w.r.t.\ $\I$ and $f$ and  satisfies the condition $\tCD(\kappa N, N+1)$ by \Cref{Th:2 dimensional non-smooth case}.
\smallskip\\
{\bf 3.}
 Let us consider an $\n$-absolutely continuous measure $\mu:=\mu_0$, i.e., $\mu= \rho \n$ for a measurable function $\rho$.  Let $\Pi$ be the measure given by Corollary \ref{cor:plans}.
 
The disintegration of $\n$ into $\{\n_q\}_{q\in Q}$ yields for $\mathfrak{q}$-a.e.\ $q\in Q$ absolutely continuous measures $\mu_q:= \rho \,\n_q=$ with respect to $\n_q$ such that 
$$ \int_{{Q}}  \mu_q \,\dint\mathfrak q(q)= \int_{{Q}}  \rho \,\n_q \,\dint\mathfrak q(q)= \rho\, \n = \mu.$$
\smallskip
\textbf{Claim 1:} For $\mathfrak q$-a.e.\ ${q\in Q}$ there exists a measure $\Pi_q$ on $\TGeo(Y)$ such that $q\in Q\mapsto \Pi_q(G)$ is a measurable map for every measurable set $G\subseteq \TGeo(Y)$, and such that 
$$\Pi= \int_{{Q}} \Pi_q \,\dint\mathfrak q(q).$$ 
The measure $\Pi_q$ is supported on $\TGeo(Y_q)$, and  is the unique measure  on $\TGeo(Y_q)$ with $(e_0)_\sharp \Pi_q= \mu_q$ and $(e_1)_\sharp \Pi_q=\delta_{\bar x}$.  The measure $\Pi_q$ is given by $(\Psi|_{Y_q})_\sharp \mu_q= \Pi_q$ for $\mathfrak q$-a.e.\ $q\in Q$.
\medskip\\
\textit{Proof of Claim 1}. The map $\Psi$ from \Cref{Lem:selection} is defined $\n$-a.e.\ on $I^+((\bar t, \bar x))$. Hence,  for $\mathfrak q$-a.e.\ $q\in Q$ the map $\Psi$ is defined $\n_q$-a.e.\ on $Y_q\cap I^+((\bar t, \bar x)).$ 
Moreover, by fiber independence $\Psi|_{Y_q} $ maps $Y_q$ to ${\mathrm{T}}\mbox{Geo}(Y_q)$. 
Hence, we can define $\Pi_q$ as in the claim.

We pick a measurable subset $  G\subseteq {\mathrm{T}}\mbox{Geo}(Y)$. It follows that $\Pi_q(G)=\mu_q(\Psi^{-1}(G)\cap Y_q)=\mu_q(\Psi|_{Y_q})^{-1}( G)=\mu_q(\Psi^{-1}(G))$ depends measurably on $q\in Q$, and
\begin{align*} 
\Pi(  G)= \mu(\Psi^{-1}(  G)) = \int_{{Q}}  \mu_q(\Psi^{-1}(  G)\cap Y_q) \,\dint \mathfrak q(q) = \int_{{Q}}  \Pi_q {(G)} \,\dint \mathfrak q(q).
\end{align*}
This gives the claim.
\hfill$\blacksquare$
\smallskip

\noindent
{\bf 4.}
From Proposition \ref{Th:2 dimensional non-smooth case} it follows that the slice $Y_q= I\times_f X_q$ equipped with the measure $f(t)^N\,\dint t \otimes \m_q$ satisfies the condition $\tCD_p(-\kappa N, N+1)$.   

From this and since $Y_q$ is in fact smooth, we infer that $(e_{{s}})_\sharp\Pi_q=\mu_{{s}, q}$ is $\n_q$-absolutely continuous, i.e., $\mu_{{s}, q}= \rho_{{s},q} \,\n_q$, and for $s_0<s_1\in [0,1)$ and $\sigma({\eta})= (1-{\eta})s_0 + {\eta} s_1$ we have
\begin{equation}
\begin{aligned}\label{ineq:tcd2} 
\rho_{\sigma({\eta}), q}(\gamma_{\sigma({\eta})})^{-\frac{1}{N}}\geq \tau_{-\kappa N, N+1}^{(1-{\eta})}(\uptau(\gamma_{s_0},\gamma_{s_1}))\rho_{s_0, q}(\gamma_{s_0})^{-\frac{1}{N}}+\tau_{-\kappa N, N+1}^{({\eta})}(\uptau(\gamma_{s_0},\gamma_{s_1})) \rho_{s_1, q}(\gamma_{s_1})^{-\frac{1}{N}}
\end{aligned}
\end{equation}
for  $\Pi_q$-almost every geodesic $\gamma$. 
\smallskip

We set $(e_s)_\sharp\Pi =: \mu_s$ for $s\in (0,1)$. 
\smallskip\\
\textbf{Claim 2:} It holds that $\int_Q \mu_{{s}, q} \dint \mathfrak q(q)= \mu_{{s}}$.
\smallskip\\
\textit{Proof of Claim 2}. We  compute for any measurable set $A\subseteq   Y$ that 
\begin{align*}
\mu_{{s}}(A)=(e_{{s}})_\sharp\Pi(A)= \Pi(e_{{s}}^{-1}(A))&= \int _{{Q}}\Pi_q(e^{-1}_{{s}}(A)) \,\dint \mathfrak q(q) \\&= \int_{{Q}} \big[(e_{{s}})_\sharp \Pi_q\, {(A)}\big] \,\dint \mathfrak q(q) = \int_{{Q}}\mu_{{s}, q} {(A)} \,\dint \mathfrak q (q).
\end{align*}
Hence
$\int_{{Q}} \mu_{{s}, q}\,\dint \mathfrak q(q)= \mu_{{s}}. $\hfill$\blacksquare$
\smallskip\\
\textbf{Claim 3:} $\mu_s$ is $\n$-absolutely continuous. 
\smallskip\\
\textit{Proof of Claim 3}.  Let us consider a $\n$-null set $N$. Then, $N$ is a $\n_q$-null set for $\mathfrak q$-a.e.\ $q\in Q$. 
It follows that $\mu_{s,q}(N)=0$ for $\mathfrak q$-a.e.\ $q\in Q$, and consequently $\mu_s(N)=0$. 
\hfill $\blacksquare$
\smallskip

Let $\rho_s$ be the density of $\mu_s$ w.r.t.\ $\n$. 
\smallskip\\
\textbf{Claim 4:} $\rho_{{s}, q}= \rho_{{s}}|_{Y_q}$ for $\mathfrak q$-a.e.\ $q\in Q$.
\smallskip\\
\textit{Proof of Claim 4}.
We pick a measurable function $\phi: Y\rightarrow [0,\infty)$. Then
\begin{align*}
\int_Q\int \phi \rho_s\, \dint \n_q\,\dint \mathfrak q(q)= \int \phi\,\dint\mu_{{s}} = \int_{{Q}}\int \phi \,\dint\mu_{{s}, q}\,\dint \mathfrak q(q) &= \int_{{Q}}\int \phi \rho_{{s}, q}\,\dint \n_q \,\dint \mathfrak q(q)
\end{align*}
This finishes the proof of the claim. \hfill $\blacksquare$
\smallskip\\

We pick a measurable set $G$ in ${\mathrm{T}}\mbox{Geo}(Y)$ and integrate \eqref{ineq:tcd2} w.r.t.\ $\Pi_q|_A$ and then w.r.t.\ $\mathfrak q$. It follows that
\begin{align*}
\int_G \rho_{\sigma({\eta})}(\gamma_{\sigma({\eta})})^{-\frac{1}{N}} \,\dint \Pi (\gamma)\geq \int_G\bigg[ \tau_{\kappa N, N+1}^{(1-{\eta})}&(\uptau(\gamma_{s_0},\gamma_{s_1})) \rho_{s_0}(\gamma_{s_0})^{-\frac{1}{N}}\\&+\tau_{\kappa N, N+1}^{({\eta})}(\uptau(\gamma_{s_0},\gamma_{s_1})) \rho_{s_1}(\gamma_{s_1})^{-\frac{1}{N}}\bigg] \dint \Pi(\gamma).
\end{align*}
Since $G$ was arbitrary, this finishes the proof of the theorem. 
\end{proof}
\begin{proof}[Proof of Theorem \ref{th:totmcp}] We will show the condition $\tmcp^-$. $\tmcp^+$ follows analogously.

Let $(\bar t, \bar p)\in Y$ and let $\mu_0$ be an $\n$-absolutely continuous probability measure such that  $\supp \mu_0\subseteq I^-((\bar t, \bar x))$. We specialize \eqref{ineq:concave} to $s_0=0$ and drop the second term to obtain
$$
         \rho_{\eta s_1}(\gamma_{\eta s_1)})^{-\frac{1}{N}}  \geq \tau_{-\kappa N, N+1}^{({1-\eta})}\big({ s_1L(\gamma)}\big) \rho_{0}(\gamma_{0})^{-\frac{1}{N}}
$$
for $\Pi$-a.e.\ $\gamma$ where $\sigma(\eta)= \eta s_1$ and $\eta\in [0,1]$ and $s_1\in (0,1)$.  We integrate this inequality w.r.t.\ $\Pi$ to obtain
$$
\int \rho_{\eta s_1}^{-\frac{1}{N}} \dint \mu_{\eta s_1} \geq \int  \tau_{-\kappa N, N+1}^{({1-\eta})}\big({ s_1L(\gamma)}\big) \rho_{0}(\gamma_{0})^{-\frac{1}{N}} \dint \Pi(\gamma).
$$
Finally, let $s_1 \uparrow 1$. The negative of the left hand side is the Renyi entropy and therefore lower semi-continuous. Moreover, $s_1\mapsto \tau_{-\kappa N, N+1}^{({1-\eta})}\big({ s_1L(\gamma)}\big)$ is continuous. Hence, the right hand side converges as $s_1$ goes to $1$. 
\end{proof}
\medskip
It is clear that our method also gives the corresponding statement for generalized cones as metric measure spaces which is  new in this generality.
\begin{thm}[From $X$ to ${}^+I\times^N_fX$]\label{th:mmsmcp} Assume $f:I\rightarrow[0,\infty)$ is smooth and satisfies
\begin{itemize}
    \item $f''+ \kappa f\leq 0$, 
    \item $(f')^2 + \kappa f^2 \leq \eta$. 
\end{itemize}
Assume that $(X,\met,\m)$ is an essentially nonbranching metric measure space satisfying\\ $\CD(\eta (N-1), N)$ for $N>1$, and $f: I \rightarrow [0, \infty)$ is smooth as above. 

Then $I\times_f^N X$, i.e., the generalized cone $I\times_f X$ equipped with the measure $f(t)^N \,\dint t \otimes\dint\m$, satisfies $\mcp(\kappa N, N+1)$. 
\end{thm}
\begin{proof}
The proof is a word by word adaption of the proof of \Cref{th:totmcp}. There are analogous statements of  Lemma  \ref{Lem:selection} and Corollary \ref{cor:plans}, and  we apply again \Cref{Th:2 dimensional non-smooth case}. 

It is important to notice that geodesics in $I\times_f X$ may intersect with points where $f$ vanishes. But if $\partial I= f^{-1}(\{0\})\neq \emptyset$, then it was shown \cite[Thm.\ 3.4] {Ket:13} that  any transport plan $\Pi$ such that $(\e_0)_\sharp \Pi$ is absolutely continuous w.r.t.\ the reference measure, is concentrated on a set of geodesics that don't intersect with points where $f$ vanishes. 
\end{proof}

\begin{rem}
Let us compare the previous theorem with the results in \cite{ohta:cones} where Ohta proved the following theorem. 
\medskip\\
{\bf Theorem} (\cite[Thm.\ 4.2]{ohta:cones}){\bf .} {\it 
Let $X$ be a metric measure space that satisfies the measure contraction property $\mcp(N-1,N)$. Then $[0, \infty)\times_r^N X$ satisfies $\mcp(0,N+1)$. }
\medskip\\
In this theorem the assumption on the metric measure space $X$ is weaker.  $X$ is only required to satisfy the measure contraction property $\mcp(N+1,N)$.  The proof is an explicit computation for the volume distortion of optimal transport to a delta distribution.
We observed that this works similarly for the Minkowski cone, i.e., the property $\mcp(-(N-1),N)$ for $X$ implies the property $\tmcp(0, N+1)$ for $\I\times_r^N X$.
However, it is not clear to us how to generalize this approach to  general warped products. 

On the other hand, the proof of our theorem uses a $2D$-localisation technique. This reduces the problem to the case of smooth $2$-dimensional weighted warped products and we can apply the results of Section \ref{sec-2d-mod-spa} that establish the full theorem for the case of  warped products over weighted intervals. But this method seems to break down  if we just assume the measure contraction property, since the  property $\mcp(K,N)$ for a smooth Riemannian (or Lorentzian) space is equivalent to a lower bound $K$ for the Ricci tensor, only if  $N$ is a priori equal to the  Hausdorff dimension of the space. 
\smallskip\\
In conclusion, the following question remains open. 
\begin{que}Can we replace the condition  $\CD(\eta (N-1), N)$ for $X$ in Theorem \ref{th:totmcp} (or in Theorem \ref{th:mmsmcp}) with the weaker property $\mcp(\eta (N-1), N)$? 
\end{que} 
\end{rem}

\section{From the generalized cone \texorpdfstring{$Y$}{Y} to the fiber}\label{sec-Y-fib}
In this subsection we will prove the following result. 
\begin{thm}[From $\I\times^N_fX$ to $X$]\label{Th: Y to X}
Let $N\in (1,\infty)$, $\kappa\in \R$, and let $X$ be a proper, essentially nonbranching, complete, geodesic metric space with a Radon measure $\m$. Assume that $\smash{Y:=\I\times^N_f X}$ satisfies $\tCD_p(-\kappa N, N+1)$ and is {timelike} {$p$-essentially nonbranching} where $f: I\rightarrow [0, \infty)$ is smooth with $\partial I=f^{-1}(\{0\})$. Then \begin{itemize}
    \item $f''- \kappa f\leq 0$, and
    \item $X$ satisfies $\CD(\eta(N-1), N)$ where
\begin{align}
\label{Eq: first order condition, Y to X}\sup_I\{ -(f')^2 + f'' f\} = :\eta.
\end{align}
\end{itemize}
\end{thm}
\begin{rem}  By fiber independence we know that $Y$  is timelike nonbranching  if and only if   $X$ is  nonbranching. 
\end{rem}
\begin{conjecture}
     {Timelike} $p$-essential nonbranchingness for $Y$ follows from  $X$ being essentially nonbranching, and $X$ being essentially nonbranching implies $Y$ to be timelike $p$-essentially nonbranching. 
\end{conjecture}

From \Cref{Th: Y to X} two corollaries follow: 
\begin{cor}\label{cor:cone}
Let $N\in (1,\infty)$, and let $X$ be a proper,  nonbranching, complete, geodesic metric space with a Radon measure $\m$. Assume that the Minkowski $N$-cone  $\smash{[0, \infty)\times^N_f X}$ satisfies $\tCD_p(0, N+1)$. Then, the metric measure space $(X, \m)$ satisfies the condition $\CD(-(N-1), N)$.
\end{cor}
\begin{cor}\label{cor-Y-TCD-X-RCD}
Let $N\in (1,\infty)$, $\kappa\in \R$ and $X$ be a proper, essentially nonbranching, complete, infinitesimal Hilbertian, geodesic metric space with a Radon measure $\m$. Assume that $\smash{Y:=\I\times^N_f X}$ satisfies $\tCD_p(-\kappa N, N+1)$ and is {timelike} {$p$-essentially nonbranching} where $f: I\rightarrow [0, \infty)$ is smooth with $\partial I=f^{-1}(\{0\})$. Then \begin{itemize}
    \item $f''- \kappa f\leq 0$, and
    \item $X$ satisfies $\RCD(\eta(N-1), N)$, where
$\sup\limits_I\{ -(f')^2 - f'' f\} = :\eta.$
\end{itemize}
\end{cor}

We notice that according to \Cref{Pr: various properties} $\I\times_fX$ is a globally hyperbolic  Lorentzian geodesic space. 

Since $\m$ is a Radon measure, it follows that $\n= f(t)^N\,\dint t \otimes \m$ is a Radon measure. 
\medskip

For the proof of \Cref{Th: Y to X} we will utilize the characterization of the curvature dimension condition $\CD$ for essentially nonbranching spaces in term of the $1D$-localisation, i.e., \Cref{thm:cavmil}. For this  we will fix a  localisation of $X$ given by a signed distance function. 
\medskip

 More precisely, we consider a partition $\{X_q\}_{q \in Q}$ of $X$  induced by a $1D$ localization of a signed distance function $u=\met_\phi$ as introduced in Subsection \ref{subsec:CD1}. More precisely, $\{X_q\}_{q\in Q}$ is a decomposition of $\mathcal T_u^b$, the transport set associated to $u$, and there exist distance-preserving maps $ \beta_q: [a_q, b_q]\rightarrow X$ such that $ \beta_q([a_q, b_q])= \overline{X}_q$.  This yields a partition $\{Y_q\}_{q\in Q}$ of $Y$ as in step {\bf 2} of the proof of Proposition \ref{prop:prop} and a disintegration  $\{\n_q\}_{q\in Q}$ of the measure $\n$. The slices $Y_q\subseteq Y$ are geodesically convex subsets in $Y$ and $Y_q$ is isometric to the (smooth) generalized cone $\I\times_f [a_q, b_q]$ via the map $F_\beta$ defined in  {\Cref{subsec:con_fiber}}. As before, we have  
$$\n_q=f(t)^N \dint t\otimes \m_q$$
where $\{\m_q\}_{q\in Q}$ is the disintegration of $\m|_{\mathcal T_u^b}$ that arises from $\{X_q\}$. 
Let $\mathfrak q$ be the quotient measure on $Q$ such that $\int_Q \m_q \dint \mathfrak q(q)= \m$. 

Since $\m$ is a Radon measure, we have the following property: if $J\times A$ is a bounded subset in $I\times X$, then $\n(J\times A)<\infty$, and consequently $\n_q(J\times A)<\infty$ for $\mathfrak q$-a.e.\ $q\in Q$.

\begin{rem}
We do not  know at this point whether the conditional measures $\m_q$ are {$\mathcal{H}^1$-}absolutely continuous. In the theory of $1D$-localization this follows usually from a curvature-dimension-type condition like the measure contraction property. However we do not assume this a priori for $X$ in this section.
\end{rem}

It is useful to recall how the {ray map} $\mathfrak G$ {was defined after \Cref{Th:1Dloc}}.
Since the function $u$ is the signed distance function associated to the set $S=\phi^{-1}(\{0\})$, we can modify the ray map accordingly so that each $(q,{r})\in \mathcal{V}$ gets mapped via $\mathfrak{G}$ to the {unique} point $x$ in  $X_q$ at signed distance {$\met$ equal to} ${r}$ from $S$. More precisely:
\begin{align*}
    &\mathfrak G:  \mathcal V\subseteq Q\times \R\rightarrow X,\\
    \mbox{graph}(&\mathfrak G)=\{ (q,{r},x) \in Q\times \R\times X:  \met_S(\beta_q({r}))={r} \mbox{ and } \beta_q({r})=x\}. 
\end{align*}
In particular there exist $a_q,b_q\in\R$ with $(a_q,b_q)\subseteq{\mathrm{Dom}(\mathfrak{G}(q,\cdot))}$.

We then build a natural extension  $\mathfrak H$ of $\mathfrak G$ to $Y$. 
\begin{defi}[{Sheet map}]\label{Def: shit map} The   {\it sheet map} $\mathfrak H$ is defined as
$$\mathfrak H: \mathcal W \subseteq Q\times (I\times_f \R) \rightarrow Y \ \ \mbox{ by } \ \ \mathfrak H(q, {t}, r)=({t}, \mathfrak G(q, r)).$$
\begin{itemize}
\item 
The map $\mathfrak H$ is Borel measurable and  $\mathcal W= \mathfrak H^{-1}(I\times \T^b_u)$. 
\item $\mathfrak H(q,\cdot,\cdot)=F_q: I\times_f [a_q, b_q]\rightarrow {Y_q}\subseteq Y$ is a Lorentzian embedding. 
\item $\mathfrak H:\mathcal W\rightarrow  I\times \T^b_u$ is bijective and its inverse is measurable and given by $$\mathfrak H^{-1}({t},x)=( \mathfrak Q(x), {t}, \met_S(x)).$$
\end{itemize}
\end{defi}

\begin{lem}\label{lem:cycmon}
Let $\mathcal C$ be any $\uptau^p$-cyclically monotone set in  $({\I}\times_f \R)^2$. Then the set 
    $$\smash{ \tilde{\mathcal C} = \left\{ \left(({t^0},{x^0}),({t^1},{x^1})\right): \exists q\in Q \mbox{ s.t. } (t^0, x^0), (t^1, x^1)\in Y_q \mbox{ and } \left(({t^0}, \met_\phi({x^0})), ({t^1}, \met_\phi({x^1}))\right)\in \mathcal C\right\}}$$
is a $\uptau^p$-cyclically monotone subset  of $({\I}\times_f X)^2.$
\end{lem}
\begin{proof}
Let $\smash{\left\{\left(({t^0_i}, {x^0_i}), ({t^1_i}, {x^1_i})\right)\right\}_{i=1, \dots, n}\subseteq \tilde{\mathcal C}}$.
Since $\left((t^0_i, x^0_i), (t^1_i, x^1_i)\right)\in Y_{q_i}$ for some $q_i\in Q$, it follows that $(x^0_i, x^1_i)\in X_{q_i}$. Hence, $$\left|\met_\phi(x^1_i) - \met_\phi(x^0_i)\right|= \met(x^0_i, x^1_i).$$
Together with fiber independence it follows that 
\begin{align}\label{eq:tra}\uptau^p\left( ({t^0_i},{x^0_i}), ({t^1_{i}}, {x^1_{i}})\right) =  \uptau^p\left( ({t^0_i},0), ({t^1_{i}}, \met(x^0_i,{x^1_{i}})\right)=\uptau^p\left( ({t^0_i}, \met_\phi({x^0_i})), ({t^1_{i}}, \met_\phi( {x^1_{i}}))\right)  .
\end{align}
We consider $\sigma\in \mathcal S_n$, a permutation of $\{1, \dots, n\}$, and points $(t_i^0, x_i^0)\in Y_{q_i}$ and $(t^1_{\sigma(i)}, x^1_{\sigma(i)})\in Y_{\hat q_i}$ where $x^0_i\in X_{q_i}$ and $x^1_{\sigma(i)}\in X_{\hat q_i}$.
We also consider the points $\tilde x^1_{\sigma(i)}\in X_{q_i}$ and $(t^1_{\sigma(i)}, \tilde x^1_{\sigma(i)})\in Y_{q_i}$ such that $\met_\phi(\tilde x^1_{\sigma(i)})= \met_\phi(x^1_{\sigma(i)})$.  We note that $$\met (x^0_i, \tilde x^1_{\sigma(i)})=| \met_{\phi}(\tilde x^1_{\sigma(i)})- \met_{\phi}(x^0_i)|\leq \met(x^0_i, x^1_{\sigma(i)}).$$ 
Hence, by Lemma 5.1 in \cite{AGKS:23} we have \begin{align}\label{ineq:it}\uptau((t_i^0, x_i^0), (t^1_{\sigma(i)}, x^1_{\sigma(i)}))\leq \uptau((t_i^0, x_i^0), (t_{\sigma(i)}^1, \tilde x^1_{\sigma(i)})).\end{align}
Consequently
\begin{align*} \sum_{i=1}^n\uptau^p\left( ({t^0_i}, {x^0_i}), ({t^1_{\sigma(i)}}, {x^1_{\sigma(i)}})\right) &\leq \sum_{i=1}^n \uptau^p\left((t^0_i, x^0_i), (t^1_{\sigma(i)}, \tilde x^1_{\sigma(i)})\right)\\
&=\sum_{i=1}^n \uptau^p\left( ({t^0_i}, \met_\phi({x^0_i})), ({t^1_{\sigma(i)}}, \met_\phi({x^1_{\sigma(i)}}))\right)\\
&\leq \sum_{i=1}^n \uptau^p\left( ({t^0_i}, \met_\phi({x^0_i})), ({t^1_{i}}, \met_\phi( {x^1_{i}}))\right) = \sum_{i=1}^n \uptau^p\left( ({t^0_i},{x^0_i}), ({t^1_{i}}, {x^1_{i}})\right).
\end{align*}
The first inequality follows from \eqref{ineq:it}.
The last inequality follows from the $\uptau^p$-cyclical monotonicity of $\mathcal C$ and the last equality follows from \eqref{eq:tra}
\end{proof}
{}

We fix a finite interval $[L_0, L_1]\subseteq \R$ and define 
\begin{align}\label{Eq:tilde Q}
\tilde Q:= Q_{[L_0, L_1]}:=\{ q \in Q: [L_0, L_1]\subseteq (a_q, b_q)\}\subseteq Q.
\end{align}
Hence $\I\times_f [L_0, L_1]\subseteq Y_q$ for every $q\in \tilde Q$.
\begin{prop}[{{The sheets satisfy $\tmcp$}}]
For $\mathfrak q$-a.e.\ $q\in \tilde Q$, the smooth Lorentzian length space $\I\times_f [L_0,L_1]$ equipped with the measure $\smash{\n_q|_{\I\times_f [L_0, L_1]}}$ satisfies a {timelike} measure contraction property. 
\end{prop}
\begin{proof} We will show that for every point $(\bar t, \bar r) \in \I \times_f [L_0, L_1]$ and for every bounded set $A\subseteq I^\pm(({\bar t}, \bar{r}))\subseteq\I\times_f [L_0, L_1]$  it holds \begin{align}\label{ineq:mmm}
    \n_q\geq (\e_s)_\sharp\left(\sigma_{K,N}^{(s)}\big(\uptau((\bar t, \bar r), \cdot)\big)^{{N}}\n_q(A) \bdpi_q\right) \quad\mbox{ for } \mathfrak q\mbox{-a.e.\ } q\in \tilde Q. 
\end{align}
Here $K\leq 0$ and its precise value does not play a role. 

\begin{rem} 
This is a measure contraction property in the sense  proposed by Ohta in \cite{Oht:07} for metric measure spaces. Instead of the coefficient $\smash{\tau_{K,N}^{(s)}}$ we use the coefficient $\smash{\sigma_{K,N}^{(s)}}$. By standard techniques, taking into account that  the underlying space is essentially nonbranching and even smooth, one can deduce the measure contraction property in the form of \Cref{Def: TMCP},  with $\sigma^{(t)}_{K,N}$ instead of $\tau^{(t)}_{K,N}$ (compare with \cite[Prop.\ 9.1]{CM:21}). 
\end{rem}
\noindent
{\bf (1)}
We fix a point $({\bar t}, {\bar r})\in\I\times_f [L_0, L_1]$ 
and  a bounded set $A\subseteq I^\pm(({\bar t}, \bar{r}))\subseteq\I\times_f [L_0, L_1]$. Let us assume w.l.o.g.\ that $A\subseteq I^+(({\bar t}, \bar r))$. We will prove that 
\begin{align}\label{ineq:amcp}
    \n_q\geq (\e_s)_\sharp\left(\sigma_{K,N}^{(s)}\big(\uptau((\bar t, \bar r), \cdot)\big)^{{N}}\n_q(A) \bdpi_q\right) \quad\mbox{ for } \mathfrak q\mbox{-a.e.\ } q\in \tilde Q. 
\end{align}
 It is clear that the set $\{(\bar t, \bar r)\}\times A$ is $\uptau^p$-cyclically monotone for every $p\in (0,1)$.

We consider the set $\smash{\TGeo_{(\bar t, \bar r)}^{A}}$ of future-directed timelike geodesics $\gamma
\in \TGeo(\I\times_f [L_0,L_1])$ such that $\gamma(0)=(\bar t, \bar r)$ and $\gamma(1)\in A$.
We set 
\begin{align*}
\smash{A_{{s}}:= \{({t},{r})\in I^+(({\bar t}, {\bar r})): ({t}, {r})= \gamma({s})\mbox{ for } \gamma\in \TGeo_{(\bar t, \bar r)}^{A}\}}.
\end{align*}

The map $\smash{T: A\rightarrow \TGeo_{(\bar t, \bar r)}^{A}}$ such that  $T(t,r)$ is a timelike future-directed geodesic between $(\bar t, \bar r)$ and $(t,r)\in A$ is measurable. In particular $\e_s(T(t,r))\in A_s$ for all $(t,r)\in A$.  
\medskip\\
{\bf (2)}
We choose a measurable and bounded  subset $\Lambda\subseteq \tilde{Q}$ with $0<\mathfrak q(\Lambda)<\infty$. 
Since $\Lambda \subset \tilde Q$ and $A\subset \I\times_f[L_0, L_1]$, we have $\Lambda\times A\subseteq \mathcal W$. If we set $\mathfrak H(\Lambda\times A)=\tilde A$, then $\tilde A$ is a measurable subset in $\I\times_f X$. Recalling the definition of  $\mathfrak H$ we notice that
$$\tilde A= \mathfrak H (\Lambda\times A)= \{(t,x): \beta_q(r)=x, \ q\in \Lambda,\  (t, r)\in A\}.$$
We recall that $\beta_q(r)=x$ implies $\met_\phi(x)=r$.

We can assume that $\n(\tilde A)>0$. Otherwise $\n_q(A)=0$ for $\mathfrak q$-a.e.\ $q\in \tilde Q$. In this case the inequality \eqref{ineq:amcp} holds. 

Since $\Lambda$, $A$ are bounded, $\tilde A$ is bounded. Hence $\n(\tilde A)<\infty$  since $\n$ is a Radon measure. By Fubini's theorem it follows that $\n_q( A)<\infty$ for $\mathfrak q$-a.e.\ $q\in \tilde Q$.

We also define $\mathfrak H(\Lambda \times \{(\bar t, \bar r)\})=: L$. Then $$L= \mathfrak H(\Lambda\times \{(\bar t, \bar r)\})=\{(\bar{t}, \bar{x}): \beta_q(\bar r)=\bar x, q\in \Lambda\}.$$
\noindent
{\bf (3)}
The set 
$$
\left\{ \big(({\bar t}, {\bar x}), ({t}, {x})\big):  \big(({\bar t}, \met_\phi({\bar x})), ({t}, \met_\phi({x}))\big)\in \{(\bar t, \bar r)\}\times A\right\}=\tilde{ \mathcal{C}}
$$
is $\uptau^p$-cyclically monotone in $({\I}\times_f X)^2$ by Lemma \ref{lem:cycmon}.

We set $\smash{\mu_0= \n(\tilde A)^{-1} \n|_{\tilde A}}$. 
{We define a map $\tilde T: Y\rightarrow\TGeo(Y)$ so that  $\tilde T(t,x)$ is a timelike future-directed geodesic that connects a point $(\bar t, \bar x)\in L$ and  $(t,x)$ such that $\bar x, x\in X_q$ and $\tilde T(t,x)(s)\in Y_q$ for all $s\in [0,1]$}. 

In particular, $(\e_0, \e_1)(\tilde T(t,x))\subseteq \tilde{\mathcal C}$. 
Since $\tilde{\mathcal{C}}$ is $\uptau^p$-cyclically monotone, it follows that 
$T_\sharp \mu_1=:~\bdpi$
is an optimal geodesic plan between $(\e_0)_\sharp \bdpi= \mu_0$ and $(\e_1)_\sharp\bdpi=\mu_1${, cf.\ \cite[Prop.\ 2.8]{CM:24a}}. Moreover $(\e_s)_\sharp \bdpi$ is concentrated in $\tilde A_s = \mathfrak H(\Lambda \times A_s)$. In particular, $\mu_0, \mu_1$ are strongly dualizable and since we assume that $\I\times_f X$ is {timelike} $p$-essentially nonbranching, $\bdpi$ is the unique optimal plan.

Moreover $\tilde T\circ \mathfrak H(q, \cdot)= T$ for $\mathfrak q$-a.e.\ $q\in \tilde Q$.
\medskip\\
{\bf (4)}
Fix a Borel set $B$. The {timelike} curvature-dimension condition on ${\I}\times_f X$ implies (see \Cref{rem:entropicmcp}) 
$$\n(B\cap \tilde A_{{s}})\geq g({s})\, \n\big((\e_s \circ \tilde T(t,x))^{-1}(B)\cap \tilde A\big),$$
where
$g: [0,1] \rightarrow (0, \infty)$ is given by $$g({s}):=g_{B,K,N}({s}):= \inf_{(t,{r})\in B}\sigma_{K,N}^{({s})}\left(\uptau((\bar t,{\bar r}),(t,{r}))\right)^N.$$

In terms of the disintegration of $\n$ w.r.t.\ $\mathfrak q$ this is 
$$\int_\Lambda \n_q(B\cap \tilde A_{{s}}) \,\dint \mathfrak q(q)\geq g({s}) \int_\Lambda \,\n_q\big((\e_s\circ T(r,x))^{-1}(B)\cap \tilde A\big)\, \dint \mathfrak q(q).$$
Since the set $\Lambda$ was arbitrary, this implies 
\begin{align}\label{ineq:mmcp}
   \n_q(B)\geq  \n_q( B\cap A_{{s}})\geq g({s}) \,&\n_q\big( (\e_s \circ T(t,r))^{-1}(B)\cap A\big) 
   \qquad\mbox{ for } \mathfrak q\mbox{-a.e.\ } q\in \tilde Q.
\end{align} 
We note 
$$\n_q\big( (\e_s \circ T(t,r))^{-1}(B)\cap A\big)=  \n_q ((t,r)\mapsto T(t,r)^{-1}(\e_s^{-1}(B))\cap A).$$

We set $
\mu_{q,0}:=\frac{1}{\n_q(A)}\n_q|_A
$ if $\n_q(A)>0$.
The measure $T_\sharp \mu_{q,0}=: \bdpi_q$ is the optimal dynamical plan between $\mu_{q, 0}$ and $\mu_{q,1}= \delta_{(\bar t, \bar r)}$, and $(\e_s)_\sharp \bdpi_q= \mu_{q,s}$ is supported on $A_s$.
Hence $$\mu_{q,s}(B)= \mu_{q,0} ((t,r)\mapsto T(t,r)^{-1}(\e_s^{-1}(B))\cap A)= \frac{1}{\n_{q}(A)} \n_q ((t,r)\mapsto T(t,r)^{-1}(\e_s^{-1}(B))\cap A).$$
Together with \eqref{ineq:mmcp} it follows
\begin{align*}
   \n_q(B)\geq  \n_q( B\cap A_{{s}})\geq & g(s) \,\n_q(A) \mu_{q,s}(B) 
   \qquad\mbox{ for } \mathfrak q\mbox{-a.e.\ } q\in \tilde Q.
\end{align*} 
\noindent
{\bf (5)}
Decomposing $B$ into a union of pairwise disjoint sets $B=\bigcup_i B_i$ where $B_i= B\cap \{ (t,r): \eps (i-1)< \uptau((\bar t, \bar r), (t, r))< \eps i\}$, applying the previous estimate with $B=B_i$ and finally letting $\eps\rightarrow 0$ yields 
$$ \n_q(B) \geq \int_B \sigma_{K,N}^{(s)}\big(\uptau((\bar t, \bar r), \cdot )\big)^{{N}}\,\n_q(A) (\e_s)_\sharp \bdpi_q.$$
If we choose $B=B_n$ where  $\{B_n\}_{n\in \N}$ is a countable stable-by-intersection generator for of the Borel $\sigma$-field,  we get
\begin{align*}
    \n_q\geq (\e_s)_\sharp\left(\sigma_{K,N}^{(s)}\big(\uptau((\bar t, \bar r), \cdot)\big)^{{N}}\n_q(A) \bdpi_q\right) \quad\mbox{ for } \mathfrak q\mbox{-a.e.\ } q\in \tilde Q. 
\end{align*}
This is inequality \eqref{ineq:amcp}.
\medskip\\
{\bf (6)} We can pick a countable collection of points $({\bar t}, {\bar r})$ such that \eqref{ineq:amcp} holds for $\mathfrak q$-a.e.\ $q\in \tilde Q$. By the stability of inequality \eqref{ineq:amcp} we derive that  for $\mathfrak q$-a.e.\ $q\in \tilde Q$ the inequality holds for all $({\bar t}, {\bar r})$. 
Hence $
({\I}\times_f [L_0, L_1], \n_q)$
satisfies \eqref{ineq:mmm}. 
\end{proof}

We obtain the following corollary directly from \Cref{lem:2Dmcp}.
\begin{cor}
The measure $\m_q$ is $\mathcal L^1$-absolutely continuous for $\mathfrak q$-a.e.\ $q\in \tilde Q$, i.e., $\m_q= h_q \mathcal L^1$, and $h_q\in C^0([L_0, L_1])$.
\end{cor}
Let $(\mu_{{s}})_{{s}\in [0,1]}$ be a $\ell_p$-Wasserstein geodesic concentrated in $\I\times_f (L_0, L_1)$ and absolutely continuous w.r.t.\ $f(t)\,\dint t \otimes\dint r$, the  Lorentzian volume  of $\I\times_f (L_0, L_1)$, {and such that $\mu_0, \mu_1$ are $p$-dualizable}.
 
\begin{rem}\label{Rem: smooth lmfd}
Note that $\I\times_f (L_0, L_1)$ is a smooth Lorentzian manifold. Hence, since $(\mu_{{s}})_{{s}\in [0,1]}$ is concentrated on $\I\times_f (L_0, L_1)$, it follows for the density $\xi_s$ w.r.t.\ $f(t)\,\dint t\otimes\dint r$ that 
$$\xi_0(p)= \det DT_{{s}}(p)\xi_{{s}}(T_{{s}}(p)) $$
where $T_{{s}}(x)$ is the unique optimal map between $\mu_0$ and $\mu_{{s}}$ \cite[Sec.\ 5]{McC:20}. The map $$\smash{{s}\in [0,1]\mapsto T_{{s}}(p)=\gamma_p(s)}=\smash{(\alpha_{p}({s}), \beta_{p}({s}))}$$ is a geodesic in ${\I}\times_f(L_0, L_1)$ for $\mu_0$-a.e.\ $p=(t_p,r_p)$. 
There also exists an optimal geodesic plan $\Pi\in\mathscr{P}({\mathrm{T}\!}\Geo({\I}\times_f (L_0, L_1)))$ such that $(\e_{{s}})_\sharp\Pi=\mu_{{s}}$ $\forall {s}\in [0,1]$.
\end{rem}
\begin{lem}\label{lem:prel}
For $\mathfrak q$-a.e.\ $q \in \tilde Q$ there exists a unique $\ell_p$-geodesic $(\mu_{{s}}^q)_{{s}\in [0,1]}$ between $\mu^q_0$ and $\mu^q_1$ where $\mu^q_i:=\mathfrak H(q, \cdot, \cdot)_{\sharp} \mu_i$ for $i=0,1$ in $\I\times_f X_q$ such that $\mu^q_{{s}}$ is $\n_q$-absolutely continuous with density $\rho_{{s}}^q$, $\mu^q_{{s}}= (e_{{s}})_\sharp \Pi^q$ for some $\Pi^q\in\mathscr P({\mathrm{T}\!}\Geo({\I}\times_f X_q))$, and 
\begin{align*}
    \rho^q_{{s}}(\gamma({s}))^{-\frac{1}{N+1}} \geq \tau_{-\kappa N, N+1}^{(1-{s})}(L_{{\uptau}}(\gamma))\rho^q_0(\gamma(0))^{-\frac{1}{N+1}} + \tau_{-\kappa N, N+1}^{({s})}(L_{{\uptau}}(\gamma) )\rho^q_1(\gamma(1))^{-\frac{1}{N+1}},
\end{align*}
for $\Pi_q$-a.e.\ $\gamma\in\TGeo(\I\times_fX_q)$.
\end{lem}
\begin{proof}
We define probability measures $\nu_i\in{\mathscr{P}}({\I}\times_f X)$ as follows: {we pick $\tilde Q'\subseteq \tilde Q$ with $0<\mathfrak q(\tilde Q')=:\lambda<\infty$. Then we set}
\begin{align*}
    \nu_i= {\frac{1}{\lambda}}\int_{\tilde Q'} \mathfrak H(q, \cdot, \cdot)_{\sharp}\mu_i \,\dint\mathfrak q(q), \ \ i=0,1. 
\end{align*}
\noindent
\textbf{Claim:} 
For $i=0,1$ the measures $\nu_i$ are $\n$-absolutely continuous with a density $\rho_i$ that satisfies 
\begin{align}\label{equ:density} \rho_i(\mathfrak H(q, t, r))=\lambda^{-1} \cdot{h_q(r)^{-1}f(t)^{-(N-1)}}\xi_i(t, r),\end{align}
where $\xi_i$ are the densities of $\mu_i$ w.r.t.\ $f({t})\,\dint {t} \otimes \dint {r}$.\smallskip\\
\noindent\textit{Proof of the Claim}. First we prove that $\nu_i$, $i=0,1$, is $\n$-absolutely continuous.  Since $\mu_i$ is absolutely continuous w.r.t.\ $f({t})\,\dint{t} \otimes \dint{r}$, $\mu^q_i$ is  absolutely continuous w.r.t.\ $\n_q=f({t})^N\,\dint {t} \otimes \m_q$ for $\mathfrak q$-almost every $q\in \tilde Q$. This is because ${\I}\times_f (L_0, L_1)$ equipped with the measure $\n_q$ satisfies a measure contraction property for $\mathfrak q$-a.e.\ $q\in \tilde Q$, and hence $\m_q$ is $\mathcal L^1$-absolutely continuous such that its density $h_q$ has full support in $X_q$. 

Let $N\subseteq Y$ be a $\n$-null set. We have to prove that $\nu_i(N)=0$. By Fubini's theorem it follows that $Y_q\cap N$ is a $\n_q$-null set for $\mathfrak q$-a.e.\ $q\in \tilde Q$.  Hence $\mu_i^q(N\cap X_q)=0$ for $\mathfrak q$-a.e.\ $q\in \tilde Q$ and therefore $$\nu_i(N)=\frac{1}{\lambda}\int_{\tilde Q'}\mu_i^q(N\cap X_q) \,\dint\mathfrak q(q)=0.$$ So $\nu_i$ is $\n$-absolutely continuous with $\nu_i=\rho_i\,\n$.
\smallskip

Let $\phi: \mathcal W\rightarrow [0, \infty)$ be any measurable function where $\mathcal W$ is the domain of the sheet map.  Then, integrating $\phi$ w.r.t.\ $\nu_i$, we get
\begin{align*}\int_{\tilde Q'}\iint \phi(q,{t},{r}) &\rho_i(\mathfrak H(q,{t},{r})) {f({t})^N\,\dint{t}\,  \dint\m_q({r}) } \,\dint\mathfrak q(q)\\
&=\int_{\tilde Q'}\left(\int  \phi \circ \mathfrak H(q, \cdot, \cdot)^{-1}(y)\rho_i(y) \,\dint\n_q(y)  \right)\dint\mathfrak q(q)\\
&= \int   \phi \circ \mathfrak H(q, \cdot, \cdot)^{-1}(y) \rho_i(y)\,\dint\n(y)\\
&= \int   \phi \circ \mathfrak H(q, \cdot, \cdot)^{-1}(y) \,\dint\nu_i(y)\\
&= \frac{1}{\lambda}\int_{\tilde Q'} \int   \phi \circ \mathfrak H(q, \cdot, \cdot)^{-1}(y)\,\dint \mu_i^q(y)\, \dint \mathfrak q(q)\\
&= \frac{1}{\lambda}\int_{\tilde Q'} \int  \phi \circ \mathfrak H(q, \cdot, \cdot)^{-1}(y) \,\dint(\mathfrak {H}(q, \cdot, \cdot)_\sharp \mu_i)(y)\, \dint \mathfrak q(q)\\
&= \frac{1}{\lambda}\int_{\tilde Q'} \iint \phi(q, {t}, {r}) \,\dint\mu_i({t},{r}) \,\dint \mathfrak q(q)\\
&= \frac{1}{\lambda}\int_{\tilde Q'} \left(\iint \phi(q, {t},{r})\xi_i({t},{r})  f({t})\,\dint {t}\,  \dint {r}\right)\dint\mathfrak q(q)\\
&= \frac{1}{\lambda}\int_{\tilde Q'}\iint \phi(q, {t}, {r}) \xi_i({t},{r}) h_q({r})^{-1} f(t)^{-(N-1)}f({t})^N\, \dint {t}\, \dint\m_q({r})\,\dint\mathfrak q(q).
\end{align*}
This proves \eqref{equ:density} and therefore the claim.
\hfill$\blacksquare$
\smallskip

We consider a $\ell_p$-geodesic $(\mu_{{s}})_{{s}\in [0,1]}$ between $\mu_0$ and $\mu_1$. By Lemma \ref{lem:cycmon}, 
$${s}\in [0,1] \mapsto \nu_{{s}}:=\frac{1}{\lambda}\int_{\tilde Q'} \mathfrak H(q, \cdot, \cdot)_{\sharp} \mu_{{s}}\,\dint\mathfrak q(q)$$
is an  $\ell_p$-geodesic in ${\I}\times_f X$ between $\nu_0$ and $\nu_1$.  Moreover, also by \Cref{lem:cycmon},  $\nu_0, \nu_1$ are strongly timelike $p$-dualizable in $Y$. Since $Y$ is also $p$-essentially nonbranching, $\nu_s$ is the unique geodesic between $\nu_0$ and $\nu_1$. Hence, by the curvature-dimension condition $\nu_{{s}}$ is $\n$-absolutely continuous for all ${s}\in [0,1]$, and as before one can check that its  density satisfies $$ 
\rho_{{s}}(\mathfrak H(q, t, r))={h_q(r)^{-1}f^{-(N-1)}(t)}\xi_{{s}}(t, r) $$
where $\xi_{{s}}$ is the density of $\mu_{{s}}$ w.r.t.\ $f({t})^N\,\dint{t} \otimes \dint {r}$.

We also note that $\mathfrak H(q, \cdot, \cdot)_\sharp \mu_{{s}}= \mu^q_{{s}}$ is the $\ell_p$-geodesic between $\mu_0^q$ and $\mu_1^q$ since $\mathfrak H(q, \cdot, \cdot)= F_q$ is a Lorentzian isometric embedding. Since $\nu_{{s}}$ has the density $\rho_{{s}}$ w.r.t.\ $\n=f(t)^N\,\dint{t} \otimes \m$, it is then clear that $\mu_{{s}}^q$ has the density $\rho_{{s}}|_{Y_q}=\rho_{{s}}^q$ w.r.t.\ $\n_q= f({t})^N\,\dint{t} \otimes \m_q$. 

We define the map $\hat F_q\colon \TGeo({\I}\times_f (L_0,L_1))\rightarrow\TGeo(Y_q)$ as follows. 
We first notice that the map $F_q: {\I}\times_f [a_q, b_q]\rightarrow Y_q$ is a Lorentzian isometric embedding. 
Hence, if we set $$\hat F_q(\gamma)({s}):= F_q\circ \gamma({s}), \ \gamma\in \TGeo({\I}\times_f(L_0,L_1)),$$ then $\hat F_q(\gamma)=:\hat \gamma$ is a geodesic that satisfies $\hat \gamma({s})\in Y_q$ for ${s}\in [0,1]$. We note that \begin{align}\label{eq:com}\e_{{s}}(\hat F_q(\gamma))=F_q\circ \gamma({s})= F_q(\e_{{s}}(\gamma)).\end{align}
Let $\Pi$ be the dynamical plan between $\mu_0$ and $\mu_1$, i.e., $(\e_s)_\sharp \Pi=\mu_s$ for $s\in[0,1]$. Hence, if we define 
$$\hat \Pi = \int_{\tilde Q} (\hat F_q)_\sharp\Pi \,\dint\mathfrak q(q)$$ 
then $(\e_{{s}})_\sharp \hat \Pi= \nu_{{s}}$. Indeed
\begin{align*}\nu_{{s}}(A)&=\int_{\tilde Q} \mathfrak H(q, \cdot, \cdot)_{\sharp} \big((\e_{{s}})_\sharp \Pi\big) (A) \,\dint\mathfrak q(q)\\
&= \int_{\tilde Q} \Pi\big((\e_{{s}})^{-1}((F_q)^{-1}(A))\big)\,\dint\mathfrak q(q)=\int_{\tilde Q}\Pi((\hat F_q)^{-1}((\e_{{s}})^{-1}(A))\,\dint \mathfrak q(q)=\hat \Pi((\e_{{s}})^{-1}(A)).\end{align*}
We set $(\hat F_q)_\sharp \Pi=: \Pi^q$. Then it also follows that $\mu_{{s}}^q= (\e_{{s}})_\sharp \Pi^q$. Indeed we compute
$$(\e_{{s}})_\sharp \Pi^q=(\e_{{s}} \circ \hat F_q)_\sharp \Pi\stackrel{\eqref{eq:com}}{=} (F_q\circ \e_{{s}})_\sharp\Pi= \mathfrak H(q, \cdot, \cdot)_\sharp \mu_{{s}}= \mu_{{s}}^q.$$

By the curvature-dimension condition on $Y$ it  holds
\begin{align*}
    \rho_s(\gamma({s}))^{-\frac{1}{N+1}}\geq \tau_{-\kappa N,N+1}^{(1-{s})}(\uptau(\gamma_0,\gamma_1))\rho_0(\gamma(0))^{-\frac{1}{N+1}}+ \tau_{-\kappa N,N+1}^{({s})}(\uptau(\gamma_0,\gamma_1))\rho_1(\gamma(1))^{-\frac{1}{N+1}}
\end{align*}
for $\Pi$-a.e.\ $\gamma\in \TGeo(Y)$. In other words, \begin{align*}
    \rho^q_{{s}}(\gamma({s}))^{-\frac{1}{N+1}}\geq \tau_{-\kappa N,N+1}^{(1-{s})}(\uptau(\gamma_0,\gamma_1))\rho^q_0(\gamma(0))^{-\frac{1}{N+1}}+ \tau_{-\kappa N,N+1}^{({s})}(\uptau(\gamma_0,\gamma_1))\rho^q_1(\gamma(1))^{-\frac{1}{N+1}}
\end{align*}
for $\Pi^q$-a.e.\ $\gamma\in \TGeo(Y_q)$ and $\mathfrak q$-a.e.\ $q\in \tilde Q$. This finishes the proof of the lemma.
\end{proof}

\begin{prop}\label{prop:cdslice}
    For $\mathfrak q$-a.e.\ $q\in \tilde Q$ the sheet $(\I\times_f (L_0, L_1), f({t})^N\,\dint{t} \otimes{h_q(r)}\,\dint{r})$ satisfies  $\tCD_p(-\kappa N, N+1)$. 
\end{prop}
\begin{proof} We fix $q\in Q$ and consider the corresponding sheet $^- I \times_f [L_0, L_1]$ equipped with $\n_q|_{^- I \times_f [L_0, L_1]}$ where $\n_q= f(t)^N \dint t \otimes h_q(r) \dint r.$  We set $\dint \vol: = f(t) \dint t \otimes \dint r.$



We fix $\mu_0, \mu_1 \in \mathscr P_{p,c}\left(^-I\times_f [L_0,L_1]\right)$ such that $\ell_p(\mu_0, \mu_1)<\infty$, $\mu_0, \mu_1\ll \n_q$ and $\mu_0, \mu_1$ are (strongly) timelike $p$-dualizable (note that $^-I\times_f[L_0, L_1]$ is smooth). 

Let $\Pi$ be the unique optimal dynamical plan between $\mu_0$ and $\mu_1$, and let $\pi=(\e_0, \e_1)_\sharp \Pi$ be the corresponding optimal coupling. 

Since $\mu_0$ and $\mu_1$ are compactly supported in $I \times [L_0,L_1]$, we  can assume that $I\times [L_0, L_1]$ is compact, and $(\e_t)_{\sharp}\Pi$ is concentrated in $I\times [L_0, L_1]$ for all $t\in [0,1]$.

There are densities $\tilde \rho_i, i=0,1,$ such that 
$$\mu_i= \tilde \rho_i \dint  \n_q= \tilde \rho_i f^{N}  \dint t \otimes h_q \dint r.$$
We set $\rho_i = \tilde \rho_i f^{N-1} h_q$. Then $\rho_i \in L^1(\vol)$. 
By compactness, we can assume that the  density $\tilde \rho_i$ is bounded by  $\tilde C>0$. 
%
%
\medskip

As a consequence of Lemma 3.16 in \cite{CM:24a} we have the following:
There exists a sequence of chronological couplings  $(\hat \pi^l)_{l \in \N}$ with the following properties: 
\begin{enumerate}
\item $\hat \pi^l=  \rho^l \dint \vol\otimes \dint \vol$ with $ \rho^l\in L^\infty\cap L^1(\vol\otimes \vol).$
\item $\hat \pi^l \rightarrow \pi$ converges weakly.
\item $(P_i)_\sharp \hat \pi^l =: \hat \mu_i^l=  \rho_i^l \dint \vol$ with $ \rho_i^l \rightarrow \rho_i$ in $L^1(\vol)$ as $l\rightarrow \infty$, $i=0,1$. 

From step 3 in the proof of Lemma 3.16 in \cite{CM:24a} we know that $\rho_i^l\leq \rho_i$ $\vol$-a.e.

Hence, we have also that $\tilde \rho_i^l\leq \tilde \rho_i$ $\n_q$-a.e.\ where $\tilde \rho_i^l= \rho_i^l f^{-N+1} h_q^{-1}.$ Moreover, after possibly going to a subsequence, $\tilde \rho_i^l\rightarrow \tilde \rho_i$ $\n_q$-a.e.

Then it follows by dominated convergence that $$S_N(\hat \mu_i^l | \n_q) \rightarrow S_N(\mu_i |\n_q)\quad (i=0,1). $$
 \end{enumerate}
 {\bf (1)} We choose a countable subset $\mathfrak D\subset C_c((I\times (L_0, L_1))^2)$ that is dense in $L^1(\vol\otimes \vol)$. 
Let $\mathfrak S$ be the  pairs $(\mu_0'', \mu_1'')$ of probability measures on $I\times_f [L_0, L_1]$ defined as follows: 


For $g\in \mathfrak D$ and  $\hat \pi= g\, \dint \vol \otimes \dint \vol$ we  consider $c^{-1} \hat \pi |_{\{\uptau>0\}}=: \pi'$ provided $c=\hat \pi(\{\uptau>0\})>0$. It follows that there exits an optimal causal coupling $\pi$ between the marginal distributions of $\pi'$ by Proposition 2.3 in \cite{CM:24a}. Then,  we define the chronological coupling $b^{-1} \pi|_{\{\uptau>0\}}=:\pi''$ provided $b=\pi(\{\uptau>0\})>0$. This yields, that $\mu_0''$ and $\mu_1''$ are  the marginal distributions of $\pi''$ and, in particular, are $p$-dualizable. 
 \medskip

 We will show that given the probability measures $\mu_0$ and $\mu_1$ from above, we can find a sequence $(\mu_0^k, \mu_1^k)\in \mathfrak S$ that approximates the pair $(\mu_0, \mu_1)$ well. 
\medskip\\
 {\bf (1.1)} Consider $\rho^l$ from before. 
 Since  $\rho^l\in L^\infty\cap L^1(\vol\otimes \vol)$, we can pick  $M:=M_l>0$ such that $\rho^l\leq M$ and
there exists a sequence of continuous functions $( g^{l,k})_{k\in \N}$ in $\mathfrak D$, uniformly bounded by $M$, such that $ g^{l,k}\rightarrow  \rho^l$ in $L^1(\vol\otimes \vol)$ for $k\rightarrow \infty$.  

We normalize $ g^{l,k}$ such that $ g^{l,k} \vol\otimes \vol=: \pi^{l,k}$ is a probability measure.  Then $\pi^{l,k}\rightarrow \hat \pi^l$ weakly. 

We define $(P_i)_\sharp \pi^{l,k}=\mu_i^{l,k}$ where $\mu^{l,k}_i= g_i^{l,k} \dint \vol$ with 
$$g^{l,k}_0(x_0)= \int g^{l,k}(x_0, x_1) \dint \vol(x_1) \ \ \& \ \  g^{l,k}_1(x_1)= \int g^{l,k}(x_0, x_1) \dint \vol(x_0)$$
and we note that $g^{l,k}_i(x_i)$, $i=0,1$,  is continuous with compact support in $I\times(L_0, L_1)$.

Let $\bar x_i$ be a Lebesgue point of $\rho^l_i\in L^1(\vol)$. Then, for all $\epsilon>0$ there exists $r>0$ such that 
$$\left| \rho^l_i(\bar x_i) -
-\!\!\!\!\!\!\int_{B_r(\bar x_i)}\rho^l_i(x_i) \dint \vol(x_i)\right|< \epsilon 
\  \ \& \ \  \left| g^{l,k}_i(\bar x_i) --\!\!\!\!\!\!\int_{B_r(\bar x_i)}g^{l,k}_i(x_i) \dint \vol(x_i)\right|< \epsilon.
$$
Hence 
\begin{align*}\left| \rho^l_i(\bar x_i) - g_i^{l,k}(\bar x_i)\right|
& \leq
{-}\!\!\!\!\!\!\int_{B_r(\bar x_i)} \int  | \rho^l(x_0, x_1)- g^{l,k}(x_0, x_1)| \dint \vol(x_0) \dint \vol(x_1)+ 2\epsilon,
\end{align*}
and therefore 
$\limsup_{k\rightarrow \infty} \left| \rho^l_i(\bar x_i) - g_i^{l,k}(\bar x_i)\right| \leq 2\epsilon.$
Since $\epsilon>0$ was arbitrary, it follows that $g^{l,k}_i \rightarrow \rho^l_i$ pointwise $\vol$-a.e.\  
\medskip

We showed that there exist couplings $\pi^{l,k}$ with the following properties: 
\begin{enumerate}
\item $\pi^{l,k}= g^{l,k} \vol\otimes \vol$ with $g^{l,k} \in C_c( (I\times (L_0, L_1))^2)$. 
\item $\pi^{l,k} \rightarrow \hat \pi^l$ weakly. 
\item $(P_i)_\sharp \pi^{l,k} = \mu_i^{l,k} = g_i^{l,k} \dint \vol$ with $g_i^{l,k}\in C_c(I\times (L_0, L_1))$ and $g_i^{l,k} \rightarrow \rho_i^l$  $\vol$-a.e.\ as $k\rightarrow \infty$. 

We define $\tilde g_i^{l,k}= g_i^{l,k} f^{-N+1} h_q^{-1}$. Then $\tilde g^{l,k}_i\rightarrow \tilde \rho^l_i$  as well as $(\tilde g^{l,k}_i)^{1-\frac 1 N} \rightarrow (\tilde \rho^l_i)^{1-\frac 1 N}$ $\n_q$-a.e. 

It also holds $(\tilde g^{l,k}_i)^{1-\frac 1 N} \leq (M f^{-N+1} h_q^{-1})^{1- \frac 1 N}$, and  since $f$ is smooth and $h_q$ is continuous, we have
\begin{align*}-S_N(Mf^{-N+1} h_q^{-1} \dint \n_q | \n_q)&= M^{1-\frac{1}{N}}\int (f^{-N+1} h_q^{-1})^{1-\frac{1}{N}} \dint \n_q\\
&= M^{1-\frac{1}{N}}\int (f^{-N+1} h_q^{-1})^{-\frac{1}{N}} \dint \vol = M^{1-\frac{1}{N}} \int f^{1-\frac{1}{N}} h_q^{\frac{1}{N}} \dint \vol <\infty .\end{align*}
Hence, by dominated convergence we have that $$S_N(\mu^{l,k}_i|\n_q)\rightarrow S_N(\mu_i^l|\n_q).$$
\end{enumerate}
Now, we can  pick a diagonal sequence  such  that $\pi^{l_k, k}\rightarrow \pi$ weakly and \begin{align}\label{lim:entropy} \mbox{$S_N(\mu^{l_k, k}_i|\n_q)\rightarrow S_N(\mu_i|\n_q)$, $i=0,1$, as $k\rightarrow \infty$. }
\end{align}
We set $\pi^{l_k, k}=:\pi^k$, $g^{l_k, k}_i=: g_i^k$ and $g^k:=g^{l_k, k}$. 
\medskip\\
{\bf (1.2)}
By lower semicontinuity over open subsets it holds that 
$$\liminf_{k\rightarrow \infty}  \pi^k(\{\tau>0\})\geq \pi(\{\tau>0\})=1.$$
Thus,  if we set $c_k=  \pi^k(\{\tau>0\})$, then $c_k\rightarrow 1$ for $k\rightarrow \infty$ and  it holds that
$$( \pi^k)':= c_k^{-1}  \pi^k|_{\{\tau>0\}}= c_k^{-1} g^k 1_{\{\tau>0\}} \dint \vol \otimes \dint \vol\rightarrow \pi \ \mbox{ weakly}.$$
We set $(g^k)'= c_k^{-1} g^k 1_{\{\tau>0\}}$ and define $(P_i)_\sharp( \pi^k)' =:( g_i^k)' \dint \vol=:( \mu_i^k)'$, $i=0,1$, where 
$$ (g_0^k)'(x_0) = \int (g^k)'(x_0, x_1) \dint \vol(x_1), \ \ (g_1^k)'(x_1) = \int (g^k)'(x_0, x_1) \dint \vol(x_0).$$
Moreover one shows that $0\leq (\tilde g_i^k)'\leq c_k^{-1} \tilde g^k_i \leq c_k^{-1} M_{l_k} f^{-N+1} h_q^{-1}$, $i=0,1$.

Therefore, we have $( \pi^k)'\in \Pi_{\geq 0}( (\mu_0^k)', (\mu_1^k)')$, and $\ell_p((\mu_0^k)', (\mu_1^k)')\in (0, \infty)$, and consequently, by Proposition 2.3 in \cite{CM:24a}  there exists an optimal coupling $\pi^k\in \Pi_{\geq}^{p, \opt}((\mu_0^k)', (\mu_1^k)')$.  

Using  compactness arguments we can extract a subsequence from $(\pi^k)_{k\in \N}$ (again denoted as $(\pi^k)_{k\in \N}$) such that $\pi^k$ converges weakly to an
optimal coupling $\bar \pi$ between $\mu_0, \mu_1$. Since $\mu_0, \mu_1$ are (strongly) timelike $p$-dualizable it follows that $\bar \pi= \pi$, because $\mu_0,\mu_1$ are absolutely continuous and the underlying space is smooth. 

It holds again
$\liminf_{k\rightarrow \infty}  \pi^k(\{\uptau>0\})\geq \pi(\{\uptau>0\})=1.$ Thus, we set $b_k= \pi^k(\{\uptau>0\})$ and define 
\begin{align*}(\pi^k)'':= b_k^{-1} \pi^k|_{\{\uptau>0\}} \mbox{ and } (\mu^k_i)''= (P_i)_{\sharp} (\pi^k)'', i=0,1, \mbox{ for $k$ sufficiently large.}\end{align*}
Then, $(\pi^k)''$ as well as $(\mu_i^k)'', i=0,1,$ converge weakly to $\pi$ and to $\mu_i, i=0,1$, respectively. $(\mu_i^k)''$ is still $\vol$-absolutely continuous and we write $(g^k_i)''$ for the densities.  It holds that $(g^k_i)''\leq b_k^{-1}(g_i^k)'$, $i=0,1$. 

We repeat: $(\mu^k_0, \mu^k_1)$ is a pair of $\n_q$-absolutely continuous probability measures that admits a unique optimal coupling $\pi$ concentrated on $\{\uptau>0\}$. Hence $\mu^k_0$ and $\mu^k_1$ are $p$-dualizable. 
\medskip\\
%
{\bf (1.3)}
If we set $(\tilde g^k_i)''= (g_i^k)'' f^{-N+1} h_q^{-1}$, then also $(\tilde g_i^k)''$ converges pointwise $\n_q$-a.e.\ to $\tilde \rho_i$.  We recall that $\tilde \rho_i$ is bounded by a constant $\tilde C$.

%
%
%
Hence, we consider $A^k_i= \{ (\tilde g^k_i)''\leq 2\tilde C\}$ and  define
$$\mathcal G^k=\{ (x_0, x_1)\in (I\times_f (L_0,L_1))^2: x_0\in A^k_0, x_1\in A^k_1\}. $$
We have $\n_q(\{g^k_i>2\tilde C\})\rightarrow 0$ for $k\rightarrow \infty$ by  pointwise convergence $\n_q$-a.e.\ since $\tilde \rho_i$ is bounded by $\tilde C$. Hence, $(\pi^k)''(\mathcal G^k)\rightarrow 1$ if $k\rightarrow \infty$. 

Let $(\pi^k)'''= (\pi^k)''(\mathcal G^k)^{-1} (\pi^k)''|_{\mathcal G^k}$, then $(\pi^k)'''$ is still an optimal plan and its marginal distributions  $(\mu_i^k)'''$ are still $\n_q$-absolutely continuous with densities bounded by $2\tilde C.$ Let $(\tilde g^k_i)'''$ the density w.r.t.\ $\n_q$. It satisfies $(\tilde g^k_i)''' \leq a_k^{-1} (g^k_i)''$. 

In summary: we produced an optimal, chronological coupling $(\pi^k)'''$  such that $(\pi^k)''' \rightarrow \pi$ weakly, and such that the densities of the marginal distributions w.r.t.\ $\n_q$  are uniformly bounded by $2\tilde C$ and satisfy 
\begin{align}\label{inequ:density_estimate}(\tilde g^k_i)''' \leq a_k^{-1} (\tilde g_i^k)'' \leq a_k^{-1}b_k^{-1} c_k^{-1} \tilde g^k\leq 2 \tilde g_k.
\end{align}

We can apply the argument developed in step 2b of the proof of  Theorem 3.15 in \cite{CM:24a} and show that 
$$\limsup_{k\rightarrow \infty} S_N( (\mu_i^k)'''|\n_q) \leq \limsup_{k\rightarrow \infty} S_N(\mu^k_i|\n_q).$$
We omit details and only note that  we use the concavity of the function $u(x)=x^{1-\frac{1}{N}}$, Jensen's inequality and the density  estimate \eqref{inequ:density_estimate}.
Together with \eqref{lim:entropy} it follows that 
$$\limsup_{k\rightarrow \infty} S_N( (\mu_i^k)'''|\n_q) \leq S_N(\mu_i|\n_q)$$
and by lower semi-continuous w.r.t.\ weak convergence of $S_N$ we can replace $\limsup$ with $\lim$.

Finally, since the densities $(g_i^k)'''$ are uniformly bounded, it follows that $(g_i^k)'''\rightarrow \rho_i$ $\n_q$-a.e.\ as in Corollary 5.7 of \cite{gigli_splitting_context}.
%

By Egorov's theorem, for a fixed $\epsilon>0$ we can find a set $N_\epsilon$ such that, after possibly extracting another subsequence, $\tilde g_i^k$ converges uniformly to $\tilde \rho$ on $\mathcal K\backslash N_\epsilon$. We say $g_i^k$ converges to $g_i$ almost uniformly. 

The following lemma generalizes  Lemma 3.3 in \cite{Stu:06b} and is of independent interest.
\begin{lem}\label{lem:help} Let $(X, \dint, \n)$ be a metric measure space. Let $\mu^k$ be probability measures  on $X$ with densities $\rho^k$, $k\in \N\cup\{\infty\}=: \bar \N$. We assume that $\rho^k$ converges to $\rho^\infty$ almost uniformly, that $\rho^k, k\in \bar \N$, is bounded by $C$, and that $(\rho^k)^{1-\frac{1}{N}}\rightarrow (\rho^\infty)^{1-\frac{1}{N}}$ converges in $L^1(\n)$.  Let $\Pi^k, k\in \bar \N$, be  couplings such that $(P_1)_\sharp \Pi^k=\mu^k$,  such that $\mu^k$ and $(P_2)_\sharp \Pi^k$ are both supported in a compact subset $\mathcal K$. We assume that $\Pi^k$ converges weakly to $\Pi^\infty$.   We set $\tau^{(t)}(x,y):= \tau_{K,N}^{(t)}(\uptau (x,y))$. Then 
$$\liminf_{k\rightarrow \infty} \int \tau^{(1-t)}(x,y) (\rho^k)^{-\frac{1}{N}}(x) \dint \Pi^k(x,y) \geq \int \tau^{(1-t)}(x,y) (\rho^\infty)^{-\frac{1}{N}}(x) \dint \Pi^\infty(x,y).$$
\end{lem}
\begin{proof}[Proof of the lemma]
Let $Q^k(x_0, \dint x_1)$ be a disintegration of $\Pi^k$ w.r.t.\ $\mu^k$. 
For $k \in \bar \N$ we define 
$$ v^k(x)= \int \tau^{(t)}( x_0,x_1) \dint Q(x_0, \dint x_1). $$
Since the marginal distributions of $\Pi^k$ are both supported in $\mathcal  K$, it follows that the functions $v^k(x)$ are uniformly   bounded by a constant $D>0$ that depends on $\mathcal K$. 

Let $\epsilon>0$ be arbitrary. 
Since $(\rho^\infty)^{-\frac{1}{N}}\in L^1(\mu^\infty)$,  there exists $0\leq \psi \in C_b(X)$ such that \begin{align}\label{ineq:e1}\int v^k \left| (\rho^\infty)^{-\frac{1}{N}} - \psi\right| \dint \mu^\infty\leq \epsilon.\end{align}
Moreover, the weak convergence of $\Pi^k$ to $\Pi^\infty$ implies that 
\begin{align*}&\int v^\infty(x) \psi(x) d\mu^\infty(x)= \int_{X^2} \tau^{(t)}(x_0, x_1) \psi (x_0) d\Pi^\infty(x_0, x_1)\\
& \ \ \ \ \ \ \ \ \  \ \ \ \ \ \ \ \leftarrow  \int_{X^2} \tau^{(t)}(x_0, x_1) \psi (x_0) d\Pi^k(x_0, x_1)=\int v^k(x) \psi(x) d\mu^k(x)\end{align*}
and therefore, for $k\in \N$ sufficiently large, like $k\geq k_1$, 
\begin{align}\label{ineq:e2}\int v^\infty(x) \psi (x)\dint \mu^\infty \leq \int v^k(x) \psi(x) \dint \mu^k(x) + \epsilon.\end{align}
Then, from \eqref{ineq:e1} and \eqref{ineq:e2} we get 
\begin{align}\label{ineq:e3}\int v^\infty(x) (\rho^\infty)^{-\frac{1}{N}} \dint \mu^\infty\leq \int v(x) \psi \dint \mu^\infty + \epsilon\leq \int v^k \psi \rho^k \de \n + 2 \epsilon. 
\end{align}

We can pick $N_\epsilon\subset X$ such that $\rho^k \rightarrow \rho^\infty$ uniformly on $X\backslash N_\epsilon$ and $\n(N_\epsilon)< \epsilon$.  In particular, for $\eta>0$ we choose $k_2$ such that for evergy $k\geq k_2$, we have 
$\left| \rho^k - \rho^\infty\right|\leq \eta \mbox{ on } X\backslash N_\epsilon.$

Consequently,  for  $k \geq k_2$ it follows
\begin{align*}\int v^k \psi \rho^k \dint \n &= \int_{X\backslash N_\epsilon} v^k \psi \rho^k \dint \n + \int_{N_\epsilon} v^k(x) \psi \rho^k \dint \n\\
&\leq  \int_{X\backslash N_\epsilon} v^k \psi (\rho^\infty + \eta) \dint \n + \int_{N_\epsilon} v^k \psi \rho^k \dint \n\leq \int_X v^k \psi \rho^\infty  \dint \n  + CM \eta \n(\mathcal K)+ DMC \epsilon . 
\end{align*}
Here $M>0$ is a constant such that $\psi(x)\leq M$. 
We  set $N=\n(\mathcal K)$ and set $C(\epsilon, \eta)=C(\epsilon, \eta | C, M, N, D) =CM N \eta + DMC \epsilon >0$. It follows
\begin{align}\label{ineq:e4}\int v^k \psi \rho^k \dint \n \leq \int v^k \psi \rho^\infty  \dint \n + C(\epsilon, \eta)\leq  \int  v^k (\rho^\infty)^{-\frac{1}{N}} \dint \mu^\infty + \epsilon + C(\epsilon, \eta).\end{align}
Since, by assumption, we have that $(\rho^k)^{1-\frac{1}{N}}$ converges to $(\rho^\infty)^{1-\frac{1}{N}}$ in $L^1(\n)$, and since $v^k$ is uniformly bounded, it follows from \eqref{ineq:e3} and \eqref{ineq:e4} for $k\geq \max\{k_1, k_2\}$ large enough
$$\int v^\infty (x) (\rho^\infty)^{-\frac{1}{N}} \dint \mu^\infty \leq \int v^k (\rho^k)^{-\frac{1}{N}} \dint \mu^k + C(\epsilon, \eta)+ \epsilon.$$
Since $\epsilon>0$ and $\eta>0$ are arbitrary, and since $C(\epsilon, \eta)\rightarrow 0$ for $\epsilon, \eta\rightarrow 0$, the claim follows. 
\end{proof}

\noindent{\bf (2)} 
Since $\mathfrak S$ is  countable, Lemma \ref{lem:prel} implies that there exists $Q'\subset Q$ such that $\mathfrak q (Q\backslash Q')=0$ and for every $q\in Q'$  we have the following property:

For every pair $(\mu_0, \mu_1)\in \mathfrak S$ there exists a $\ell_p$-Wasserstein $(\mu_t)_{t\in [0,1]}$ geodesic supported in $\I\times_f (L_0, L_1)$ such that 
\begin{align}\label{ineq:rhoq}
    \tilde g^q_t(\gamma(s))^{-\frac{1}{N+1}}\geq \tau_{\kappa N,N+1}^{(1-s)}(\uptau(\gamma_0,\gamma_1))\tilde g^q_0(\gamma(0))^{-\frac{1}{N+1}}+ \tau_{\kappa N,N+1}^{(s)}(\uptau(\gamma_0,\gamma_1))\tilde g^q_1(\gamma(1))^{-\frac{1}{N+1}}
\end{align}
for all $s\in (0,1)$ and for $\Pi^q$-a.e.\ $\gamma\in \Geo(Y_q)$ where for $s\in [0,1]$ $\tilde g^q_s$ is the density of $\mu_s$ w.r.t.\ $\n_q$ and $\Pi^q\in \mathcal P(\Geo(Y_q))$ with $(\e_t)_\sharp \Pi^q = \mu_t$.

Moreover, this stays true if we pick $(\mu_0, \mu_1)\in \mathfrak S$ with the associated dynamical plan $\Pi^q$, and  restrict $\Pi^q$ to  $\Pi^q(\mathcal G)^{-1} \Pi^q|_{\mathcal G}$ where 
$$\mathcal G= \{ \gamma\in \Geo : \gamma(i) \in \{ \rho_i^q \leq 2\tilde C\}, i=0,1\}.$$
{\bf (3)}
We pick two $\n_q$-absolutely continuous measures $\mu_0, \mu_1$ that are compactly supported in a set $\mathcal K\subset I\times_f (L_0, L_1)$ and have  densities $\tilde \rho_i$, $i=0,1$, w.r.t.\ $\n_q$ that are bounded by $\tilde C$.
\medskip

We find a sequence $(\mu^k_0,  \mu^k_1)\in \mathfrak S$ and the $\ell_p$-geodesic $(\mu^k_s)_{s\in [0,1]}$ between these two measures as constructed in the previous steps.

The sequence $(\mu_i^k)_{k\in \N}$  converges weakly to $\mu_i$, $i=0,1$.  Moreover, after extracting a subsequence $(\mu_s^k)_{k\in \N}$ converges to a $\mu_s$, and $(\mu_s)_{s\in [0,1]}$ is the $\ell_p$-geodesic between $\mu_0$ and $\mu_1$. Similarly, the sequence of plans $\Pi^{q,k}, k\in \N$, converges to $\Pi^{q}$ such that $(\e_s)_{\sharp} \Pi^q= \mu_s$. 
We integrate the inequality \eqref{ineq:rhoq} w.r.t.\ $\Pi^{q, k}$ and obtain
\begin{align*} &-\int (\tilde g^{q,k}_s)^{-\frac{1}{{N+1}}}(\gamma(t))\,\dint \Pi^{q,k}(\gamma)= -\int (\tilde g^{q,k}_s)^{1-\frac{1}{{N+1}}}(x)\,\dint \n_q(x)= S_N(\mu^{k}_s|\n_q) \\
&\ \ \ \ \ \ \ \ {\leq}-\!\int\Big[\tau_{-\kappa N,N+1}^{(1-s)}(\mathcal T)\tilde g^{q,k}_0(x_0)^{-\frac{1}{{N+1}}} 
+\tau_{-\kappa N,N+1}^{(s)}(\mathcal T)\tilde g^{q,k}_1(x_1)^{-\frac{1}{{N+1}}}\Big]\,\dint\left((\e_0, \e_1)_\sharp\Pi^{q,{k}}\right)(x_0,x_1).\end{align*}
where $\mathcal T:=\uptau(x_0,x_1)$.
The left hand side of this inequality is the $(N+1)$-Renyi entropy which is lower semi-continuous w.r.t.\ weak convergence.

The right hand side, on the other hand, is upper semi-continuous by the properties of the sequence $(\mu_i^{k})_{k\in \N}\subset \mathfrak D$ and because of Lemma \ref{lem:help}. 

It follows that 
\begin{align*} S_N(\mu_s|\n_q)\,{\leq}-\int\Big[\tau_{-\kappa N,N+1}^{(1-s)}(&\uptau(x_0,x_1))\tilde g^{q}_0(x_0)^{-\frac{1}{{N+1}}} \\&+\tau_{-\kappa N,N+1}^{(s)}(\uptau(x_0,x_1))\tilde g^{q}_1(x_1)^{-\frac{1}{{N+1}}}\Big]\,\dint\left((\e_0, \e_1)_\sharp\Pi^{{{k}}}\right)(x_0,x_1).\end{align*}
Hence, $(\I\times_f (L_0,L_1), \n_q)$ satisfies $\tCD_p(-\kappa N, N+1)$ for $\mathfrak q$-a.e.\ $q\in \tilde Q$.
\end{proof}

\begin{cor}
For $\mathfrak q$-a.e.\ ${q} \in Q$ the needle $(X_{ q}, \m_{ q})$ satisfies the condition $\CD(\eta (N-1), N)$, where $\eta =\sup_I\{-(f')^2 -f'' f\}$, and 
$$f''-\kappa f\leq 0.$$\end{cor}
\begin{proof} By Proposition \ref{prop:cdslice} and by Theorem \ref{th:lasttheorem} it follows that $h_q|_{(L_0, L_1)} $  for $\mathfrak q$-a.e.\  $q \in \tilde Q$ is a $\CD(\eta(N-1), N)$-density and $f$ satisfies $f''-\kappa f\leq 0$.

Let $(L_0^i, L_1^i)$ be a countable collection of open intervals such that $\forall q \in Q$ and $\forall t_0\in (a_q, b_q)$ there exists $i\in \N$ such that $t_0\in (L^i_0, L_1^i)\subseteq (a_q, b_q)$. Given the $\tilde Q^i$ associated to $(L^i_0, L_1^i)$ we have that $h_q$ restricted to $(L^i_0, L^i_1)$ is as above for $\mathfrak q$-a.e.\ $q\in \tilde Q^i$. In other words, there exists $\mathcal N^i$ with $\mathfrak q(\mathcal N^i)=0$ such that $h_q|_{(L^i_0, L^i_1)}$ is such a density $\forall q \in \tilde Q^i\setminus\mathcal N^i$. We set $\mathcal N:= \bigcup_i \mathcal N_i$, which is still a null set. Hence, $\forall q\in Q\setminus \mathcal N$ and $\forall t_0\in (a_q, b_q)$ there exists a neighbourhood of $t_0$ such that $h_{\ q}$ is a $\CD(\eta(N-1), N)$ density. Consequently $\forall q \in Q\setminus\mathcal N$ we have that $h_q$ is a $\CD(\eta(N-1), N)$ density. 
\end{proof}

\begin{proof}[\it Proof of \Cref{Th: Y to X}.]
Since the distance function $\met_\phi$ from which we derived the needle decomposition on $X$ was arbitrary we obtain  \Cref{Th: Y to X} from Theorem \ref{thm:cavmil}.\end{proof}

\section{Applications}\label{sec-app}
{In this final section we give several applications, in particular we provide new examples of \LLSn s which satisfy the timelike measure contraction property. Also, we point out several open questions.}

\subsection{New examples of spaces with synthetic timelike Ricci curvature bounds}\label{subsec-new-ex}
\subsubsection{Generalized cones as Lorentzian length spaces}
\Cref{th:totmcp} provides a new class of examples of measured Lorentzian length spaces $\I\times_f^N X$ that satisfy a timelike measure contraction property $\tmcp$, i.e., a generalized strong energy condition.

Let $f: I\rightarrow [0, \infty)$ be a smooth function such that \begin{enumerate}
    \item $f$ is $(-\kappa)$-concave, and
    \item $ \kappa f^2 - (f')^2\leq \eta$. 
\end{enumerate}
We can now pick any essentially nonbranching $\CD(\eta(N-1), N)$ space $X$ to produce a generalized cone ${}^-I \times_f^N X$ that is a measured Lorentzian geodesic space satisfying the condition $\tmcp(-\kappa N, N+1)$. 

In particular,  we can choose  $X$ as an  $\RCD(\eta(N-1), N)$ space. By now there is wide spectrum of different classes of metric measure spaces that satisfy an $\RCD$ condition:
\begin{itemize}
    \item $N$-dimensional Alexandrov spaces with curvature bounded from below $\eta$ \cite{pet11}.
    \item Ricci limit spaces, i.e., metric measure spaces that arise as pointed measured Gromov--Hausdorff limits of Riemannian manifolds with Ricci curvature bounded from below by $\eta (N-1)$ and dimension bounded from above by $N$ \cite{Stu:06a, Stu:06b, LV:09, gsm15}.
    \item Weighted $N$-dimensional Riemannian manifolds whose Bakry--\'Emery $N$-Ricci tensor is bounded from below by $\eta (N-1)$ \cite{Stu:06a, Stu:06b, LV:09}.
    \item $N$-dimensional smooth manifolds equipped with a $C^0$-Riemannian metric that satisfy a lower Ricci curvature bound $\eta (N-1)$ in the distributional sense \cite{mr24}.
    \item $N$-dimensional stratified spaces that satisfy a generalized lower Ricci curvature bound $\eta (N-1)$ \cite{bkmr21}.
\end{itemize}
{This leads to the following natural question (to which we expect a positive answer).}
\begin{que} Is $I\times^N_f X$ infinitesimal Minkowskian in the sense of \cite[Def.\ 1.4]{BBCGMORS:24} if the metric measure space $X$ is infinitesimal Hilbertian?
\end{que}
The more general class of $\CD$ spaces includes
\begin{itemize}
\item Weighted $N$-dimensional Finsler manifolds such that the generalized Bakry--\'Emery\\ $N$-Ricci curvature is bounded from below by $\eta (N-1)$ \cite{ohta_interpolation}.
\end{itemize}

We can also choose $f$ to be $(-\kappa)$-affine, i.e., $f''- \kappa f=0$. Then $\kappa f^2 - (f')^2 \equiv \eta\in \R$. By rescaling $f$ and $I$ we can assume $\kappa, \eta \in \{-1, 0, +1\}$. 
Since $\partial I= f^{-1}(\{0\})$, it follows that, up to Lorentzian isometries,  we have precisely the following list of examples. Here $X$ is a general metric measure space.
\begin{enumerate}
    \item[(L1)] Anti-de-Sitter $N$-suspension: $^-[0,\pi]\times_{\sin}^N X$,
    \item[(L2)] Minkowski $N$-cone: $^-[0, \infty)\times_{\mathrm{id}}^N X$, 
    \item[(L3)] Minkowski product: $^-[0, \infty)\times X$, 
    \item[(L4)] Lorentzian Elliptic $N$-cone: $^-[0, \infty)\times^N_{\sinh} X$, 
    \item[(L5)] Lorentzian Parabolic $N$-cone: $^-\R\times_{\exp}^N X$, 
    \item[(L6)] De-Sitter $N$-cone: $^-\R\times^N_{\cosh} X$. 
    \end{enumerate}
For the Minkowski product we have $f(r)\equiv 1$. Especially, the reference measure is\\ $1^N\mathcal L^1|_I\otimes~\m_X=\mathcal L^1|_I\otimes \m_X$, and the parameter $N$ disappears. 

In the following table we list  $I$ and $f$ together with the {corresponding}values of $\kappa$ and $\eta$.
\begin{center}
\begin{tabular}{|c|c|c|c|c|}
\hline &
\textbf{$I$}                     & {$f$}    & $\eta$& \ $\kappa$  \\ \hline
(L1)&$[0,\pi]$                        & $\sin$        & $-1$                 & $-1$                \\ \hline
(L2)&$[0,\infty]$                     & $\mathrm{id}$ & $-1$                 & $\hphantom{-}0$                 \\ \hline
          (L3)&       $\R$                             & $1$           & $\hphantom{-}0$                  & $\hphantom{-}0$                 \\ \hline
(L4)&$[0,\infty]$    & $\sinh$       & $-1$                 & $\hphantom{-}1$                 \\ \hline
(L5)&$\R$                             & $\exp$        & $\hphantom{-}0$                  & $\hphantom{-}1$                 \\ \hline
(L6)&$\R$                             & $\cosh$       & $\hphantom{-}1$                  & $\hphantom{-}1$                 \\ \hline
\end{tabular}\end{center}

If we choose  $N\in \N$ and $X=\mathbb H^N$, the standard hyperbolic space, in (L1), then ${}^-[0,\pi]\times_{\sin}\mathbb H^N $ is the $(N+1)$-dimensional  Anti-de-Sitter space. For the same choice in (L2) we obtain the $(N+1)$-Minkowski space. If we choose $N\in \N$ and  $X=\mathbb S^{N}$  in (L6) we obtain the  de-Sitter space.

In view of our result we make the following general conjecture.
\begin{conjecture}\label{conj:wp}
Let $N\in (1,\infty)$, $\kappa\in \R$, and let $X$ be a proper, complete, geodesic metric space with a Radon measure $\m$. Let $f: I\rightarrow [0, \infty)$ be Lipschitz such that $f^{-1}(\{0\})= \partial I$. 

Then $\smash{\I\times^N_f X}$ satisfies $\tCD(-\kappa N, N+1)$ \emph{if and only if }\begin{itemize}
    \item $f''- \kappa f\leq 0$ in distributional sense in $I$, and
    \item $X$ satisfies $\CD(\eta(N-1), N)$ where $\eta:=\sup_I\{ -(f')^2 + \kappa f^2\}$.
\end{itemize}
\end{conjecture}
In particular, this includes the following conjecture.
\begin{conjecture}\label{conj:cone}
Let $N\in (1, \infty)$ and let $X$ be a proper, complete, geodesic metric space with a Radon measure $\m$. The Minkowski $N$-cone $[0, \infty)\times_r^N X$  satisfies the condition $\tCD(0, N+1)$ if and only if the metric measure space $(X, \m)$ satisfies the condition $\CD(-(N-1), N)$.
\end{conjecture}
\subsubsection{Generalized cones as metric measure spaces}
In the positive signature case we choose a smooth function $f$ such that 
\begin{enumerate}
    \item $f$ is $\kappa$-concave, and
    \item $ \kappa f^2 + (f')^2\leq \eta$. 
\end{enumerate}
 as well as an essentially nonbranching $\CD(\eta(N-1), N)$ space $X$. Then $I\times_f^N X$ satisfies $\mcp(\kappa N, N+1)$ by \Cref{th:mmsmcp}.

 For  $f''+\kappa f=0$ and $\kappa f^2 + (f')^2= \eta$ the construction yields exactly one of the following spaces 
\begin{enumerate}
    \item[(R1)] Spherical $N$-suspension: $[0,\pi]\times_{\sin}^N X$,
    \item[(R2)] Euclidean $N$-cone: $[0, \infty)\times_{\mathrm{id}}^N X$, 
    \item[(R3)] Cartesian product: $[0, \infty)\times X$, 
    \item[(R4)] Elliptic $N$-cone: $[0, \infty)\times^N_{\sinh} X$, 
    \item[(R5)] Parabolic $N$-cone: $\R\times_{\exp}^N X$, 
    \item[(R6)] Hyperbolic $N$ cone: $\R\times^N_{\cosh} X$. 
    \end{enumerate}
The next table lists the possible choices for $I$, $f$, $\kappa$ and $\eta$. 
\begin{center}
\begin{tabular}{|c|c|c|c|c|}
\hline &
\textbf{$I$}                     & {$f$}    & $\eta$& \ $\kappa$  \\ \hline
(R1)&$[0,\pi]$                        & $\sin$        & $\hphantom{-}1$                 & $\hphantom{-}1$                \\ \hline
(R2)&$[0,\infty]$                     & $\mathrm{id}$ & $\hphantom{-}1$                 & $\hphantom{-}0$                 \\ \hline
          (R3)&       $\R$                             & $1$           & $\hphantom{-}0$                  & $\hphantom{-}0$                 \\ \hline
(R4)&$[0,\infty]$    & $\sinh$       & $\hphantom{-}1$                 & ${-}1$                 \\ \hline
(R5)&$\R$                             & $\exp$        & $\hphantom{-}0$                  & ${-}1$                 \\ \hline
(R6)&$\R$                             & $\cosh$       & ${-}1$                  & ${-}1$                 \\ \hline
\end{tabular}\end{center}

\subsection{Singularity and splitting theorems for generalized cones}
Next, we discuss singularity and splitting theorems in the setting of generalized cones.

\subsubsection{Hawking singularity theorem}
Assume $\I\times_f^N X$ satisfies a measure contraction property $\tmcp$.
We recall that a subset $\Sigma$  is achronal if it contains no two timelike related points. In a generalized cone the fibers $\{r_0\}\times X$ are achronal as well as future timelike complete (FTC) in the sense of Definition 1.7 in \cite{CM:24a}.  We fix $\{r_0\}\times X=: \Sigma$. The signed time separation distance function $\tau_\Sigma: X\rightarrow \R$ of $\Sigma$ is given by 
$$\tau_\Sigma((s,x))= s-r_0.$$
The intrinsic volume of the level sets $\{r\}\times X= \tau_\Sigma^{-1}(\{r\})$ is given by $f^N(r)\m_X$.

The time  function $\tau_\Sigma$ yields a disintegration of the reference $\n= f^N(r)\met r \otimes \m_X$ into $\n_x=~f^N(r) \mathcal H^1(r)$ such that 
$$\int_Q \n_x \met \mathfrak q(x)= \n,$$
for the quotient space $Q=X$ and the quotients  measure $\mathfrak q= \m_X$. Since the disintegration is induced by the product structure of the underlying space,  the conditional measure $\n_x$ actually does not depend on $x\in X$.

A synthetic notion of mean curvature for achronal FTC subsets of a Lorentzian measure space was given  in \cite[Def.\ 5.2]{CM:24a}. Here, we use a definition  that   was introduced in \cite{bkmr21, Ket:24} for Borel subsets in metric measure spaces that satisfy a curvature-dimension condition.  The hypersurface $\Sigma$ has  forward mean curvature bounded from below by $H\in\R$ if 
$$\limsup_{h\downarrow 0} \frac{1}{h} \left[\int_Y f(r_0+h)^N \met \m_X- \int_Y f(r_0)^N \met \m_X\right]\geq H \int_Y f(r_0)^N \met \m_X$$
for any bounded, Borel set $Y\subseteq X$. We remark that $\int_Y f(r)^N \met \m_X= f(r)^N \m_X(Y)$. In the general definition of lower mean curvature bounds the density of the conditional measures $\n_x$ may also depend on $x$. The inequality is equivalent to  
requiring that the derivative of $\log f(r)$ in $r_0$ is bounded from below by {$\frac{H}{N}$}. Moreover this  is equivalent to the definition in \cite{CM:24a}.

As a consequence of the synthetic Hawking singularity theorem in \cite[Thm.\ 5.6]{CM:24a}, we obtain the following corollary --- a singularity theorem for generalized cones satisfying a timelike measure contraction property. 
\begin{cor}\label{cor-hawking}
Let $\I\times_f^N X$ satisfy $\tmcp(KN,N+1)$ and assume that $\frac{d}{dr} \log f (r_0)\geq {\frac{H}{N}}$. If 
\begin{enumerate}
\item $K>0, N>0$ and $H\in \R$, or
\item $K=0, N>0$ and $H<0$, or 
\item $K<0, N>0$ and $H<- \sqrt{ -KN^2}<0$
\end{enumerate}
Then for every $x\in I^+(\Sigma)$ we have $\tau_{\Sigma}(x)< D_{K,H,N}<\infty$, where the precise constant can be found in \cite[Eq.\ (5.10)]{CM:24a}. In particular, for every timelike geodesic $\gamma:[0,L]\rightarrow I\times_f^N X$ (parametrized w.r.t.\ $\tau$-arclength) such that $\gamma_0\in \Sigma$ the domain $[0, L]$ is contained in $[0, D_{K,H,N}]$. 
\end{cor}
\subsubsection{Volume singularity theorem}
Based on  previous results by Treude and Grant \cite{treudegrant}  García-Heveling \cite{garcia:24} proposed the concept of volume singularities for (smooth) spacetimes. A generalization to the  setting of $\tCD$ spaces is: $X$ is future volume incomplete if it contains a point
whose chronological future has finite $\m_X$-measure. In this context Braun \cite{Bra:24} proved a volume singularity theorem that directly applies to our situation. 
\begin{cor}\label{cor-vol-sing}
Assume $\I\times_f^N X$ satisfies $\tmcp(KN,N+1)$ and is timelike essentially nonbranching. Suppose that $\m_X(X)<\infty$ and $\frac{d}{dr} \log f(r_0)\geq \frac{H}{N}$ for some $r_0\in \R$. 
Then, we get {$\n(I^+(\{r_0\}\times X))<\infty$}
provided,
\begin{enumerate}
\item If $K>0$, then $H\in \R$ is arbitrary.
\item If $K=0$, then $H<0$.
\item If $K<0$, then $H\leq -N \sqrt{-K}$.
\end{enumerate}
In particular, M is future volume incomplete.
\end{cor}

\subsubsection{Splitting theorem}
\begin{cor}\label{cor-spl-thm}
Let $f: I \rightarrow [0, \infty)$ be smooth. If $\I\times_f^N X$ satisfies the condition $\tCD(0,N+1)$, where $X$ is a proper, essentially nonbranching, complete, geodesic metric space with a Radon measure, and there exists a future directed timelike geodesic line, then $I= \R$ and $f\equiv c=const$. Especially, $\I\times_f^N X={}^-\R\times  X$.
\end{cor}
\begin{proof} A future directed timelike geodesic line is a timelike curve $\gamma: \R\rightarrow \I\times_f N$ such that $\gamma|_{[-L, L]}$ is a future directed timelike maximal geodesic for every $L>0$.  In particular, if we write $\gamma= (\alpha, \beta)$,  then $\alpha:\R\rightarrow I$ is a smooth curve by fiber independence.
\smallskip

We show below that  $I=\R$. Since the timelike curvature-dimension condition $\tCD(0,N+1)$ for $^- \R \times_f^N X$ implies that $f: \R\rightarrow (0, \infty)$ is concave (Theorem \ref{Th: Y to X}), it then follows that $f$ is constant.
\smallskip\\
{\it Claim:} $f\circ \alpha(t)>0$ for all $t\in \R$. 

If there exists $t_0\in \R$ such that $f\circ \alpha(t_0)=0$, then $\alpha(t_0)$ is an endpoint of $I$ by definition of our class of generalized cones (i.e., we have $\partial I  = f^{-1}(\{0\}))$. Since $\alpha'\geq 0$, it follows that $\alpha$ is constant equal to $\alpha(t_0)$ before (or after, respectively) time $t_0$ and hence $\gamma$ either starts or terminates in $t_0$ which is a contradiction.
\smallskip

Let $T>0$ and $\gamma|_{[-T, T]}=: \gamma_T=(\alpha_T, \beta_T)$. Then $\beta_L: [-T, T] $ is a  pre-geodesic in $X$. Let $2L_T=L(\beta_T)$. By fiber independence we can embed $I\times_f [-L_T, L_T]\subset I\times_f \R$ into $I\times_f X$ via a map $\Phi_T$. Moreover, if $\tilde \gamma=(\tilde \alpha, \tilde \beta):[-T, T] \rightarrow I\times_f \R$ is the geodesic between $(r_0, -L_T)$ and $(r_1, L_T)$, then $\alpha_T= \tilde \alpha$ and $\Phi_T\circ \tilde \beta= \beta_T$. $L_T$ is increasing as $T\rightarrow \infty$. We set $L =\sup_{T>0} L_T$. Since $T>0$ is arbitrary, we can construct an embedding $\Phi$  of $I\times_f [-L, L]$ into $I\times_f X$, and a future directed timelike geodesic line $\tilde \gamma: \R\rightarrow I\times_f \R$ such that $\Phi\circ \tilde \gamma= \gamma$. 

Since $f$ is smooth, $I\times_f\R$ is a smooth, time-oriented Lorentzian manifold that contains a timelike geodesic line. Moreover, it is globally hyperbolic as the fiber is complete by \cite[Thm.\ 3.66]{BEE:96}. Hence $\I\times_f \R\simeq ^-\R\times \R$, by the smooth Lorentzian splitting theorem for globally hyperbolic spacetimes \cite{Gal:89}. 
\end{proof}

\subsection{A new definition of curvature bounds}
Our main theorems apply for the special cases where $I= [0, \infty)$ and $f(r)=r$. These choices yield the Euclidean and the Minkowski cone respectively, depending on the signature of $I$. 
\begin{itemize}
\item  The Euclidean $N$-cone $C^N(X)= {}^+I \times_{\mathrm{id}}^N X$ satisfies $\mcp(0,N+1)$ if $X$ satisfies the condition $\CD(N-1,N)$ (Theorem \ref{th:mmsmcp}).
If $C^N(X)$ satisfies the condition $\CD(0,N+1)$ then $X$ satisfies the condition $\CD(N-1,N)$ (Remark \ref{rem:metricanalog}).  
\item  The Minkowski $N$-cone $M^N(X)= \I\times_{\mathrm{id}}^N X$ satisfies the timelike measure contraction property $\tmcp(0,N+1)$  if  $X$ satisfies the condition $\CD(-(N-1),N)$ (Theorem \ref{th:totmcp}).  If $M^N(X)$ satisfies the condition $\tCD_p(0, N+1)$ for some $p\in (0,1)$, then $X$ satisfies the condition $\CD(-(N-1),N)$ (Theorem \ref{Th: Y to X}).
\end{itemize} 
 As a special case of our general conjectures in Subsection \ref{subsec-new-ex}, the measure contraction properties above should improve to curvature-dimension conditions. 

In view of our results we suggest the following definition, which is another reason to study the Lorentzian spaces with timelike curvature bounds.
\begin{defi}\label{def-new-cb}
Let $X$ be a metric measure space. $X$ satisfies the conic curvature-dimension condition $\CD^{\mathrm{Con}}(K,N)$ if 
\begin{enumerate}
\item $K=0$: $X$ satisfies $\CD^{\mathrm{Con}}(0, N)$ if it satisfies the condition $\CD(0,N)$.
\item  $K=\pm (N-1)$:  if  $^\pm I \times_{\mathrm{id}}^N X$ satisfies 
$
\begin{cases} \CD(0,N+1) &  K>0\,,\\
\tCD_p(0, N+1) \mbox{ for  some $p\in (0,1)$} &  K<0\,.
\end{cases}
$
\item $K\in \R\backslash \{0\}$: if $\sqrt{\frac{N-1}{|K|}} X$ satisfies the condition $\CD^{\mathrm{Con}}\bigl(\mathrm{sgn}K (N-1), N\bigr)$. 
\end{enumerate}
\end{defi}
Corollary \ref{cor:cone} then tells us that the condition $\CD^{\mathrm{Con}}(K,N)$ with $K<0$ for a nonbranching metric measure space $X$ implies the condition $\CD(K,N)$.  This   yields geometric estimates like the sharp Bishop--Gromov comparison, or the sharp Brunn--Minkowski inequality for $X$. Moreover the verification of our Conjecture \ref{conj:cone} would imply the equivalence of the condition $\CD^{\mathrm{Con}}(K,N)$ with the curvature-dimension condition $\CD(K,N)$ in the sense of Lott--Sturm--Villani. 

This idea to define the curvature-dimension condition through coning also makes sense for Alexandrov lower curvature bounds.  
\begin{defi} A geodesic space $X$ satisfies  conic Alexandrov curvature bounded from by $K$, i.e., $CBB^{\mathrm{Con}}(K)$, if  
\begin{enumerate}
\item  $K=0$: $X$ satisfies $CBB^{\mathrm{Con}}(0)$ if it satisfies $CBB(0)$.
\item  $K=\pm 1$:  if  $^\pm I \times_{\mathrm{id}} X$ satisfies 
$
\begin{cases} CBB(0)& K>0\,,\\
TCBB(0)  & K<0\,.
\end{cases}
$
\item  $K\in \R\backslash \{0\}$: if $\frac{1}{\sqrt{|K|}} X$ satisfies $CBB^{\mathrm{Con}}(\sgn K ).$ 
\end{enumerate}
\end{defi}
By results in \cite{BBI:01} and in \cite{AGKS:23} we know that $CBB^{\mathrm{Con}}$ is equivalent to $CBB$ for any geodesic space $X$.

\section*{Acknowledgments}
The authors extend their sincere thanks to the referee  for their very thorough reading of our manuscript and for the detailed and constructive report. MC is grateful to Mathias Braun for helpful discussions. CK wants to thank  Chiara Rigoni and the University of Vienna, where parts of this article were written during a research visit. 

This research was funded in part by the Austrian Science Fund (FWF) [Grants DOI\\ \href{https://doi.org/10.55776/PAT1996423}{10.55776/PAT1996423}, \href{https://doi.org/10.55776/STA32}{10.55776/STA32} and \href{https://doi.org/10.55776/EFP6}{10.55776/EFP6}].

For open access purposes, the authors have applied a CC BY public copyright license to any author accepted manuscript version arising from this submission. 

\bibliographystyle{halpha-abbrv}
\bibliography{Master}

\newcommand{\etalchar}[1]{$^{#1}$}
\begin{thebibliography}{BGM{\etalchar{+}}25}
\expandafter\ifx\csname url\endcsname\relax
  \def\url#1{\texttt{#1}}\fi
\expandafter\ifx\csname doi\endcsname\relax
  \def\doi#1{\burlalt{doi:#1}{http://dx.doi.org/#1}}\fi
\expandafter\ifx\csname urlprefix\endcsname\relax\def\urlprefix{URL }\fi
\expandafter\ifx\csname href\endcsname\relax
  \def\href#1#2{#2}\fi
\expandafter\ifx\csname burlalt\endcsname\relax
  \def\burlalt#1#2{\href{#2}{#1}}\fi

\bibitem[AB98]{AB:98}
S.~B. Alexander and R.~L. Bishop.
\newblock Warped products of {H}adamard spaces.
\newblock {\em Manuscripta Math.}, 96(4):487--505, 1998.
\newblock \doi{10.1007/s002290050078}.

\bibitem[AB03]{AB:03}
S.~Alexander and R.~L. Bishop.
\newblock {$\mathcal FK$}-convex functions on metric spaces.
\newblock {\em Manuscripta Math.}, 110(1):115--133, 2003.
\newblock \doi{10.1007/s00229-002-0330-8}.

\bibitem[AB04]{AB:04}
S.~B. Alexander and R.~L. Bishop.
\newblock Curvature bounds for warped products of metric spaces.
\newblock {\em Geom. Funct. Anal.}, 14(6):1143--1181, 2004.
\newblock \doi{10.1007/s00039-004-0487-2}.

\bibitem[AB08]{AB:08}
S.~B. Alexander and R.~L. Bishop.
\newblock Lorentz and semi-{R}iemannian spaces with {A}lexandrov curvature
  bounds.
\newblock {\em Comm. Anal. Geom.}, 16(2):251--282, 2008.
\newblock \doi{10.4310/CAG.2008.v16.n2.a1}.

\bibitem[AB16]{AB:16}
S.~B. Alexander and R.~L. Bishop.
\newblock Warped products admitting a curvature bound.
\newblock {\em Adv. Math.}, 303:88--122, 2016.
\newblock \doi{10.1016/j.aim.2016.07.005}.

\bibitem[ABS22]{ABS:22}
L.~{Ak{\'{e}} Hau}, S.~Burgos, and D.~A. Solis.
\newblock Causal completions as {L}orentzian pre-length spaces.
\newblock {\em Gen. Relativity Gravitation}, 54(9):Paper No. 108, 20, 2022.
\newblock \doi{10.1007/s10714-022-02980-x}.

\bibitem[ACM{\etalchar{+}}21]{ACMcCCF:21}
A.~Akdemir, A.~Colinet, R.~McCann, F.~Cavalletti, and F.~Santarcangelo.
\newblock Independence of synthetic curvature dimension conditions on transport
  distance exponent.
\newblock {\em Trans. Amer. Math. Soc.}, 374(8):5877--5923, 2021.
\newblock \doi{10.1090/tran/8413}.

\bibitem[AG13]{AG:13}
L.~Ambrosio and N.~Gigli.
\newblock A user's guide to optimal transport.
\newblock In {\em Modelling and optimisation of flows on networks}, volume 2062
  of {\em Lecture Notes in Math.}, pages 1--155. Springer, Heidelberg, 2013.
\newblock \doi{10.1007/978-3-642-32160-3\_1}.

\bibitem[AGKS23]{AGKS:23}
S.~B. Alexander, M.~Graf, M.~Kunzinger, and C.~S\"{a}mann.
\newblock Generalized cones as {L}orentzian length spaces: causality,
  curvature, and singularity theorems.
\newblock {\em Comm. Anal. Geom.}, 31(6):1469--1528, 2023.
\newblock \doi{10.4310/cag.2023.v31.n6.a5}.

\bibitem[AGS05]{AGS:05}
L.~Ambrosio, N.~Gigli, and G.~Savar\'e.
\newblock {\em Gradient flows in metric spaces and in the space of probability
  measures}.
\newblock Lectures in Mathematics ETH Z\"urich. Birkh\"auser Verlag, Basel,
  2005.

\bibitem[AH98]{AH:98}
L.~Andersson and R.~Howard.
\newblock Comparison and rigidity theorems in semi-{R}iemannian geometry.
\newblock {\em Comm. Anal. Geom.}, 6(4):819--877, 1998.
\newblock \doi{10.4310/CAG.1998.v6.n4.a8}.

\bibitem[BBC{\etalchar{+}}24]{BBCGMORS:24}
T.~Beran, M.~Braun, M.~Calisti, N.~Gigli, R.~J. McCann, A.~Ohanyan, F.~Rott,
  and C.~Sämann.
\newblock A nonlinear d'{A}lembert comparison theorem and causal differential
  calculus on metric measure spacetimes, 2024,
  \burlalt{2408.15968}{http://arxiv.org/abs/2408.15968}.
\newblock \urlprefix\url{https://arxiv.org/abs/2408.15968}.

\bibitem[BBI01]{BBI:01}
D.~Burago, Y.~Burago, and S.~Ivanov.
\newblock {\em A course in metric geometry}, volume~33 of {\em Graduate Studies
  in Mathematics}.
\newblock American Mathematical Society, Providence, RI, 2001.
\newblock \doi{10.1090/gsm/033}.

\bibitem[BEE96]{BEE:96}
J.~K. Beem, P.~E. Ehrlich, and K.~L. Easley.
\newblock {\em Global {L}orentzian geometry}, volume 202 of {\em Monographs and
  Textbooks in Pure and Applied Mathematics}.
\newblock Marcel Dekker Inc., New York, second edition, 1996.

\bibitem[BFH23]{BFH:23}
S.~Burgos, J.~L. Flores, and J.~Herrera.
\newblock The {$c$}-completion of {L}orentzian metric spaces.
\newblock {\em Classical Quantum Gravity}, 40(20):Paper No. 205013, 25, 2023.
\newblock \doi{10.1088/1361-6382/acf7a5}.

\bibitem[BGM{\etalchar{+}}24]{BGMcCOS:24}
M.~Braun, N.~Gigli, R.~J. McCann, A.~Ohanyan, and C.~Sämann.
\newblock An elliptic proof of the splitting theorems from {L}orentzian
  geometry.
\newblock 2024.
\newblock \doi{10.48550/arXiv.2410.12632}.

\bibitem[BGM{\etalchar{+}}25]{BGMcCOS:25}
M.~Braun, N.~Gigli, R.~J. McCann, A.~Ohanyan, and C.~S{\"a}mann.
\newblock Lorentzian splitting theorems for continuously differentiable metrics
  and weights.
\newblock 2025.
\newblock In preparation.

\bibitem[BKMR21]{bkmr21}
J.~Bertrand, C.~Ketterer, I.~Mondello, and T.~Richard.
\newblock Stratified spaces and synthetic {Ricci} curvature bounds.
\newblock {\em Ann. Inst. Fourier}, 71(1):123--173, 2021.
\newblock \doi{10.5802/aif.3393}.

\bibitem[BM23]{BMcC:23}
M.~Braun and R.~J. McCann.
\newblock Causal convergence conditions through variable timelike {R}icci
  curvature bounds.
\newblock 2023.
\newblock \doi{10.48550/arXiv.2312.17158}.

\bibitem[BMS24]{BMS:24}
A.~Bykov, E.~Minguzzi, and S.~Suhr.
\newblock Lorentzian metric spaces and gh-convergence: the unbounded case.
\newblock 2024, \burlalt{2412.04311}{http://arxiv.org/abs/2412.04311}.
\newblock \urlprefix\url{https://arxiv.org/abs/2412.04311}.

\bibitem[Bog07]{Bog:07b}
V.~I. Bogachev.
\newblock {\em Measure theory. {V}ol. {II}}.
\newblock Springer-Verlag, Berlin, 2007.
\newblock \doi{10.1007/978-3-540-34514-5}.

\bibitem[BR24]{BR:24}
T.~Beran and F.~Rott.
\newblock Gluing constructions for {L}orentzian length spaces.
\newblock {\em Manuscripta Math.}, 173(1-2):667--710, 2024.
\newblock \doi{10.1007/s00229-023-01469-4}.

\bibitem[Bra23a]{Bra:23a}
M.~Braun.
\newblock Good geodesics satisfying the timelike curvature-dimension condition.
\newblock {\em Nonlinear Anal.}, 229:Paper No. 113205, 30, 2023.
\newblock \doi{10.1016/j.na.2022.113205}.

\bibitem[Bra23b]{Bra:23b}
M.~Braun.
\newblock R\'{e}nyi's entropy on {L}orentzian spaces. {T}imelike
  curvature-dimension conditions.
\newblock {\em J. Math. Pures Appl. (9)}, 177:46--128, 2023.
\newblock \doi{10.1016/j.matpur.2023.06.009}.

\bibitem[Bra24]{Bra:24}
M.~Braun.
\newblock Exact d'{A}lembertian for {L}orentz distance functions.
\newblock 2024.
\newblock \doi{10.48550/arXiv.2408.16525}.

\bibitem[Bra25]{Bra:25}
M.~Braun.
\newblock New perspectives on the d'{A}lembertian from general relativity. {A}n
  invitation.
\newblock 2025.
\newblock \doi{10.48550/arXiv.2501.19071}.

\bibitem[BS10]{BS:10}
K.~Bacher and K.-T. Sturm.
\newblock Localization and tensorization properties of the curvature-dimension
  condition for metric measure spaces.
\newblock {\em J. Funct. Anal.}, 259(1):28--56, 2010.
\newblock \doi{10.1016/j.jfa.2010.03.024}.

\bibitem[Bus67]{Bus:67}
H.~Busemann.
\newblock Timelike spaces.
\newblock {\em Dissertationes Math. Rozprawy Mat.}, 53:52, 1967.
\newblock ISSN: 0012-3862.

\bibitem[Cav17]{Cav:17}
F.~Cavalletti.
\newblock An overview of {$L^1$} optimal transportation on metric measure
  spaces.
\newblock In {\em Measure theory in non-smooth spaces}, Partial Differ. Equ.
  Meas. Theory, pages 98--144. De Gruyter Open, Warsaw, 2017.

\bibitem[CC96]{CC:96}
J.~Cheeger and T.~H. Colding.
\newblock Lower bounds on {R}icci curvature and the almost rigidity of warped
  products.
\newblock {\em Ann. of Math. (2)}, 144(1):189--237, 1996.
\newblock \doi{10.2307/2118589}.

\bibitem[Che99]{chenwp}
C.-H. Chen.
\newblock Warped products of metric spaces of curvature bounded from above.
\newblock {\em Trans. Amer. Math. Soc.}, 351(12):4727--4740, 1999.
\newblock \doi{10.1090/S0002-9947-99-02154-6}.

\bibitem[CM17a]{CM:17a}
F.~Cavalletti and A.~Mondino.
\newblock Optimal maps in essentially non-branching spaces.
\newblock {\em Commun. Contemp. Math.}, 19(6):1750007, 27, 2017.
\newblock \doi{10.1142/S0219199717500079}.

\bibitem[CM17b]{CM:17b}
F.~Cavalletti and A.~Mondino.
\newblock Sharp and rigid isoperimetric inequalities in metric-measure spaces
  with lower {R}icci curvature bounds.
\newblock {\em Invent. Math.}, 208(3):803--849, 2017.
\newblock \doi{10.1007/s00222-016-0700-6}.

\bibitem[CM20]{CM:20}
F.~Cavalletti and A.~Mondino.
\newblock New formulas for the {L}aplacian of distance functions and
  applications.
\newblock {\em Anal. PDE}, 13(7):2091--2147, 2020.
\newblock \doi{10.2140/apde.2020.13.2091}.

\bibitem[CM21]{CM:21}
F.~Cavalletti and E.~Milman.
\newblock The globalization theorem for the curvature-dimension condition.
\newblock {\em Invent. Math.}, 226(1):1--137, 2021.
\newblock \doi{10.1007/s00222-021-01040-6}.

\bibitem[CM22]{CM:22}
F.~Cavalletti and A.~Mondino.
\newblock A review of {L}orentzian synthetic theory of timelike {R}icci
  curvature bounds.
\newblock {\em Gen. Relativity Gravitation}, 54(11):Paper No. 137, 39, 2022.
\newblock \doi{10.1007/s10714-022-03004-4}.

\bibitem[CM24a]{CM:24a}
F.~Cavalletti and A.~Mondino.
\newblock Optimal transport in {L}orentzian synthetic spaces, synthetic
  timelike {R}icci curvature lower bounds and applications.
\newblock {\em Camb. J. Math.}, 12(2):417--534, 2024.
\newblock \doi{10.4310/cjm.2024.v12.n2.a3}.

\bibitem[CM24b]{CM:24b}
F.~Cavalletti and A.~Mondino.
\newblock A sharp isoperimetric-type inequality for {L}orentzian spaces
  satisfying timelike {R}icci lower bounds.
\newblock 2024.
\newblock \doi{10.48550/arXiv.2401.03949}.

\bibitem[CMM24]{CMM:24}
F.~Cavalletti, D.~Manini, and A.~Mondino.
\newblock Optimal transport on null hypersurfaces and the null energy
  condition.
\newblock 2024.
\newblock \doi{10.48550/arXiv.2408.08986}.

\bibitem[Col96]{coldingshape}
T.~H. Colding.
\newblock Shape of manifolds with positive {R}icci curvature.
\newblock {\em Invent. Math.}, 124(1-3):175--191, 1996.
\newblock \doi{10.1007/s002220050049}.

\bibitem[EKS15]{EKS:15}
M.~Erbar, K.~Kuwada, and K.-T. Sturm.
\newblock On the equivalence of the entropic curvature-dimension condition and
  {B}ochner's inequality on metric measure spaces.
\newblock {\em Invent. Math.}, 201(3):993--1071, 2015.
\newblock \doi{10.1007/s00222-014-0563-7}.

\bibitem[EM17]{EM:17}
M.~Eckstein and T.~Miller.
\newblock Causality for nonlocal phenomena.
\newblock {\em Ann. Henri Poincar\'{e}}, 18(9):3049--3096, 2017.
\newblock \doi{10.1007/s00023-017-0566-1}.

\bibitem[Fre03]{fremlin}
D.~H. Fremlin.
\newblock {\em Measure theory. {Vol}. 2. {Broad} foundations}.
\newblock Colchester: Torres Fremlin, corrected second printing of the 2001
  original edition, 2003.

\bibitem[Gal89]{Gal:89}
G.~J. Galloway.
\newblock The {L}orentzian splitting theorem without the completeness
  assumption.
\newblock {\em J. Differential Geom.}, 29(2):373--387, 1989.
\newblock
  \urlprefix\url{http://projecteuclid.org.myaccess.library.utoronto.ca/euclid.jdg/1214442881}.

\bibitem[GH24]{garcia:24}
L.~Garc\'ia-Heveling.
\newblock Volume singularities in general relativity.
\newblock {\em Lett. Math. Phys.}, 114(3):Paper No. 71, 20, 2024.
\newblock \doi{10.1007/s11005-024-01814-y}.

\bibitem[Gig13]{gigli_splitting_context}
N.~Gigli.
\newblock The splitting theorem in non-smooth context.
\newblock 2013.
\newblock \urlprefix\url{https://arxiv.org/abs/1302.5555}.

\bibitem[GKS19]{GKS:19}
J.~D.~E. Grant, M.~Kunzinger, and C.~S\"{a}mann.
\newblock Inextendibility of spacetimes and {L}orentzian length spaces.
\newblock {\em Ann. Global Anal. Geom.}, 55(1):133--147, 2019.
\newblock \doi{10.1007/s10455-018-9637-x}.

\bibitem[GMS15]{gsm15}
N.~Gigli, A.~Mondino, and G.~Savar{\'e}.
\newblock Convergence of pointed non-compact metric measure spaces and
  stability of {Ricci} curvature bounds and heat flows.
\newblock {\em Proc. Lond. Math. Soc. (3)}, 111(5):1071--1129, 2015.
\newblock \doi{10.1112/plms/pdv047}.

\bibitem[Han20]{Han:20}
B.-X. Han.
\newblock Measure rigidity of synthetic lower {R}icci curvature bound on
  {R}iemannian manifolds.
\newblock {\em Adv. Math.}, 373:107327, 31, 2020.
\newblock \doi{10.1016/j.aim.2020.107327}.

\bibitem[Har82]{Har:82}
S.~G. Harris.
\newblock A triangle comparison theorem for {L}orentz manifolds.
\newblock {\em Indiana Univ. Math. J.}, 31(3):289--308, 1982.
\newblock \doi{10.1512/iumj.1982.31.31026}.

\bibitem[Kel17]{Kel:17}
M.~Kell.
\newblock On interpolation and curvature via {W}asserstein geodesics.
\newblock {\em Adv. Calc. Var.}, 10(2):125--167, 2017.
\newblock \doi{10.1515/acv-2014-0040}.

\bibitem[Ket13]{Ket:13}
C.~Ketterer.
\newblock Ricci curvature bounds for warped products.
\newblock {\em J. Funct. Anal.}, 265(2):266--299, 2013.
\newblock \doi{10.1016/j.jfa.2013.05.008}.

\bibitem[Ket17]{Ket:17}
C.~Ketterer.
\newblock On the geometry of metric measure spaces with variable curvature
  bounds.
\newblock {\em J. Geom. Anal.}, 27(3):1951--1994, 2017.
\newblock \doi{10.1007/s12220-016-9747-2}.

\bibitem[Ket24]{Ket:24}
C.~Ketterer.
\newblock Characterization of the null energy condition via displacement
  convexity of entropy.
\newblock {\em J. Lond. Math. Soc. (2)}, 109(1):Paper No. e12846, 24, 2024.
\newblock \doi{10.1112/jlms.12846}.

\bibitem[Ket25a]{Ket:25}
C.~Ketterer.
\newblock Warped products and synthetic lower curvature bounds: an overview.
\newblock 2025.
\newblock \doi{10.48550/arXiv.2503.05521}.

\bibitem[Ket25b]{ketterer:25}
C.~Ketterer.
\newblock Warped products over one-dimensional base spaces and the {RCD}
  condition.
\newblock 2025.
\newblock \doi{10.48550/arXiv.2506.10809}.

\bibitem[KP67]{KP:67}
E.~H. Kronheimer and R.~Penrose.
\newblock On the structure of causal spaces.
\newblock {\em Proc. Cambridge Philos. Soc.}, 63:481--501, 1967.
\newblock \doi{10.1017/s030500410004144x}.

\bibitem[KS01]{kush01}
K.~Kuwae and T.~Shioya.
\newblock On generalized measure contraction property and energy functionals
  over {Lipschitz} maps.
\newblock {\em Potential Anal.}, 15(1-2):105--121, 2001.
\newblock \doi{10.1023/A:1011218425271}.

\bibitem[KS18]{KS:18}
M.~Kunzinger and C.~S\"amann.
\newblock Lorentzian length spaces.
\newblock {\em Ann.\ Glob.\ Anal.\ Geom.}, 54(3):399--447, 2018.
\newblock \doi{10.1007/s10455-018-9633-1}.

\bibitem[Li24]{zhenhao24}
Z.~Li.
\newblock The globalization theorem for {{\(\mathrm{CD}(K, N)\)}} on locally
  finite spaces.
\newblock {\em Ann. Mat. Pura Appl. (4)}, 203(1):49--70, 2024.
\newblock \doi{10.1007/s10231-023-01352-9}.

\bibitem[LS23]{LySt:23}
A.~Lytchak and S.~Stadler.
\newblock Ricci curvature in dimension 2.
\newblock {\em J. Eur. Math. Soc. (JEMS)}, 25(3):845--867, 2023.
\newblock \doi{10.4171/JEMS/1196}.

\bibitem[LV09]{LV:09}
J.~Lott and C.~Villani.
\newblock Ricci curvature for metric-measure spaces via optimal transport.
\newblock {\em Ann. of Math. (2)}, 169(3):903--991, 2009.
\newblock \doi{10.4007/annals.2009.169.903}.

\bibitem[McC20]{McC:20}
R.~McCann.
\newblock Displacement concavity of {B}oltzmann's entropy characterizes
  positive energy in general relativity.
\newblock {\em Camb. J. Math.}, 8(3):609--681, 2020.
\newblock \doi{10.4310/CJM.2020.v8.n3.a4}.

\bibitem[McC24]{McC:24}
R.~J. McCann.
\newblock A synthetic null energy condition.
\newblock {\em Comm. Math. Phys.}, 405(2):Paper No. 38, 24, 2024.
\newblock \doi{10.1007/s00220-023-04908-1}.

\bibitem[McC25]{McC:25}
R.~J. McCann.
\newblock Trading linearity for ellipticity: a nonsmooth approach to
  {E}instein's theory of gravity and the {L}orentzian splitting theorems.
\newblock 2025.
\newblock \doi{10.48550/arXiv.2501.00702}.

\bibitem[Min23]{Min:23}
E.~Minguzzi.
\newblock Further observations on the definition of global hyperbolicity under
  low regularity.
\newblock {\em Classical Quantum Gravity}, 40(18):Paper No. 185001, 9, 2023.
\newblock \doi{10.1088/1361-6382/acdd40}.

\bibitem[MR24]{mr24}
A.~Mondino and V.~Ryborz.
\newblock On the equivalence of distributional and synthetic {Ricci} curvature
  lower bounds.
\newblock Preprint, {arXiv}:2402.06486 [math.{DG}] (2024), 2024.
\newblock \urlprefix\url{https://arxiv.org/abs/2402.06486}.

\bibitem[MS23]{MS:23}
A.~Mondino and S.~Suhr.
\newblock An optimal transport formulation of the {E}instein equations of
  general relativity.
\newblock {\em J. Eur. Math. Soc. (JEMS)}, 25(3):933--994, 2023.
\newblock \doi{10.4171/jems/1188}.

\bibitem[MS24]{MS:24}
E.~Minguzzi and S.~Suhr.
\newblock Lorentzian metric spaces and their {G}romov-{H}ausdorff convergence.
\newblock {\em Lett. Math. Phys.}, 114(3):Paper No. 73, 63, 2024.
\newblock \doi{10.1007/s11005-024-01813-z}.

\bibitem[MS25]{MS:25}
A.~Mondino and C.~Sämann.
\newblock Lorentzian {G}romov-{H}ausdorff convergence and pre-compactness.
\newblock 2025.
\newblock \doi{10.48550/arXiv.2504.10380}.

\bibitem[M{\"u}l22]{Mue:22}
O.~M{\"u}ller.
\newblock Gromov-{H}ausdorff distances for {L}orentzian length spaces.
\newblock {\em preprint, arXiv:2209.12736 [math.DG]}, 2022.
\newblock \doi{10.48550/arXiv.2209.12736}.

\bibitem[Oht07a]{Oht:07}
S.-i. Ohta.
\newblock On the measure contraction property of metric measure spaces.
\newblock {\em Comment. Math. Helv.}, 82(4):805--828, 2007.
\newblock \doi{10.4171/CMH/110}.

\bibitem[Oht07b]{ohta:cones}
S.-i. Ohta.
\newblock Products, cones, and suspensions of spaces with the measure
  contraction property.
\newblock {\em J. Lond. Math. Soc., II. Ser.}, 76(1):225--236, 2007.
\newblock \doi{10.1112/jlms/jdm057}.

\bibitem[Oht09]{ohta_interpolation}
S.-i. Ohta.
\newblock Finsler interpolation inequalities.
\newblock {\em Calc. Var. Partial Differ. Equ.}, 36(2):211--249, 2009.
\newblock \doi{10.1007/s00526-009-0227-4}.

\bibitem[O'N83]{ONe:83}
B.~O'Neill.
\newblock {\em Semi-{R}iemannian geometry with applications to relativity},
  volume 103 of {\em Pure and Applied Mathematics}.
\newblock Academic Press, Inc. [Harcourt Brace Jovanovich, Publishers], New
  York, 1983.

\bibitem[Pet11]{pet11}
A.~Petrunin.
\newblock Alexandrov meets {Lott}-{Villani}-{Sturm}.
\newblock {\em M{\"u}nster J. Math.}, 4(1):53--64, 2011.

\bibitem[Raj12]{Raj:12}
T.~Rajala.
\newblock Interpolated measures with bounded density in metric spaces
  satisfying the curvature-dimension conditions of {S}turm.
\newblock {\em J. Funct. Anal.}, 263(4):896--924, 2012.
\newblock \doi{10.1016/j.jfa.2012.05.006}.

\bibitem[Rot23]{Rot:23}
F.~Rott.
\newblock Gluing of {L}orentzian length spaces and the causal ladder.
\newblock {\em Classical Quantum Gravity}, 40(17):Paper No. 175002, 28, 2023.
\newblock \doi{10.1088/1361-6382/ace585}.

\bibitem[S{\"a}m24]{Sae:24}
C.~S{\"a}mann.
\newblock A brief introduction to non-regular spacetime geometry.
\newblock {\em Internationale Mathematische Nachrichten}, 256, 2024.
\newblock \doi{10.48550/arXiv.2404.18651}.

\bibitem[Sou25]{soultanis}
E.~Soultanis.
\newblock Generalized products and {L}orentzian length spaces.
\newblock {\em Lett. Math. Phys.}, 115(1):Paper No. 21, 38, 2025.
\newblock \doi{10.1007/s11005-025-01910-7}.

\bibitem[SS24]{SS:24}
A.~Sakovich and C.~Sormani.
\newblock Introducing various notions of distances between space-times.
\newblock 2024.
\newblock \doi{10.48550/arXiv.2410.16800}.

\bibitem[Stu06a]{Stu:06a}
K.-T. Sturm.
\newblock On the geometry of metric measure spaces. {I}.
\newblock {\em Acta Math.}, 196(1):65--131, 2006.
\newblock \doi{10.1007/s11511-006-0002-8}.

\bibitem[Stu06b]{Stu:06b}
K.-T. Sturm.
\newblock On the geometry of metric measure spaces. {II}.
\newblock {\em Acta Math.}, 196(1):133--177, 2006.
\newblock \doi{10.1007/s11511-006-0003-7}.

\bibitem[SV16]{SV:16}
C.~Sormani and C.~Vega.
\newblock Null distance on a spacetime.
\newblock {\em Classical Quantum Gravity}, 33(8):085001, 29, 2016.
\newblock \doi{10.1088/0264-9381/33/7/085001}.

\bibitem[TG13]{treudegrant}
J.-H. Treude and J.~D.~E. Grant.
\newblock Volume comparison for hypersurfaces in {L}orentzian manifolds and
  singularity theorems.
\newblock {\em Ann. Global Anal. Geom.}, 43(3):233--251, 2013.
\newblock \doi{10.1007/s10455-012-9343-z}.

\bibitem[Vil09]{Vil:09}
C.~Villani.
\newblock {\em Optimal transport. Old and new}, volume 338 of {\em Grundlehren
  der Mathematischen Wissenschaften [Fundamental Principles of Mathematical
  Sciences]}.
\newblock Springer-Verlag, Berlin, 2009.
\newblock \doi{10.1007/978-3-540-71050-9}.

\bibitem[Wyl17]{wylie:warped}
W.~Wylie.
\newblock A warped product version of the {C}heeger-{G}romoll splitting
  theorem.
\newblock {\em Trans. Amer. Math. Soc.}, 369(9):6661--6681, 2017.
\newblock \doi{10.1090/tran/7003}.

\end{thebibliography}
\addcontentsline{toc}{section}{References}

\end{document}